\documentclass[12pt]{amsart}

\usepackage{amsfonts,amsthm,amsmath,amssymb,latexsym}
\usepackage{graphicx}
\usepackage{eucal}
\usepackage{labelfig}
\usepackage[all]{xy}
\usepackage{hyperref}

\title[A Thurston boundary for Teichm\"uller spaces]{A Thurston boundary for infinite-dimensional Teichm\"uller spaces}

\author{Francis Bonahon} 
\address{Department of Mathematics, University of Southern California, 3620 S. Vermont Avenue, Los Angeles CA 90089-2532}
\email{fbonahon@usc.edu}
 
\author{Dragomir \v Sari\' c}
\address{Department of Mathematics, Queens College of the City University of New York, 65--30 Kissena Blvd., Flushing, NY 11367}
\address{Mathematics PhD. Program, Graduate Center of the City University of New York, 365 Fifth Avenue, New York, NY 10016-4309}
\email{Dragomir.Saric@qc.cuny.edu}

\thanks{This research was partially supported by the grants DMS-1102440, DMS-1406559 and DMS-1711297 from the US National Science Foundation, and by the collaboration grant  346391 from the Simons Foundation.}

\newtheorem{thm}{Theorem}
\newtheorem{prop}[thm]{Proposition}
\newtheorem{cor}[thm]{Corollary}
\newtheorem{lem}[thm]{Lemma}

\theoremstyle{definition}

\newtheorem{rem}[thm]{Remark}

\theoremstyle{plain}

\newcommand{\D}{\mathbb D}
\newcommand{\C}{\mathbb C}
\newcommand{\N}{\mathbb N}
\newcommand{\R}{\mathbb R}
\newcommand{\HH}{\mathbb H}
\newcommand{\CC}{\mathcal  C_{\mathrm{bd}}}
\newcommand{\T}{\mathcal  T}
\newcommand{\LL}{\mathbf  L}
\newcommand{\PL}{\mathbf  {PL}}
\newcommand{\HHH}{\mathbf H}

\newcommand{\I}{\mathrm i}
\newcommand{\E}{\mathrm e}
\newcommand{\ML}{\mathcal  {ML}_{\mathrm{bd}}}
\newcommand{\PML}{\mathcal  {PML}_{\mathrm{bd}}}
\newcommand{\PCC}{\mathcal  {PC}_{\mathrm{bd}}}  
\newcommand{\dT}{d_{\mathrm T}}
\newcommand{\MCG}{\mathbf  {MCG}_{\mathrm{qc}}}

\renewcommand{\leq}{\leqslant}
\renewcommand{\geq}{\geqslant}
\renewcommand{\epsilon}{\varepsilon}
\renewcommand{\phi}{\varphi}


\date{\today}

\begin{document}
\maketitle
\begin{abstract}
For a compact surface $X_0$, Thurston introduced a compactification of its Teich\-m\"uller space $\T(X_0)$ by completing it with a boundary $\mathcal{PML}(X_0)$ consisting of projective measured geodesic laminations. We introduce a similar bordification for the Teichm\"uller space $\T(X_0)$ of a noncompact Riemann surface $X_0$, using the technical tool of geodesic currents. The lack of compactness requires the introduction of certain uniformity conditions which were unnecessary for compact surfaces. A technical step, providing a convergence result for earthquake paths in $\T(X_0)$, may be of independent interest. 
\end{abstract}

The Teichm\"uller space of a Riemann surface $X_0$ is the space of quasiconformal deformations of the complex structure of $X_0$. When $X_0$ is compact of genus at least 2, W.P.~Thurston famously introduced a compactification of $\T(X_0)$ by adding a boundary at infinity consisting of projective measured foliations \cite{Thu2, FLP, FLP2} or, equivalently, projective measured geodesic laminations \cite{Thu0, Bo2}. In this paper, we introduce a similar construction of a boundary for the Teichm\"uller space of a noncompact surface $X_0$. In addition to the fact that Teichm\"uller spaces of noncompact Riemann surfaces are fundamental objects in complex analysis, our motivation here is to put in evidence the hidden features that underlie Thurston's construction, by tying it more closely to the quasiconformal geometry of $X_0$ and less to the purely topological considerations that suffice for compact surfaces. 

Like Thurston, we restrict attention to Riemann surfaces $X_0$ that are conformally hyperbolic, in the sense that the conformal structure of $X_0$ can be realized by a complete hyperbolic metric. This is equivalent to the property that the universal cover $\widetilde X_0$ is biholomorphically equivalent to the unit disk $\D\subset \C$. This condition only excludes the cases where $X_0$ is an elliptic surface, diffeomorphic to the torus, or is the Riemann sphere minus 0, 1 or 2 points. A case of particular interest is that of the disk $\D$, in which case the Teichm\"uller space $\T(\D)$ is Bers's \emph{Universal Teichm\"uller Space} \cite{Be}. 

Thurston's original length spectrum approach \cite{Thu2, FLP} is not available here, and we follow the strategy introduced in \cite{Bo2} by embedding the Teichm\"uller space $\T(X_0)$ in the space $\mathcal C(X_0)$ of geodesic currents. These are defined as those measures on the space $G(\widetilde X_0)$ of Poincar\'e geodesics of the universal cover $\widetilde X_0$ which are invariant under the action of the fundamental group $\pi_1(X_0)$. When $X_0$ is compact, these are purely topological objects, which were introduced in \cite{Bo1} as a completion of the set of free homotopy classes of closed curves on the surface; in fact, geodesic currents can be described  \cite{Bo3} solely in terms of the algebraic structure of  $\pi_1(X_0)$. The definition of geodesic currents  was motivated by Thurston's definition of measured foliations and measured geodesic laminations, introduced  as a way to complete the set of isotopy classes of simple closed curves on the surface \cite{Thu0, Thu1, FLP, FLP2}. The topological nature of geodesic currents and measured geodesic laminations becomes much weaker for noncompact surfaces, and this requires the consideration of uniformity conditions which were taken for granted in the compact case. 

More precisely, if $X_0$ is a conformally hyperbolic Riemann surface and if its universal cover $\widetilde X_0$ is endowed with the Poincar\'e metric, the space $G(\widetilde X_0)$ of complete geodesics of $\widetilde X_0$ comes with a preferred measure, the \emph{Liouville measure} $L_{\widetilde X_0}$. If we have a quasiconformal deformation of the complex structure of $X_0$, represented by a quasiconformal diffeomorphism $f\colon X_0 \to X$ from $X_0$ to another Riemann surface $X$, we can then use $f$ to pull back the Liouville measure $L_{\widetilde X}$ of $G(\widetilde X)$ to a $\pi_1(X_0)$--invariant measure on $G(\widetilde X_0)$, namely to a geodesic current in $X_0$. 

This enables us to define what we call the \emph{Liouville embedding}
$$
\LL \colon \T(X_0) \to \mathcal C(X_0)
$$
of the Teichm\"uller space, which associates the Liouville current $L_f$ to each element $[f]\in \T(X_0)$ represented by a quasiconformal diffeomorphism $f \colon X_0 \to X$. 

There is nothing new so far. But a challenge arises when the surface $X_0$ is noncompact: Find a ``good'' topology on the space $\mathcal C(X_0)$ of geodesic currents for which the Liouville embedding $\LL$ is really a topological embedding, namely restricts to a homeomorphism $ \T(X_0) \to \LL\big( \mathcal T(X_0) \big)$. The natural topology on $\T(X_0)$ is the Teichm\"uller topology, defined by the Teichm\"uller metric; see \S \ref{sect:Teichmueller}. As a space of measures, $\mathcal C(X_0)$ is traditionally endowed with the weak* topology (see \S \ref{sect:GeodesicCurr}). However, this topology fails to take into account the many symmetries of the universal cover $\widetilde X_0$ coming from the  group $\HHH(\widetilde X_0)\cong \mathrm{PSL}_2(\R)$ of all biholomorphic maps of $\widetilde X_0$. 

This leads us to restrict attention to \emph{bounded} geodesic currents, which satisfy a certain boundedness property with respect to the action of $\HHH(\widetilde X_0)$, and to introduce the \emph{uniform weak* topology} on the space $\CC(X_0)$ of bounded geodesic currents. See \S \ref{sect:GeodesicCurr} for precise definitions. When the surface $X_0$ is compact, every geodesic current is bounded and the uniform weak* topology coincides with the usual weak* topology on $\mathcal C(X_0) = \CC(X_0)$ (Proposition~\ref{prop:UniformNonUniformWhenCompact}). See \cite{Sar4, Sar5, Otal, MiSa} for earlier (and slightly different) incarnations of the uniform weak* topology. 

\begin{thm}
\label{thm:LiouvilleEmbeddingIntro}
 The Liouville embedding $\LL \colon \T(X_0) \to \mathcal C(X_0)$ is valued in the space $\CC(X_0)$ of bounded geodesic currents, and restricts to a homeomorphism $\T(X_0) \to \LL \big( \T(X_0) \big) \subset \CC(X_0)$ when $\CC(X_0)$ is endowed with the uniform weak* topology. In addition, the image $ \LL \big( \T(X_0) \big) $ is closed in $ \CC(X_0)$, and the embedding $\LL \colon \T(X_0) \to \CC(X_0)$ is proper. 
\end{thm}

This theorem is proved  as Theorem~\ref{thm:LiouvilleEmbedding}. Recall that a map is \emph{proper} if the preimage of a bounded subset is bounded, which makes sense here because the topologies of $\T(X_0)$ and $\CC(X_0)$ are defined by families of seminorms. 

See Remark~\ref{rem:UniformWeak*Needed} for an explanation of why Theorem~\ref{thm:LiouvilleEmbeddingIntro} would fail if $\CC(X_0)$ was only endowed with the usual weak* topology, as opposed to the uniform weak* topology. 

Following Thurston's original approach, we now consider the rays $\R^+ \alpha \subset \CC(X_0)$ that are asymptotic to the image $ \LL \big( \T(X_0) \big) $, namely the  set of those bounded geodesic currents $\alpha \in \CC(X_0)$ for which there exists a sequence $\big\{ [f_n]\big\}_{n\in\N}$ of points of the Teichm\"uller space and a sequence of positive numbers $\{t_n\}_{n\in\N}$ such that $\alpha = \lim_{n\to\infty} t_n \LL\big([f_n] \big)$ and $\lim_{n\to\infty} t_n=0$. The union of these rays is the \emph{asymptotic cone} of the Liouville embedding $\LL$. 

\begin{thm}
\label{thm:TeichmAsymptoticConeIntro}
 The asymptotic cone of the Liouville embedding $\LL \colon \T(X_0) \to \CC(X_0)$ coincides with the subset $\ML(X_0)$ of bounded measured geodesic laminations in $X_0$, namely with the set of bounded geodesic currents $\alpha \in \CC(X_0)$ such that no two geodesics of the support of $\alpha$ in $G(\widetilde X_0)$ cross each other in $\widetilde X_0$. 
\end{thm}

It is not too hard to see that every element of the asymptotic cone of $\LL$ is a bounded measured geodesic lamination. It is more difficult to show that every bounded measured geodesic lamination belongs to this cone. For this, we use Thurston's construction of earthquakes \cite{Ke, Thu1}. A bounded measured geodesic lamination $\alpha \in \ML(X_0)$ defines an \emph{earthquake map} $E^\alpha \colon \T(X_0) \to \T(X_0)$ \cite{Thu1, EpM, Sar1}. See Remark~\ref{rem:BoundedLamAndQuasiconfEarthqke} for comments about the close relationship, when the surface $X_0$ is noncompact, between the boundedness condition for measured geodesic laminations and the quasiconformal geometry of points of the Teichm\"uller space $\T(X_0)$. 

The following property proves that every bounded measured geodesic lamination belongs to the asymptotic cone of the Liouville embedding. It is also of independent interest as, when the surface $X_0$ is noncompact, the estimates of \cite{Ke} or \cite[Exp.~8]{FLP} cannot be used here. 

\begin{thm}
\label{thm:LimitEarthquakesIntro}
Let $\alpha \in \ML(X_0)$ be a bounded measured geodesic lamination in the Riemann surface $X_0$. Then, for every $[f] \in \T(X_0)$,
$$
\lim_{t\to \infty} \frac1t \LL \big( E^{t\alpha} [f] \big) = \alpha
$$
for the uniform weak* topology on the space $\CC(X_0)$ of bounded geodesic currents. 
\end{thm}

The space of rays in the asymptotic cone is the space $\PML(X_0)$ of projective bounded measured geodesic laminations. 
Theorem~\ref{thm:TeichmAsymptoticConeIntro} enables us to add its elements  as  boundary points to the Teichm\"uller space. By analogy with the case of compact surfaces, we call the space $\T(X_0) \cup \PML(X_0)$ the \emph{Thurston bordification} of the Teichm\"uller space $\T(X_0)$. Note that this bordification is not compact when $X_0$ is noncompact, as $\T(X_0)$ is not even locally compact in this case. See \cite{HS2, HS1, HS3, MiSa2} for related results. 

The article concludes with a result, Proposition~\ref{prop:QuasiconformalExtendsBoundary}, which shows that our construction is natural under quasiconformal diffeomorphisms. More precisely,  the homeomorphism $\T(X_1) \to \T(X_2)$ induced by a quasiconformal diffeomorphism $X_1 \to X_2$ uniquely extends to a homeomorphism $\T(X_1) \cup \PML(X_1) \to \T(X_2) \cup \PML(X_2)$ between the respective bordifications of the Teichm\"uller spaces $\T(X_1)$ and $\T(X_2)$. In particular, the quasiconformal mapping class group $\MCG(X_0)$ acts on $\T(X_0) \cup \PML(X_0)$.

\medskip
This article started as a preprint \cite{Sar3} by the second author alone.  The first author, who had been informally involved in the introduction of the uniform weak* topology, later joined to help with the exposition. However, the major technical steps were already fully in \cite{Sar3}. See also \cite{Sar6} for a different approach, in a much more restricted context. 

The authors thank the referee for a careful reading of the manuscript.

\section{The Teichm\"uller space  of a Riemann surface}
\label{sect:Teichmueller}

Let $X_0$ be a Riemann surface which is \emph{conformally hyperbolic}. This means that its universal cover $\widetilde X_0$ is biholomorphically equivalent to the unit disk
$$\D = \{ z\in \C; |z|<1\}.$$
Equivalently, $X_0$ is not the Riemann sphere $\C\cup \{ \infty\}$, the plane $\C$, the punctured plane $\C-\{0\}$, or a torus.

In the unit disk $\D$, the hyperbolic metric $2|dz|/(1-|z|^2)$ is invariant under the group $\HHH(\D)$ of biholomorphic maps of $\D$. It consequently descends to a hyperbolic metric on $X_0$ which does not depend on the biholomorphic identification $\widetilde X_0 \cong \D$. This is the \emph{Poincar\'e metric} of the conformally hyperbolic Riemann surface $X_0$. 

All Riemann surfaces in this article will be implicitly assumed to be conformally hyperbolic. We are particularly interested in the case where $X_0$ is non-compact, and a fundamental example will be that of the unit disk  $X_0=\D$. 

Recall that a \emph{quasiconformal diffeomorphism} $f \colon X_1 \to X_2$ between two Riemann surfaces is an orientation-preserving diffeomorphism such that
$$
K(f) = \sup_{z\in X_1}\, \frac{ \left|\frac{\partial f}{\partial z}(z)\right| +  \left|\frac{\partial f}{\partial \bar z}(z)\right|}{\left|\frac{\partial f}{\partial z}(z)\right| -  \left|\frac{\partial f}{\partial \bar z}(z)\right|} 
$$
is finite. Note that the denominator is always positive by the orientation-preserving hypothesis. The number $K(f)$ is the \emph{quasiconformal dilatation} of $f$.

The \emph{Teichm\"uller space} $\T(X_0)$ of the Riemann surface $X_0$ is the space of equivalence classes of all quasiconformal diffeomorphisms $f \colon X_0\to X$ from $X_0$ to another Riemann surface $X$. Two such quasiconformal maps $f_1\colon X_0\to X_1$ and $f_2\colon X_0\to X_2$ are {\it equivalent} if there exists a biholomorphic map $g \colon X_1\to X_2$ such that $f_2^{-1}\circ g\circ f_1$ is isotopic to the identity by a \emph{bounded  isotopy}, namely by an isotopy  that moves  points of $X_0$ by a bounded amount for the Poincar\'e metric of $X_0$. See \cite{EM} for equivalent formulations of this equivalence relation. 
 We denote by $[f]\in \T(X_0)$ the equivalence class of the quasiconformal map $f \colon X_0 \to X$.

In the fundamental case where $X_0$ is the unit disk $\D$, the Teichm\"uller space $\T(\D)$ is also known as the \emph{universal Teichm\"uller space} \cite{Be, GH}.

The Teichm\"uller space $\T(X_0)$ is endowed with the \emph{Teichm\"uller distance} defined by
$$
\dT \bigl( [f_1], [f_2] \bigr) ={\textstyle \frac12}  \log \inf_g K(g)
$$
where the infimum is taken over all quasiconformal maps $g\colon X_1 \to X_2$ such that $f_2^{-1}\circ g\circ f_1$ is \emph{bounded isotopic} to the identity of $X_0$ as above, namely isotopic to the identity by an isotopy moving points by a uniformly bounded amount for the Poincar\'e metric of $X_0$. Again, see \cite{EM} for equivalent formulations.

\section{Bounded geodesic currents and the uniform weak* topology}
\label{sect:GeodesicCurr}

\subsection{Geodesic currents}
\label{subsect:GeodesicCurr}
 We consider a conformally hyperbolic Riemann surface $X_0$ of hyperbolic type,  with universal cover $\widetilde X_0$.
 
Recall that the group $\HHH(\D)$ of biholomorphic maps of the unit disk $\D$  consists of all linear fractional maps of the form 
$$
z \mapsto \frac{\alpha z + \beta}{\bar\beta z + \bar\alpha}
$$
where $\alpha$, $\beta \in \C$ are such that $|\alpha|^2 - |\beta|^2=1$. In particular, these biholomorphic maps of the open disk $\D$ extend to homeomorphisms of the closed disk $\D \cup \partial \D$. 

This enables us to introduce a compactification of the universal cover $\widetilde X_0$ by its \emph{circle at infinity} $\partial_\infty \widetilde X_0$, intrinsically defined by the property that every biholomorphic map $\widetilde X_0 \to \D$ extends to a homeomorphism $\widetilde X_0 \cup \partial_\infty \widetilde X_0 \to \D \cup \partial \D$.

Each complete hyperbolic geodesic of the disk $\D$ is determined by its two endpoints in $\partial \D$. This identifies the space $G(\D)$ of (complete, oriented) geodesics of $\D$ to $\partial \D \times \partial \D - \Delta$, where 
 $\Delta= \bigl\{(x,x); x\in\partial \D\bigl\}$ is the diagonal of  $\partial \D \times \partial \D $. 
 
 More generally,  let $G(\widetilde X_0)$ denote the space of oriented complete geodesics of $\widetilde X_0$ for its Poincar\'e metric. Using a biholomorphic identification $\widetilde X_0 \cong \D$, such a geodesic is determined by its endpoints in the circle at infinity $\partial_\infty \widetilde X_0$, and this gives  a natural identification
$$
G(\widetilde X_0) = \partial_\infty \widetilde X_0 \times \partial_\infty \widetilde X_0 - \Delta
$$
where $\Delta= \bigl\{(x,x); x\in\partial_\infty \widetilde X_0\bigl\}$ is the diagonal of $\partial_\infty \widetilde X_0 \times \partial_\infty \widetilde X_0$. In particular, $G(\widetilde X_0)$ is homeomorphic to an open annulus. 

The fundamental group $\pi_1(X_0)$ acts biholomorphically on the universal cover $\widetilde X_0$, and this action therefore respects the Poincar\'e metric of $\widetilde X_0$. As a consequence, $\pi_1(X_0)$ also acts on $G(\widetilde X_0)$.

A \emph{geodesic current} in the Riemann surface $X_0$ is a Radon measure $\alpha$ on $G(\widetilde X_0)$ that is invariant under the action of $\pi_1(X_0)$.  The Radon property means that the integral $\alpha(K) = \int_K1\,d\alpha$ is finite and non-negative for every compact subset $K \subset G(\widetilde X_0)$. 

Most of the geodesic currents considered in this article will be \emph{balanced} (or \emph{unoriented} to use a more topological terminology), in the sense that they are invariant under the involution of $G(\widetilde X_0)$ that reverses the orientation of every geodesic. 

\subsection{Bounded geodesic currents and the uniform weak* topology}

As a space of Radon measures on $G(\widetilde X_0)$, it would be natural to endow the space $\mathcal C(X_0)$ of  geodesic currents with the classical \emph{weak* topology} (also called the \emph{vague topology}), defined by the family of semi-norms 
$$
\vert \alpha \vert_\xi = \Bigl\vert \int_{G(\widetilde X_0)} \xi \, d\alpha \Bigr\vert
$$
for $\alpha \in \mathcal C(X_0)$, as $\xi$ ranges over all continuous function $\xi \colon G(\widetilde X_0) \to \R$ with compact support. 

However, this topology does not quite fit our purposes, because it does not take into account the many symmetries of $\widetilde X_0$ provided by the isometric action of the group $\HHH(\widetilde X_0)$ of biholomorphic maps of $\widetilde X_0$. It is much better to consider the semi-norms 
$$
\Vert \alpha \Vert_\xi = 
\sup_{\phi \in \HHH(\widetilde X_0)} \Bigl\vert \int_{G(\widetilde X_0)} \xi\circ \phi \  d\alpha \Bigr\vert
$$
as $\xi$ ranges over all continuous function $\xi \colon G(\widetilde X_0) \to \R$ with compact support. (We are here using the same letter to denote the biholomorphic map $\phi \colon \widetilde X_0 \to \widetilde X_0$, which respects the Poincar\'e metric of $\widetilde X_0$,  and its induced homeomorphism $\phi \colon G(\widetilde X_0) \to G(\widetilde X_0)$ on the space $G(\widetilde X_0)$ of geodesics of $\widetilde X_0$.) We will restrict the geodesic currents considered accordingly. 

A \emph{bounded geodesic current} is a  geodesic current $\alpha \in \mathcal C(X_0)$ for which all norms $\Vert \alpha \Vert_\xi$ are finite. More precisely, a bounded geodesic current on the Riemann surface $X_0$ is a Radon measure $\alpha$ on the space $G(\widetilde X_0)= \partial_\infty \widetilde X_0 \times \partial_\infty \widetilde X_0 - \Delta$ of geodesics of $\widetilde X_0$ such that:
\begin{enumerate}
 \item for every continuous function $\xi \colon G(\widetilde X_0) \to \R$ with compact support, the integrals $
 \Bigl\vert \int_{G(\widetilde X_0)} \xi\circ \phi \  d\alpha \Bigr\vert
$ are bounded independently of the biholomorphic map $\phi \in \HHH(\widetilde X_0)$;
\item $\alpha$ is invariant under the action of the fundamental group $\pi_1(X_0)$ on $G(\widetilde X_0)$. 
\end{enumerate}

We let $\CC( X_0)$ denote the set of bounded geodesic currents in the Riemann surface $X_0$. The topology defined by the seminorms $\Vert \alpha \Vert_\xi$ is the  \emph{uniform weak* topology} of $\CC(X_0)$.

In particular, a sequence $\{\alpha_n \}_{n \in \N}$  of bounded geodesic currents $\alpha_n \in \CC(X_0)$ converges to $\alpha$ for the uniform weak* topology if and only if 
$$
\sup_{\phi \in \HHH(\widetilde X_0)} 
\Bigl\vert\int_{G(\widetilde X_0)} \xi\circ \phi \ d\alpha_n- \int_{G(\widetilde X_0)} \xi\circ \phi \ d\alpha
\Bigr\vert  \to 0 \text{ as } n\to \infty
$$
for every  continuous function $\xi \colon G(\widetilde X_0) \to \R$ with compact support. 

\subsection{The weak* and uniform weak* topologies}

We collect in this section a few basic properties of the weak* and uniform weak* topologies.

The following easy lemma will enable us to make some of our arguments a little more intuitive, by interpreting continuity properties in terms of sequences. 

\begin{lem}
\label{lem:GeodCurrentSpaceMetrizable}
The weak* and uniform  weak* topology of $\CC( X_0)$ are metrizable. 
\end{lem}

This property is of course classical for the weak* topology, and we just need to make sure that the argument extends to the uniform weak* topology.

\begin{proof}
Write $G(\widetilde X_0)$ as an increasing union $G(\widetilde X_0) = \bigcup_{n=1}^\infty K_n$ of compact subsets $K_n$, with $K_n \subset K_{n+1}$. Then, for every $n$,  choose a countable family $\mathcal F_n$ of continuous functions $\xi \colon G(\widetilde X_0) \to \R$ with support contained in $K_n$, such that  the set $\mathcal F_n$ is dense  in the space of all continuous functions with support in $K_n$ for the metric 
$$d(\xi, \xi') = \max _{g \in G(\widetilde X_0)} |\xi(g) - \xi'(g)|.$$
For each $n$, also choose  a nonnegative continuous function  $\xi^{(n)} \colon G(\widetilde X_0) \to \left[ 0,\infty\right[$ with compact support such that $\xi^{(n)}(g) \geq 1$ for every $g\in K_n$.  Finally, set 
 $$\mathcal F = \bigcup_{n=1}^\infty \mathcal F_n \cup \{\xi^{(n)} \}. $$
  We want to show that, as $\xi$ ranges over all elements of the countable set $\mathcal F$, the topology defined by the corresponding family of semi-norms $ \Vert \ \Vert_\xi$ coincides with the uniform weak* topology (defined by considering all continuous functions $\xi \colon G(\widetilde X_0) \to \R$ with compact support). 
 
 The uniform weak* topology is defined by the basis consisting of all ``balls''
 $$
 \mathcal B_{\xi_1, \xi_2, \dots, \xi_k} (\alpha; r) = \big\{ \beta\in \CC( X_0); \Vert \alpha-\beta \Vert_{\xi_i} < r \text{ for all } i=1, 2, \dots, k \big\}
 $$
 where $\alpha\in \CC(X_0)$,  the functions $\xi_i \colon G(\widetilde X_0) \to \R$ with $i=1$, $2$, \dots, $k$ are continuous with compact support, and $r>0$. 
 
 For such a ball  $ \mathcal B_\xi (\alpha; r) $ associated to a single function $\xi$, the support of $\xi$ is contained in one of the compact subsets $K_n$. For an $\epsilon>0$ to be specified later, there is by definition of $\mathcal F_n$ a function $\xi' \in \mathcal F_n$ such that $d(\xi, \xi')<\epsilon$. As a consequence, remembering that $\xi^{(n)}$ is nonnegative and at least 1 on $K_n$, we have that $|\xi(g) - \xi'(g)| \leq \epsilon \xi^{(n)}(g)$ for every $g \in G(\widetilde X_0)$, and therefore
 $$
 \biggl| \int_{G(\widetilde X_0)} \xi\circ \phi \, d\alpha -  \int_{G(\widetilde X_0)}  \xi'\circ \phi \, d\alpha \biggr| \leq  \epsilon \int_{G(\widetilde X_0)}  \xi^{(n)}\circ \phi \, d\alpha 
 $$
 and
$$
 \biggl| \int_{G(\widetilde X_0)} \xi\circ \phi \, d\beta -  \int_{G(\widetilde X_0)}  \xi'\circ \phi \, d\beta \biggr| \leq  \epsilon \int_{G(\widetilde X_0)}  \xi^{(n)}\circ \phi \, d\beta 
 $$
 for every $\beta\in \CC( X_0)$ and every $\phi \in \HHH(\widetilde X_0)$.  This implies that
 $$
 \Vert \alpha - \beta \Vert_\xi \leq  \Vert \alpha - \beta \Vert_{\xi'} + \epsilon \Vert \alpha\Vert_{\xi^{(n)}} + \epsilon \Vert \beta\Vert_{\xi^{(n)}}. 
 $$
 
 If we choose $\epsilon>0$ small enough that $\epsilon \Vert\alpha\Vert_{\xi^{(n)}}<\frac r3$, this enables us to find two functions $\xi'$ and $\xi^{(n)}\in \mathcal F$ such that 
 $$
  \mathcal B_{\xi'} (\alpha; {\textstyle\frac r3}) \cap  \mathcal B_{\xi^{(n)}} (\alpha;  {\textstyle \frac r{3\epsilon}}) \subset  \mathcal B_\xi (\alpha; r) .
 $$
 
 By taking multiple intersections, it follows that for every ball 
 $$ \mathcal B_{\xi_1, \xi_2, \dots, \xi_k} (\alpha; r) =  \mathcal B_{\xi_1} (\alpha; r)  \cap  \mathcal B_{\xi_2} (\alpha; r)  \cap \dots \cap  \mathcal B_{\xi_k} (\alpha; r) 
 $$
 there exists $\xi_1'$, $\xi_2'$, \dots, $\xi_{k'}'\in \mathcal F$ and $r'>0$ such that
 $$
 \mathcal B_{\xi_1', \xi_2', \dots, \xi_{k'}'} (\alpha; r') \subset \mathcal B_{\xi_1, \xi_2, \dots, \xi_k} (\alpha; r).
 $$
This shows that the basis consisting of the $ \mathcal B_{\xi_1', \xi_2', \dots, \xi_{k'}'} (\alpha; r')$ with all $\xi' \in \mathcal F$ defines the same topology as the similar basis where all functions with compact support are considered. In other words, the uniform weak* topology $\CC(X_0)$ is also the topology defined by the family of seminorms $ \Vert \ \Vert_\xi$ with  $\xi \in \mathcal F$.
 
 Since $\mathcal F$ is countable, it follows that this topology is metrizable. More precisely, if we list the elements of $\mathcal F$ as $\{ \xi_i; i =1, 2, \dots \}$, the uniform weak* topology is the metric topology associated to the metric $\delta$ defined by 
\begin{equation*}
 \delta(\alpha, \beta) = \sum_{i=1}^\infty 2^{-i} \min\{1, \Vert\alpha-\beta \Vert_{\xi_i} \}.
\end{equation*}

The proof that the weak* topology is metrizable is almost identical (and classical). 
\end{proof}

\begin{prop}
\label{prop:UniformNonUniformWhenCompact}
 If the Riemann surface $X_0$ is compact, the space $\CC( X_0)$ of bounded geodesic currents coincide with the space $ \mathcal C(X_0)$ of all geodesic currents, and  the uniform weak* topology coincides with the weak* topology on $\CC(X_0)$. 
\end{prop}

The two topologies do differ when $X_0$ is noncompact. For instance, if $g_n \in G(\D)$ is a sequence of geodesics of $\D$ that eventually leaves any compact subset of $G(\D)$, the Dirac measures $\delta_{g_n}\in \CC(\D)$ based at $g_n$ provide a sequence of bounded geodesic currents in $\CC(\D)$ that converges to $0$ for the weak* topology but has no limit for the uniform weak* topology. Also, the sum $\sum_{n=1}^\infty n \delta_{g_n}$ is a well-defined geodesic current, which is unbounded. 

\begin{proof}
 [Proof of Proposition~{\upshape\ref{prop:UniformNonUniformWhenCompact}}] We first show that every geodesic current $\alpha \in \mathcal C(X_0)$ is bounded. 
 
 We want to prove that, for every continuous function $\xi \colon G(\widetilde X_0) \to \R$ with compact support, the semi-norm 
 \begin{equation}
 \label{eqn:UniformWeak*=Weak*0}
 \Vert \alpha  \Vert_\xi = \sup_{\phi \in \HHH(\widetilde X_0)}
 \bigg|  \int_{G(\widetilde X_0)} \xi\circ \phi \, d\alpha  \bigg| 
\end{equation}
is finite. Because $X_0$ is compact, there exists a compact subset $K \subset \widetilde X_0$ whose image under the action of $\pi_1(X_0)$ covers all of $\widetilde X_0$, in the sense that $\widetilde X_0 = \bigcup_{\gamma \in \pi_1(X_0)} \gamma(K)$. Pick a base point $x_0 \in K$. Then, for every biholomorphic map $\phi \in \HHH(\widetilde X_0)$, there exists at least one $\gamma \in \pi_1(X_0)$ such that $\phi \circ \gamma(x_0) \in K$. Note that $\phi \circ \gamma$ is also biholomorphic, and that 
 $$
   \int_{G(\widetilde X_0)} \xi\circ \phi\circ \gamma \, d\alpha =  \int_{G(\widetilde X_0)} \xi\circ \phi \, d\alpha
 $$
 by invariance of the measure $\alpha$ under the action of $\pi_1(X_0)$. Therefore, in the supremum of (\ref{eqn:UniformWeak*=Weak*0}), we can restrict attention to those $\phi \in \HHH(\widetilde X_0)$ such that $\phi(x_0) \in K$. Such $\phi$ form a compact subset of $ \HHH(\widetilde X_0)\cong \mathrm{PSL}_2(\R)$, and the supremum is therefore finite. This proves that $ \Vert \alpha  \Vert_\xi<\infty$. 

As a conclusion,  every geodesic current  $\alpha \in \mathcal C(X_0)$ is bounded, and therefore  $\mathcal C( X_0) $ coincides with $ \CC( X_0)$. 

We now  prove that the weak* and uniform weak* topologies coincide on  $\mathcal C( X_0) = \CC( X_0)$.
 By Lemma~\ref{lem:GeodCurrentSpaceMetrizable}, these topologies are metrizable.  Therefore we only need to show that, when $X_0$ is compact, a sequence $\{\alpha_n\}_{n\in \N}$ converges to $\alpha$ for the uniform weak* topology if and only if it converges to $\alpha$ for the weak* topology. 
 
 Convergence for the uniform weak* topology clearly implies convergence for the weak* topology. So we can focus on the converse statement.
 
 Suppose that $\alpha_n \in \CC(X_0)$ converges to $\alpha$ for the weak* topology. We want to show that, for every continuous function $\xi \colon G(\widetilde X_0) \to \R$ with compact support, 
 \begin{equation}
 \label{eqn:UniformWeak*=Weak*1}
 \Vert \alpha_n - \alpha \Vert_\xi = \sup_{\phi \in \HHH(\widetilde X_0)}
 \bigg|  \int_{G(\widetilde X_0)} \xi\circ \phi \, d\alpha_n
 -  \int_{G(\widetilde X_0)} \xi\circ \phi \, d\alpha  \bigg| 
\end{equation}
 tends to 0 as $n$ tends to $\infty$. 
 
 As before, the compactness of $X_0$ enables us to  restrict attention to those $\phi \in \HHH(\widetilde X_0)$ such that $\phi(x_0) \in K$, which form a compact subset of $\HHH(\widetilde X_0)$ (remember that $\HHH(\widetilde X_0)$ is also the set of isometries of the Poincar\'e metric of $\widetilde X_0$). In particular, the  supremum of (\ref{eqn:UniformWeak*=Weak*1}) is  attained at some $\phi_n \in \HHH(\widetilde X_0)$, with $\phi_n(x_0)\in K$ and
 $$
  \Vert \alpha_n - \alpha \Vert_\xi = 
 \bigg|  \int_{G(\widetilde X_0)} \xi\circ \phi_n \, d\alpha_n
 -  \int_{G(\widetilde X_0)} \xi\circ \phi _n\, d\alpha  \bigg| .
 $$
 
 In addition, again by compactness of the set of those  $\phi \in \HHH(\widetilde X_0)$ with $\phi(x_0) \in K$, we can extract a subsequence $\{\phi_{n_k}\}_{k\in\N}$ that converges to some $\phi_\infty \in \HHH(\widetilde X_0)$ uniformly on compact subsets of $\widetilde X_0$. In particular, 
 \begin{equation}
 \label{eqn:UniformWeak*=Weak*2}
 \begin{split}
  \Vert \alpha_{n_k} - \alpha \Vert_\xi 
  &=  \bigg|  \int_{G(\widetilde X_0)} \xi\circ \phi_{n_k} \, d\alpha_{n_k}
 -  \int_{G(\widetilde X_0)} \xi\circ \phi _{n_k}\, d\alpha  \bigg| 
  \\
  &\leq 
 \bigg|  \int_{G(\widetilde X_0)} \xi\circ \phi_\infty \, d\alpha_{n_k}
 -  \int_{G(\widetilde X_0)} \xi\circ \phi _\infty\, d\alpha  \bigg|
 \\
 &\qquad\qquad
 +  \int_{G(\widetilde X_0)} | \xi\circ \phi_{n_k} - \xi\circ \phi_\infty| \, d\alpha_{n_k}
 +  \int_{G(\widetilde X_0)} | \xi\circ \phi_{n_k} - \xi\circ \phi_\infty| \, d\alpha
\end{split}
\end{equation}
 
It is now time to use  the fact that $\alpha = \lim_{n\to \infty} \alpha_n$ for the weak* topology, which implies that 
 \begin{equation}
 \label{eqn:UniformWeak*=Weak*4}
\lim_{k\to \infty}
 \bigg|  \int_{G(\widetilde X_0)} \xi\circ \phi_\infty \, d\alpha_{n_k}
 -  \int_{G(\widetilde X_0)} \xi\circ \phi _\infty\, d\alpha  \bigg|=0. 
\end{equation}

Also, pick a  nonnegative continuous function $\xi_\infty \colon G(\widetilde X_0) \to \R$ with compact support, such that $\xi_\infty\geq 1$ on a neighborhood of the support of $\xi \circ \phi_\infty$. Given $\epsilon>0$, 
$$
| \xi\circ \phi_{n_k} - \xi\circ \phi_\infty| \leq \epsilon \xi_\infty
$$ 
for $k$ large enough, since $\phi_{n_k} \to \phi_\infty$ as $k\to \infty$ uniformly on compact subsets of $\widetilde X_0$ (and therefore uniformly on compact subsets of $G(\widetilde X_0)$, if we use the same letter to denote the action of $\phi_{n_k}$ on $\widetilde X_0$ and on $G(\widetilde X_0)$). It follows that 
$$
 \int_{G(\widetilde X_0)} | \xi\circ \phi_{n_k} - \xi\circ \phi_\infty| \, d\alpha_{n_k} \leq \epsilon  \int_{G(\widetilde X_0)} \xi_\infty  \, d\alpha_{n_k} . 
$$
Since $\int_{G(\widetilde X_0)} \xi_\infty  \, d\alpha_{n_k} \to \int_{G(\widetilde X_0)} \xi_\infty  \, d\alpha_\infty$ as $k\to \infty$ by weak* convergence, we conclude that
 \begin{equation}
 \label{eqn:UniformWeak*=Weak*5}
\lim_{k\to \infty} \int_{G(\widetilde X_0)} | \xi\circ \phi_{n_k} - \xi\circ \phi_\infty| \, d\alpha_{n_k} =0.
\end{equation}

Similarly,
 \begin{equation}
 \label{eqn:UniformWeak*=Weak*6}
\lim_{k\to \infty} \int_{G(\widetilde X_0)} | \xi\circ \phi_{n_k} - \xi\circ \phi_\infty| \, d\alpha =0.
\end{equation}

The combination of the equations (\ref{eqn:UniformWeak*=Weak*2}--\ref{eqn:UniformWeak*=Weak*6}) proves that 
$$
\lim_{k\to \infty}  \Vert \alpha_{n_k} - \alpha \Vert_\xi=0.
$$

Therefore, we were able to extract from the sequence $\{\alpha_n\}_{n\in\N}$ a subsequence $\{\alpha_{n_k}\}_{k\in\N}$ that converges to $\alpha$ for the uniform weak* topology. If we apply the same process to all subsequences of the original sequence $\{\alpha_n\}_{n\in\N}$, we conclude that this sequence $\{\alpha_n\}_{n\in\N}$ converges to $\alpha$ for the uniform weak* topology. 

This completes the proof of Proposition~\ref{prop:UniformNonUniformWhenCompact}.
\end{proof}

Because we will frequently use it, we state as a lemma a well-known property of the weak* topology.

\begin{lem}
\label{lem:Weak*ConvImpliesConvMasses}
Suppose that the sequence $\{\alpha_n\}_{n\in\N}$ of geodesic currents $\alpha_n \in \mathcal C(X_0)$ converges to $\alpha \in \mathcal C(X_0)$ for the weak* topology. Then, for every every measurable subset $A\subset G(\widetilde X_0)$ whose topological boundary $\delta A$ has $\alpha$--mass $\alpha(\delta A)$ equal to $0$, 
$$\lim_{n\to \infty} \alpha_n(A) = \alpha(A).$$
\end{lem}
\begin{proof}
See for instance \cite[chap. IV, \S 5, n${}^{\mathrm o}$ 12]{Bou} for this classical property of weak* convergence, which holds in a much more general setting.
\end{proof}
The example of Dirac measures shows that the hypothesis that $\alpha(\delta A)=0$ is really necessary in Lemma~\ref{lem:Weak*ConvImpliesConvMasses}. 

\section{The Liouville embedding}

\subsection{The Liouville geodesic current} We saw that the group $\HHH(\D)$ of biholomorphic maps of $\D$  acts by isometries for the Poincar\'e metric, and therefore acts on the space $G(\D)$ of complete geodesics of $\D$. A computation shows that it  respects the \emph{Liouville measure} $L_\D$ on $G(\D)$ defined by the property that, if we parametrize the unit circle $\partial \D \subset \C$ by $t\mapsto  \E^{ \I t}$,
$$
L_\D(A) = \int_A \frac{dt\,ds}{|\E^{ \I t}- \E^{\I s}|^2}
$$
for any Borel subset $A \subset G(\D)  = \partial\D \times \partial\D - \Delta$.  See for instance Lemma~\ref{lem:LiouvilleAndCrossratios} below, and the well-known invariance of cross-ratios under linear fractional maps. 

More generally, if $\widetilde X$ is a Riemann surface biholomorphically equivalent to $\D$ by a biholomorphic map $\widetilde f\colon \widetilde X \to \D$, the induced homeomorphism $ \partial_\infty \widetilde X \to \partial \D$ provides a homeomorphism from the space $G(\widetilde X) = \partial_\infty \widetilde X \times \partial_\infty \widetilde X  - \Delta$ of geodesics of $\widetilde X$ to $G(\D) = \partial\D \times \partial\D - \Delta$, which we also denote by $\widetilde f$. We can then pull back the Liouville measure $L_\D$ to a measure $L_{\widetilde X}$ on $G(\widetilde X) $. The invariance of $L_\D$ under the group $\HHH(\D)$ of biholomorphic maps of $\D$ shows that this measure is independent of the choice of the biholomorphic map  $\widetilde f\colon \widetilde X \to \D$. The measure $L_{\widetilde X}$ is the \emph{Liouville measure} of the Riemann surface $\widetilde X \cong \D$. 

Consider an element $[f] \in \T(X_0)$ of the Teichm\"uller space of the Riemann surface $X_0$, represented by a quasiconformal diffeomorphism  $f\colon X_0 \to X$. Lift $f$ to a quasiconformal diffeomorphism  $\widetilde f \colon \widetilde X_0 \to \widetilde X$ between the universal covers. A fundamental property is that this quasiconformal diffeomorphism  admits a continuous extension $\widetilde f \colon \widetilde X_0 \cup \partial_\infty \widetilde X_0 \to \widetilde X \cup \partial_\infty \widetilde X$ (see the Beurling-Ahlfors Theorem~\ref{thm:Beurling-Ahlfors} below). The  restriction of this extension to the circles at infinity induces a homeomorphism from $G(\widetilde X_0) = \partial_\infty \widetilde X_0 \times \partial_\infty \widetilde X_0  - \Delta$ to $G(\widetilde X) = \partial_\infty \widetilde X \times \partial_\infty \widetilde X  - \Delta$. We can then pull back  the Liouville measure $L_{\widetilde X}$ by $\widetilde f$ to define a measure $L_{f}$ on $G(\widetilde X_0)$. More precisely, $L_{f}(A) = L_{\widetilde X}\bigl( \widetilde f(A) \bigr)$ for every measurable subset $A \subset G(\widetilde X_0)$, while
$$
\int_{G(\widetilde X_0)} \xi \, dL_{f} = \int_{G(\widetilde X)} \xi \circ \widetilde f^{-1} \, dL_{\widetilde X}
$$
for every continuous function $\xi\colon G(\widetilde X_0) \to R$ with compact support. 

The action of the fundamental group $\pi_1(X)$ on $\widetilde X$ is biholomorphic, and therefore respects the Liouville measure $L_{\widetilde X}$ on $G(\widetilde X)$. Since two  lifts $\widetilde f \colon \widetilde X_0 \to \widetilde X$ of $f$ differ only by the action of an element of $\pi_1( X)$, it follows that the measure $L_f$ is independent of the choice of this lift. Also, because $\widetilde f$ conjugates the action of $\pi_1( X)$ on $\widetilde X$ to the action of $\pi_1(X_0)$ on $\widetilde X_0$, the measure $L_f$ is invariant under the action of $\pi_1(X_0)$ on $G(\widetilde X_0)$. In other words, $L_f$ is a  geodesic current in $X_0$. 

\begin{lem}
 \label{lem:LiouvilleBounded}
 The Liouville geodesic current $L_f$ is bounded, and therefore belongs to $ \CC(X_0)$. 
\end{lem}

We postpone the proof of Lemma~\ref{lem:LiouvilleBounded} to \S \ref{sect:QuasiconfQuasisym}, where it will be proved as Lemma~\ref{lem:LiouvilleBounded2}. 

If two quasiconformal diffeomorphisms $f_1 \colon X_0 \to X_1$ and $f_2 \colon X_0 \to X_2$ represent the same element $[f_1]=[f_2]$ in the Teichm\"uller space $\T(X_0)$, there exists a biholomorphic map $g \colon X_1 \to X_2$ such that $f_2^{-1} \circ g \circ f_1$ is bounded isotopic to the identity in $X_0$. We can therefore choose lifts $\widetilde f_1 \colon \widetilde X_0 \to \widetilde X_1$, $\widetilde f_2 \colon \widetilde X_0 \to \widetilde X_2$,  $\widetilde g \colon \widetilde X_1 \to \widetilde X_2$ of these diffeomorphisms so that $\widetilde f_2^{-1} \circ \widetilde g \circ \widetilde f_1$ is bounded isotopic to the identity in $\widetilde X_0$. A bounded isotopy fixes the boundary at infinity $\partial_\infty \widetilde X_0$; indeed, assuming $\widetilde X_0= \D$ without loss of generality, the euclidean distance by which a bounded isotopy moves a point $x\in \D$ tends to 0 as $x$ approaches the boundary circle $\partial \D$.  This implies that the restrictions of $\widetilde f_2$ and $\widetilde g \circ \widetilde f_1$  to maps $\partial_\infty \widetilde X_0  \to \partial_\infty \widetilde X_2$ coincide. As the biholomorphic map $\widetilde g$ sends the Liouville measure $L_{\widetilde X_1}$ to $L_{\widetilde X_2}$, it follows that the measures $L_{f_1}$ and $L_{f_2}$ coincide on $G(\widetilde X_0)$.

As a consequence, the Liouville geodesic current $L_f \in \CC(X_0)$ depends only on the element $[f] $ of the Teichm\"uller space $ \T(X_0)$  represented by the quasiconformal diffeomorphism $f \colon X_0 \to X$. 

The map 
$$
\LL \colon \T(X_0) \to \CC(X_0)
$$
defined by the property that $\LL \big( [f] \big) = L_f$ is the  \emph{Liouville embedding}.

\begin{thm}
\label{thm:LiouvilleEmbedding}
 Let $X_0$ be a conformally hyperbolic Riemann surface, let the Teichm\"uller space $\T(X_0)$ be equipped with the Teichm\"uller distance $\dT$, and let the space $\CC(X_0)$ of bounded geodesic currents be endowed with the uniform weak* topology defined in {\upshape \S \ref{sect:GeodesicCurr}}. Then, the Liouville embedding $\LL \colon \T(X_0) \to \CC(X_0)$ is a homeomorphism onto its image, it is a proper  map, and its image $\LL \bigl( \T(X_0) \bigr)$ is closed in $\CC(X_0)$.  
\end{thm}

\begin{rem}
\label{rem:UniformWeak*Needed}
 The above statement would be false if $\CC(X_0)$ was only endowed with the usual weak* topology. Indeed, consider a sequence $\{g_n\}_{n\in\N}$ of geodesics of the disk $\D$ that leaves every compact subset of $G(\D)$. For any $[f_0] \in \T(\D)$, let $[f_n] = E^1_{g_n}[f_0]$ be obtained from $[f_0]$  by performing an elementary earthquake along $g_n$ (see \S \ref{subsect:ElemEartqke}). Then, for every compact subset $K \subset G(\D)$, the measure $\LL\big( [f_n] \big)$ coincides with $\LL\big( [f_0] \big)$ on $K$ for $n$ sufficiently large. It follows that the sequence $\big\{ \LL\big( [f_n] \big) \big\}_{n\in\N}$ converges to  $\LL\big( [f_0] \big)$ for the weak* topology as $n$ tends to infinity. However,  the Teichm\"uller distance $\dT \big( [f_0], [f_n] \big)>0$ is constant and $[f_n]$ consequently does not converge to $[f_0]$ for the Teichm\"uller metric on $\T(X_0)$. This shows that the inverse map $\LL^{-1} \colon  \LL \big( \T(X_0) \big) \to \T(X_0)$ is not continuous when its domain is only endowed with the weak* topology, so that the uniform weak* topology is really needed.  
\end{rem}

 The proof of Theorem~\ref{thm:LiouvilleEmbedding} will take a while. It will be proved in several steps, as Propositions~\ref{prop:LiouvilleContinuous}, \ref{prop:InverseLiouvilleContinuous}, \ref{prop:LiouvilleClosed} and \ref{prop:LiouvilleProper} below. We first introduce a few technical tools to connect the quasiconformal geometry of Riemann surfaces to measures on spaces of geodesics.

 \subsection{Boxes of geodesics}
 \label{subsect:GeodesicBoxes}
 
 Let $\widetilde X$ be a simply connected conformally hyperbolic Riemann surface, and let $\partial_\infty \widetilde X$ be its circle at infinity. Typically, $\widetilde X$ will be the universal cover of a conformally hyperbolic Riemann surface $X$. 
 
 The orientation of $\widetilde X$ specifies a boundary (counterclockwise) orientation for $\partial_\infty \widetilde X$. In particular, two points $a$, $b\in \partial_\infty \widetilde X$ delimit a unique interval $[a,b] \subset \partial_\infty \widetilde X$, consisting of those points $x$ such that $a$, $x$, $b$ occur in this order for the counterclockwise orientation of $\partial_\infty \widetilde X$. Note that $[b,a]$ is different from $[a,b]$, and that $[a,b] \cup [b,a]= \partial_\infty \widetilde X$. 
 
 Four distinct points $a$, $b$, $c$, $d\in \partial_\infty \widetilde X$, occurring counterclockwise in this order, determine two disjoint intervals $[a,b]$, $[c,d] \subset \partial_\infty \widetilde X$ and a subset $Q= [a,b] \times [c,d]$ of the space of geodesics $G( \widetilde X) = \partial_\infty \widetilde X \times \partial_\infty \widetilde X - \Delta$. We will refer to such a subset $Q$ as a \emph{box of geodesics of $ \widetilde X$}, or as a \emph{box in $G(\widetilde X)$}.
 
 For the unit disk $\D$ and its Liouville geodesic current $L_\D \in \CC( \D)$, a simple integral computation expresses the Liouville mass of a box of geodesics in terms of  the cross-ratio of the four points of $ \partial \D$ determining this box. 
 
\begin{lem}
\label{lem:LiouvilleAndCrossratios}
\pushQED{\qed}
For a box of geodesics $Q= [a,b] \times [c,d] \subset G(\D)$ with $a$, $b$, $c$, $d\in \partial \D  \subset \C$, 
\begin{equation*}
L_\D \big([a,b]\times [c,d]\big)= \iint_Q \frac{ds\,dt}{|\E^{\I s} - \E^{\I t} |^2}= 
\log\frac{(a-c)(b-d)}{(a-d)(b-c)}. \qedhere
\end{equation*}
\end{lem}

\begin{lem}
 \label{lem:LiouvilleDeterminesBox}
 Let $Q$ and $Q' \subset G(\widetilde X)$ be two boxes of geodesics in $\widetilde X$. There exists a biholomorphic map $\widetilde X \to \widetilde X$ sending $Q$ to $Q'$ if and only if they have the same Liouville mass $L_{\widetilde X}(Q) = L_{\widetilde X}(Q')$. 
\end{lem}
\begin{proof}
Using a biholomorphic map $\widetilde X \to \D$, we can assume without loss of generality that $\widetilde X= \D$. Then, the biholomorphic maps of $\D$ are the linear fractional maps
$
z \mapsto \frac{\alpha z + \beta}{\bar\beta z + \bar\alpha}
$
where $\alpha$, $\beta \in \C$ are such that $|\alpha|^2 - |\beta|^2=1$. Elementary algebra shows that, given two boxes $Q= [a,b] \times [c,d]$ and $Q'= [a',b'] \times [c',d']$ in $G(\D)$, there exists such a linear fractional map sending $Q$ to $Q'$ if and only if the cross-ratios $\frac{(a-c)(b-d)}{(a-d)(b-c)}$ and $\frac{(a'-c')(b'-d')}{(a'-d')(b'-c')}$ are equal. By  Lemma~\ref{lem:LiouvilleAndCrossratios}, this is equivalent to the property that the Liouville masses $L_\D(Q)$ and $L_\D(Q')$ are equal. 
\end{proof}

For a box of geodesics $Q= [a,b] \times [c,d] \subset G( \widetilde X) $, its \emph{orthogonal box} is the box $Q^\perp = [b,c]\times [d,a]$. 

Note that the definition is not quite as symmetric as one would hope, as $Q^{\perp\perp}$ is different from $Q$. In fact, $Q^{\perp\perp} = [c,d] \times [a,b]$ consists of all geodesics obtained by reversing the orientation of the geodesics of $Q$. In particular, $Q^{\perp\perp}$ has the same $\alpha$--mass as $Q$ for any balanced geodesic current, and the distinction between $Q$ and $Q^{\perp\perp}$ will consequently have little impact in this article since most geodesic currents considered here will be balanced (as defined at the end of \S \ref{subsect:GeodesicCurr}). 

\begin{lem}
\label{lem:LiouvilleMassOrthoBox}
Let $L_{\widetilde X}$ be the Liouville measure of a simply connected conformally hyperbolic Riemann surface $\widetilde X$. For every box of geodesics $Q\subset G(\widetilde X)$,
 $$
 \E^{-L_{\widetilde X}(Q)} + \E^{-L_{\widetilde X}(Q^\perp)} =1.
 $$
\end{lem}
\begin{proof}
 Using a biholomorphic map $\widetilde X \to \D$, we can assume without loss of generality that $\widetilde X=X=\D$. Then, for a box $Q= [a,b] \times [c,d] \subset G( \D) $, Lemma~\ref{lem:LiouvilleAndCrossratios}  gives
\begin{align*}
 \E^{-L_{\D}(Q)} + \E^{-L_{\D}(Q^\perp)} 
 &= \frac {(a-d)(b-c)}{(a-c)(b-d)} + \frac {(b-a)(c-d)}{(b-d)(c-a)}\\
 &=  \frac {(a-d)(b-c)-(b-a)(c-d)}{(a-c)(b-d)} =1. \qedhere
\end{align*}
\end{proof}

\subsection{Quasiconformal and quasisymmetric homeomorphisms}
\label{sect:QuasiconfQuasisym}
Consider a quasiconformal diffeomorphism $f\colon X_1 \to X_2$ between conformally hyperbolic Riemann surfaces, and lift it to a map $\widetilde f \colon \widetilde X_1 \to \widetilde X_2$ between their universal cover. We already mentioned the Beurling-Ahlfors Theorem, which says that $\widetilde f$ has a continuous extension $\widetilde f \colon \widetilde X_1 \cup \partial_\infty \widetilde X_1 \to \widetilde X_2 \cup \partial_\infty \widetilde X_2$ to the closed disks obtained by adding their circles at infinity to $\widetilde X_1$ and $\widetilde X_2$. 
The Beurling-Ahlfors Theorem additionally relates the quasiconformal properties of $\widetilde f \colon \widetilde X_1  \to \widetilde X_2$ to another regularity property for the boundary extension $\widetilde f \colon  \partial_\infty \widetilde X_1 \to  \partial_\infty \widetilde X_2$, as we now explain. 

A box $Q\subset G(\widetilde X_1)$ is \emph{symmetric} if its Liouville mass $L_{\widetilde X_1}(Q)$ is equal to $\log 2$. This property is better explained if we translate it to the disk by a biholomorphic map $\widetilde X_0 \to \D$. Indeed, Lemma~\ref{lem:LiouvilleDeterminesBox} shows that a box $Q\subset G(\D)$ is symmetric if and only if  it is  the image $\phi\big( [1,\I] \times [-1, -\I] \bigr)$ under a biholomorphic map $\phi \in \HHH(\D)$ of the ``standard'' box $ [1,\I] \times [-1, -\I]$ delimited by the points $1$, $\I$, $-1$, $-\I \in \partial\D$. Another characterization is provided by the combination of Lemmas~\ref{lem:LiouvilleDeterminesBox} and \ref{lem:LiouvilleMassOrthoBox}, which shows that a box $Q$ is symmetric if and only if there is a biholomorphic map of $\widetilde X_1$ sending $Q$ to the orthogonal box~$Q^\perp$. 

A homeomorphism  $ \widetilde f \colon  \partial_\infty \widetilde X_1 \to  \partial_\infty \widetilde X_2$ is \emph{quasisymmetric} if the supremum
$$
M(\widetilde f) = \sup_{Q \text{ symmetric}} \frac{L_{\widetilde X} \big( \widetilde f(Q) \big)}{\log 2},
$$
as $Q$ ranges over all symmetric boxes $Q\subset G(\widetilde X_1)$, is finite. By definition, $M(h)$ is the \emph{quasisymmetric constant} of $h$. 

Note that $M(\widetilde f)= 1$ when $\widetilde f$ comes from a biholomorphic map $\widetilde X_1 \to \widetilde X_2$, and that in general $M(\widetilde f)\geq 1$ by Lemma~\ref{lem:LiouvilleMassOrthoBox}. 

\begin{rem} The quasisymmetry property is sometimes stated in a different way, by restricting attention to  homeomorphisms $f\colon \R \to \R$ and by requiring that the supremum
$$
H(f) = \sup \{ {\textstyle \frac{|f(x+t)-f(x)|}{|f(x)-f(x-t)|}}; x,t \in \R \}
$$
be finite; to clarify the terminology, let us say that a homeomorphism $f\colon \R \to \R$ satisfying this property is \emph{weakly quasi-symmetric} (compare \cite{TV}). If we identify $\R \cup \{\infty\}$ to $S^1 = \partial \D$ by stereographic projection, a simple algebraic manipulation shows that $ \log(1+H(f)) \leq M(f)$. As a consequence, if the extension $\R \cup \{ \infty \} \to \R \cup \{ \infty \}$  of $f \colon \R \to \R$ is quasisymmetric, then $f$ is weakly quasisymmetric. 
A consequence of the proof \cite{BA} of the Beurling-Ahlfors Theorem~\ref{thm:Beurling-Ahlfors} stated below is that the converse holds, namely that the extension $\R \cup \{ \infty \} \to \R \cup \{ \infty \}$  of a homeomorphism $f \colon \R \to \R$ is quasisymmetric if and only if $f$ is weakly quasisymmetric. Indeed, that proof only uses the weak quasisymmetry property, whereas the boundary extension of a quasiconformal diffeomorphism is quasisymmetric. 
\end{rem}

The following fundamental result connects quasiconformal diffeomorphisms between Riemann surfaces and quasisymmetric homeomorphisms between their circles at infinity. 

\begin{thm}[Beurling-Ahlfors]
\label{thm:Beurling-Ahlfors}
 Let $\widetilde X_1$ and $\widetilde X_2$ be two simply connected conformally hyperbolic Riemann surfaces. Every quasiconformal diffeomorphism $\widetilde f \colon \widetilde X_1 \to \widetilde X_2$ admits a unique  extension to a homeomorphism $\widetilde X_1 \cup \partial_\infty \widetilde X_1\to \widetilde X_2  \cup \partial_\infty \widetilde X_2$, whose restriction $\widetilde f \colon \partial_\infty \widetilde X_1\to  \partial_\infty \widetilde X_2$ to the circles at infinity is quasisymmetric. In addition, the quasisymmetric constant $M(\widetilde f)$ of $\widetilde f \colon \partial_\infty \widetilde X_1\to  \partial_\infty \widetilde X_2$ tends to $1$ as the quasiconformal dilatation $K(\widetilde f)$ of $\widetilde f \colon \widetilde X_1 \to \widetilde X_2$ tends to $1$. 
 
 Conversely, every quasisymmetric homeomorphism $\widetilde f \colon \partial_\infty \widetilde X_1\to  \partial_\infty \widetilde X_2$ admits a continuous extension $\widetilde X_1 \cup \partial_\infty \widetilde X_1\to \widetilde X_2  \cup \partial_\infty \widetilde X_2$, whose restriction $\widetilde f \colon \widetilde X_1 \to \widetilde X_2$ is a quasiconformal diffeomorphism. In addition, the extension can be chosen so that the quasiconformal dilatation $K(\widetilde f)$ of $\widetilde f \colon \widetilde X_1 \to \widetilde X_2$  is bounded by a constant $K'(\widetilde f)$ depending only on the quasisymmetric constant $M(\widetilde f)$ of $\widetilde f \colon \partial_\infty \widetilde X_1\to  \partial_\infty \widetilde X_2$, and tending to $1$ as   $M(\widetilde f)$ tends to $1$.
\end{thm}

\begin{proof} See \cite{BA}, \cite[\S II.6]{LV} or \cite[\S16]{GL}, for instance. 
\end{proof}

Although the definition  of a quasisymmetric homeomorphism $\widetilde f \colon \partial_\infty \widetilde X_1\to  \partial_\infty \widetilde X_2$ involves only symmetric boxes, the quasisymmetry property actually controls the Liouville mass $L_{\widetilde X_2} \big( \widetilde f(Q) \big)$ for all boxes $Q\subset G(\widetilde X_1)$.

\begin{prop}
\label{prop:QuasiSymmetricEta}
If a  homeomorphism  $\widetilde f \colon  \partial_\infty \widetilde X_1 \to  \partial_\infty \widetilde X_2$ is quasisymmetric, there exists a homeomorphism $\omega\colon \left[0,\infty\right[ \to \left[0,\infty\right[$ depending only on the quasisymmetric constant $M(\widetilde f)$ such that 
 $$
  L_{\widetilde X_2} \big( \widetilde f(Q) \big) \leq \omega \big (  L_{\widetilde X_1} \big(Q) \big) 
 $$
 for every box $Q \subset G(\widetilde X_1)$. 
 
 In addition, the homeomorphism $\omega$ can be chosen so that it converges to the identity, uniformly on compact subsets of the open interval $\left]0,\infty\right[$, as the quasisymmetric constant $M(\widetilde f)$ tends to $1$. 
\end{prop}

\begin{proof} Although there exists direct proofs of the first half of the statement (see for instance \cite{TV}), it is easier to use the full force of the Beurling-Ahlfors Theorem~\ref{thm:Beurling-Ahlfors}.  

In addition to its Liouville mass  $L_{\widetilde X_1} (Q) $, a box $Q=[a,b] \times [c,d]$ in $G(\widetilde X_1)$ has a more complex analytic invariant, its \emph{conformal modulus} $\mu_{\widetilde X_1}(Q)$. This is defined as the number $\mu= \mu_{\widetilde X_1}(Q)$ for which there exists a homeomorphism $\widetilde X_1 \cup \partial_\infty \widetilde X_1 \to [0,\mu] \times [0,1]$ that is conformal on $\widetilde X$ and sends the corners $a$, $b$, $c$, $d\in \partial_\infty \widetilde X$ of $Q$ to the corners $(0,0)$, $(\mu, 0)$, $(\mu, 1)$, $(0,1)$ of the rectangle $[0,\mu] \times [0,1] \subset \R^2$, respectively. 
These two invariants are classically related by an increasing homeomorphism $\eta \colon \left]0,\infty \right[ \to \left]0,\infty \right[ $ such that $\mu_{\widetilde X_1}(Q) = \eta \big(L_{\widetilde X_1} (Q) \big)$; indeed, these two quantities depend continuously on the corners $a$, $b$, $c$, $d$ of $Q$, they both increase as $Q$ gets larger,  they tend to $0$ as $Q$ gets arbitrarily small, and they tend to $+\infty$ as $Q$ gets arbitrarily large. 

Let $\widetilde f \colon  \widetilde X_1 \to   \widetilde X_2$ be the quasiconformal extension of $\widetilde f \colon  \partial_\infty \widetilde X_1 \to  \partial_\infty \widetilde X_2$  provided by Theorem~\ref{thm:Beurling-Ahlfors}. In particular, this quasiconformal extension  can be chosen so that its quasiconformal dilatation $K(\widetilde f)$ is bounded by a constant $K'(\widetilde f)$ depending only on the quasisymmetric constant $M(\widetilde f)$, and tending to 1 as   $M(\widetilde f)$ tends to 1. A  fundamental consequence of quasiconformality is that
$$
 \mu_{\widetilde X_2} \big( \widetilde f(Q) \big) \leq K(\widetilde f)  \, \mu_{\widetilde X_1} (Q);
 $$
 see for instance \cite{Ahl, LV}. 
 Proposition~\ref{prop:QuasiSymmetricEta} then holds for the homeomorphism $\omega$ defined by $\omega(t)= \eta^{-1} \big(K'(\widetilde f) \eta(t)\big)$. 
\end{proof}

An immediate consequence of Proposition~\ref{prop:QuasiSymmetricEta}  is that, if $\widetilde f \colon  \partial_\infty \widetilde X_1 \to  \partial_\infty \widetilde X_2$ is quasisymmetric, so is its inverse $\widetilde f^{-1} \colon  \partial_\infty \widetilde X_2 \to  \partial_\infty \widetilde X_1$.

%\begin{rem}
% Another classical proof of Proposition~\ref{prop:QuasiSymmetricEta}  follows from the Beurling-Ahlfors Theorem~\ref{thm:Beurling-Ahlfors} below. 
%\end{rem}

We  now have the tools to prove  Lemma~\ref{lem:LiouvilleBounded}, a task which we had temporarily postponed. We rephrase this statement in the following way.

\begin{lem}
\label{lem:LiouvilleBounded2}
 Let $\widetilde f \colon \widetilde X_1 \to \widetilde X_2$ be a quasiconformal diffeomorphism between  two simply connected conformally hyperbolic Riemann surfaces. Then, for every continuous function $\xi \colon G(\widetilde X_1) \to \R$ with compact support, the supremum
 $$
 \sup_{\phi \in \HHH(\widetilde X_1)} \Big\vert \int_{G(\widetilde X_1)} \xi\circ \phi\, dL_{\widetilde f}  \Big\vert
 $$
 is finite, where the supremum is taken over all biholomorphic maps $\phi \colon \widetilde X_1 \to \widetilde X_1$ and where $L_{\widetilde f}$ is the pull back  under $\widetilde f$ of the Liouville measure $L_{\widetilde X_2}$ of $\widetilde X_2$. 
 \end{lem}
\begin{proof}
Cover the support of $\xi$ by finitely many boxes $Q_1$, $Q_2$, \dots, $Q_k \subset G(\widetilde X_1)$. Then, for every $\phi \in \HHH(\widetilde X_1)$
\begin{align*}
\Big\vert \int_{G(\widetilde X_1)} \xi\circ \phi \, dL_{\widetilde f} \Big\vert 
&\leq \Big(  \max_{g \in G(\widetilde X_1)} | \xi(g)|\Big)  \sum_{i=1}^k L_{\widetilde f} \big(\phi^{-1}(Q_i)\big)\\
&\leq \Big(\max_{g \in G(\widetilde X_1)} | \xi(g)|\Big)  \sum_{i=1}^k L_{\widetilde X_2} \big(\widetilde f\circ \phi^{-1}(Q_i)\big).
\end{align*}
Since  $\widetilde f \colon \partial_\infty \widetilde X_1 \to \partial_\infty \widetilde X_2$ is quasisymmetric,  Proposition~\ref{prop:QuasiSymmetricEta} provides a function $\omega$ such that, for each box $Q_i \subset G(\widetilde X_1)$,
$$
L_{\widetilde X_2} \big(\widetilde f\circ \phi^{-1}(Q_i)\big) 
\leq \omega \big( L_{\widetilde X_1} (\phi^{-1}Q_i ) \big) = \omega \big( L_{\widetilde X_1} (Q_i ) \big).
$$
This gives the uniform bound requested. 
\end{proof}

Theorem~\ref{thm:Beurling-Ahlfors}  provides a correspondence between quasiconformal diffeomorphisms between simply connected Riemann surfaces and quasisymmetric homeomorphisms between their boundaries at infinity. We will need a slight improvement of this correspondence for maps between Riemann surfaces that are not simply connected.

Lift a quasiconformal map $f\colon X_1 \to X_2$ to a quasiconformal diffeomorphism $\widetilde f\colon \widetilde X_1 \to \widetilde X_2$ between universal covers, and consider the quasisymmetric extension $\widetilde f \colon \partial_\infty \widetilde X_1\to  \partial_\infty \widetilde X_2$ provided by the first part of Theorem~\ref{thm:Beurling-Ahlfors}. The quasisymmetry property is invariant under composition with biholomorphic maps of $\widetilde X_2$ (as these respect the Liouville measure $L_{\widetilde X_2}$).  It follows that the quasisymmetric constant $M(\widetilde f)$ is independent of the choice of the lift $\widetilde f\colon \widetilde X_1 \to \widetilde X_2$. We will refer to $M(\widetilde f)$ as \emph{the quasisymmetric constant} $M(f)$ of the quasiconformal map $f\colon X_1 \to X_2$.

The first part of Theorem~\ref{thm:Beurling-Ahlfors} indicates that this quasisymmetric constant $M(f)$ is close to 1 when the quasiconformal dilatation $K(f)$ is close to 1. We will need the following converse statement, which improves the second part of Theorem~\ref{thm:Beurling-Ahlfors} by ensuring that the quasiconformal extension $\widetilde f \colon  \widetilde X_1\to   \widetilde X_2$ comes from a quasiconformal diffeomorphism $ f \colon   X_1\to    X_2$.

\begin{thm}
\label{thm:BeurlingAhlforsImproved}

Let $f \colon X_1 \to X_2$ be a quasiconformal diffeomorphism between conformally hyperbolic Riemann surfaces, and let $M(f)$ be its quasisymmetric constant. Then, there is another quasiconformal diffeomorphism $f' \colon X_1 \to X_2$ that is bounded isotopic to $f$ and whose quasiconformal dilatation $K(f')$ is bounded by a constant depending only on the quasisymmetric constant $M(f)=M(f')$. In addition, $f'$ can be chosen so that its quasiconformal dilatation $K(f')$ tends to $1$ as the quasisymmetric constant $M(f)$ tends to $1$. 
 \end{thm}

\begin{proof}
 As usual, lift $f$ to $\widetilde f \colon \widetilde X_1 \to \widetilde X_2$, and consider the quasisymmetric extension $\widetilde f \colon \partial_\infty \widetilde X_1\to  \partial_\infty \widetilde X_2$.  A fundamental construction of Douady-Earle \cite{DE} provides another continuous extension $\widetilde f' \colon \widetilde X_1 \cup \partial_\infty \widetilde X_1\to \widetilde X_2  \cup \partial_\infty \widetilde X_2$ of $\widetilde f \colon \partial_\infty \widetilde X_1\to  \partial_\infty \widetilde X_2$ such that $\widetilde f' \colon \widetilde X_1 \to \widetilde X_2$ is quasiconformal, which has the additional property that it is equivariant with respect to the action of the biholomorphic maps of $\widetilde X_1$ and $\widetilde X_2$. Namely, for every biholomorphic map $\phi_1 \in \HHH(\widetilde X_1)$ and $\phi_2 \in \HHH(\widetilde X_2)$, the Douady-Earle quasiconformal extension of $\phi_1\circ  \widetilde f\circ \phi_2 \colon  \partial_\infty \widetilde X_1\to  \partial_\infty \widetilde X_2$ is $\phi_1\circ \widetilde f'\circ \phi_2  \colon \widetilde X_1 \to \widetilde X_2$. In addition, we still have the property that the quasiconformal constant $K(\widetilde f')$ of the Douady-Earle extension  tends to $1$ as the quasisymmetric constant $M(\widetilde f)$ tends to 1 (although the bound is not as good as for the Beurling-Ahlfors Theorem). 
 
 Applying the equivariance property to the (biholomorphic) actions of the fundamental group $\pi_1(X_1)=\pi_1(X_2)$ on $\widetilde X_1$ and $\widetilde X_2$, it follows that $\widetilde f' \colon \widetilde X_1 \to \widetilde X_2$ descends to a quasiconformal map $f'\colon X_1 \to X_2$. By construction, $K(f') = K(\widetilde f')$ tends to 1 as $M(f)=M(\widetilde f)$ tends to 1. 
 
 By construction, the quasisymmetric extensions $\widetilde f$, $\widetilde f' \colon \partial_\infty \widetilde X_1\to  \partial_\infty \widetilde X_2$ of the quasiconformal maps $\widetilde f$, $\widetilde f' \colon \widetilde X_0 \to \widetilde X$ coincide. A result of Earle-McMullen \cite{EM} then shows that $f$ and $f'$ are bounded isotopic. 
 \end{proof}
 
 \subsection{The Liouville embedding $\LL \colon \T(X_0) \to \CC(X_0)$  is injective}  We are now ready to begin proving Theorem~\ref{thm:LiouvilleEmbedding}. We begin with the easier part. 
 
 \begin{prop}
\label{prop:LiouvilleInjective}
 The Liouville embedding $\LL \colon \T(X_0) \to \CC(X_0)$  is injective. 
\end{prop}
\begin{proof}
 Suppose that $\LL\big( [f_1] \big) = \LL\big( [f_2] \big)$ for $[f_1]$, $[f_2] \in \T(X_0)$ represented by quasiconformal diffeomorphisms $f_1 \colon X_0 \to X_1$, $f_2 \colon X_0 \to X_2$. Lift $f_1$, $f_2$ to maps $\widetilde f_1 \colon \widetilde X_0 \to \widetilde X_1$, $\widetilde f_2 \colon \widetilde X_0 \to \widetilde X_2$ between universal covers, and consider the quasisymmetric extensions $\widetilde f_1 \colon \partial_\infty \widetilde X_0 \to \partial_\infty \widetilde X_1$, $\widetilde f_2 \colon\partial_\infty  \widetilde X_0 \to \partial_\infty \widetilde X_2$ provided by Theorem~\ref{thm:Beurling-Ahlfors}. 
 
 Since $\LL\big( [f_1] \big) = \LL\big( [f_2] \big)$, the homeomorphism $\widetilde f_2  \circ \widetilde f_1^{-1} \colon\partial_\infty  \widetilde X_1 \to \partial_\infty \widetilde X_2$ sends the Liouville measure $L_{\widetilde X_1}$ to $L_{\widetilde X_2}$. It follows that the quasisymmetric constant $M(\widetilde f_2  \circ \widetilde f_1^{-1}) = M( f_2  \circ  f_1^{-1})$ is equal to 1. By Theorem~\ref{thm:BeurlingAhlforsImproved}, it follows that $ f_2  \circ  f_1^{-1}$ is bounded isotopic to maps $g \colon X_1 \to X_2$ whose quasiconformal dilatation $K(g)$ is arbitrarily close to 1. This proves that the Teichm\"uller distance $\dT \big( [f_1], [f_2] \big)$ is equal to 0, so that $[f_1]=[f_2]$ in $\T(X_0)$ as required. 
 \end{proof}

\subsection{The Liouville embedding $\LL \colon \T(X_0) \to \CC(X_0)$  is continuous}  

We now prove a more substantial step in the proof of Theorem~\ref{thm:LiouvilleEmbedding}. 

\begin{prop}
\label{prop:LiouvilleContinuous}
 The Liouville embedding $\LL \colon \T(X_0) \to \CC(X_0)$  is continuous, for the Teichm\"uller topology on $\T(X_0)$ and the uniform weak* topology on $\CC(X_0)$. 
\end{prop}
\begin{proof} The Teichm\"uller space is endowed with the topology defined by the Teichm\"uller metric $\dT$, and the uniform weak* topology on $\CC(X_0)$ is metrizable by  Lemma~\ref{lem:GeodCurrentSpaceMetrizable}.  It therefore suffices to show that, for every sequence $\big\{ [f_n] \big\}_{n \in \N}$  converging to $[f_\infty]$ in $\T(X_0)$, the sequence  of Liouville geodesic currents $ \LL\big([f_n]\big)=L_{f_n}$ converges to $\LL\big([f_\infty]\big)=L_{f_\infty}$ in $\CC(X_0)$ for the uniform weak* topology.
By definition of the uniform weak* topology, this means that 
$$
\sup_{\phi \in \HHH(\widetilde X_0)} 
\Bigl\vert\int_{G(\widetilde X_0)} \xi\circ \phi \, dL_{f_n}-
\int_{G(\widetilde X_0)} \xi\circ \phi \,dL_{f_\infty}
\Bigr\vert  \to 0 \text{ as } n\to \infty
$$
for every  continuous function $\xi \colon G(\widetilde X_0) \to \R$ with compact support. 

As a first step, we begin by proving a similar statement for boxes of geodesics in~$\widetilde X_0$. 
\begin{lem}
\label{lem:ConvergenceLiouvilleBoxes}
 For every box $Q\subset G(\widetilde X_0)$,
 $$
\sup_{\phi \in \HHH(\widetilde X_0)} 
\Bigl\vert L_{f_n} (\phi (Q) ) -L_{f_\infty}(\phi (Q))
\Bigr\vert  \to 0 \text{ as } n\to \infty.
$$
\end{lem}
\begin{proof} By definition of the Teichm\"uller topology, the classes $[f_n]$, $[f_\infty] \in \T(X_0)$ can be represented by quasiconformal maps $f_n \colon X_0 \to X_n$ and $f_\infty \colon X_0 \to X_\infty$ such that the quasiconformal constant $K(f_n \circ f_\infty^{-1})$ tends to 1 as $n$ tends to $\infty$. 

Lift $f_n$ and $f_\infty$ to quasiconformal maps $\widetilde f_n \colon \widetilde X_0 \to \widetilde X_n$ and $\widetilde f_\infty \colon \widetilde X_0 \to \widetilde X_\infty$, respectively, and consider their quasisymmetric extensions $\widetilde f_n \colon \partial_\infty \widetilde X_0 \to \partial_\infty  \widetilde X_n$ and $\widetilde f_\infty \colon \partial_\infty  \widetilde X_0 \to \partial_\infty  \widetilde X_\infty$ to the circles at infinity. 

A first observation is that, as $\phi \in \HHH(\widetilde X_0)$ ranges over all biholomorphic maps of $\widetilde X_0$, the Liouville mass $L_{\widetilde X_0}\big(\phi (Q) \big)$ is constant by invariance of the Liouville measure $L_{\widetilde X_0}$ under the action of  $\HHH(\widetilde X_0)$. Applying Proposition~\ref{prop:QuasiSymmetricEta} to the quasisymmetric maps $\widetilde f_\infty$ and $\widetilde f_\infty^{-1}$ then shows that $L_{\widetilde X_\infty}\big( \widetilde f_\infty(\phi (Q))\big)$ stays in a compact subset of the interval $\left]0,\infty\right[$, independent of $\phi \in \HHH(\widetilde X_0)$. 

Since the quasiconformal dilatation $K(\widetilde f_n \circ \widetilde f_\infty^{-1})= K(f_n \circ f_\infty^{-1})$ tends to 1, it follows from Theorem~\ref{thm:Beurling-Ahlfors} that the quasisymmetric constant  $M(\widetilde f_n \circ \widetilde f_\infty^{-1})$ of $\widetilde f_n \circ \widetilde f_\infty^{-1} \colon \partial_\infty  \widetilde X_\infty \to  \partial_\infty  \widetilde X_n$ tends to 1 as $n \to \infty$. By Proposition~\ref{prop:QuasiSymmetricEta} and using the property that $L_{\widetilde X_\infty} \big( \widetilde f_\infty(\phi (Q)) \big)$ is bounded away from $0$ and $\infty$, it follows that
$$
\limsup_{n\to \infty} \frac{L_{f_n}(\phi (Q)) }{L_{f_\infty} (\phi (Q)) }
=\limsup_{n\to \infty} \frac{L_{\widetilde X_n} \Big( \widetilde f_n\circ \widetilde f_\infty^{-1} \big(\widetilde f_\infty(\phi (Q)) \big)\Big) }{L_{\widetilde X_\infty} \big( \widetilde f_\infty(\phi (Q)) \big) }
 \leq 1,
$$
and this uniformly in $\phi \in \HHH(\widetilde X_0)$. 

Similarly, since $K(f_n) \leq K(f_n \circ f_\infty^{-1}) K(f_\infty)$, the maps $f_n \colon X_0 \to X_n$ are uniformly quasiconformal and, as above, the Liouville masses $L_{f_n} \big(\phi (Q) \big)= L_{\widetilde X_n} \big( \widetilde f_n(\phi (Q)) \big)$ stay bounded away from 0 and $\infty$. Replacing $\widetilde f_n \circ \widetilde f_\infty^{-1}$ by $\widetilde f_\infty \circ \widetilde f_n^{-1}$ in the argument above gives that 
$$
\limsup_{n\to \infty} \frac{L_{f_\infty}(\phi (Q)) }{L_{f_n} (\phi (Q)) }
=\limsup_{n\to \infty} \frac{L_{\widetilde X_\infty} \Big( \widetilde f_\infty\circ \widetilde f_n^{-1} \big(\widetilde f_n(\phi (Q)) \big)\Big) }{L_{\widetilde X_n} \big( \widetilde f_n(\phi (Q)) \big) }
 \leq 1,
$$
 uniformly in $\phi \in \HHH(\widetilde X_0)$.
 
 Therefore, 
 $$ 
 \lim_{n\to \infty} \frac{L_{f_n}(\phi (Q)) }{L_{f_\infty} (\phi (Q)) }=1
 $$
 uniformly in $\phi \in \HHH(\widetilde X_0)$. Since $L_{f_\infty} (\phi (Q))$ is uniformly bounded away from 0 and $\infty$, it follows that $L_{f_n} (\phi (Q))$ tends to $L_{f_\infty} (\phi (Q))$  as $n\to \infty$, and this uniformly in $\phi \in \HHH(\widetilde X_0)$.  This proves Lemma~\ref{lem:ConvergenceLiouvilleBoxes}. 
\end{proof}

We now return to the proof of Proposition~\ref{prop:LiouvilleContinuous}. Consider a continuous test function $\xi \colon G(\widetilde X_0) \to \R$ with compact support.

We begin by covering the support of $\xi$ by finitely many boxes $Q_1$, $Q_2$, \dots, $Q_m \subset G(\widetilde X_0)$. 

For a number $\epsilon_0>0$ to be specified later, we then cover the support of $\xi$ by finitely many boxes $Q_1'$, $Q_2'$, \dots, $Q_{m'}' \subset G(\widetilde X_0)$, contained in the union of the boxes $Q_i$ and small enough that
\begin{equation}
\label{eqn:Continuity0}
\bigl| \max_{x\in Q_i'}  \xi(x)-  \min_{x\in Q_i'} \xi(x) \bigr| < \epsilon_0.
\end{equation}
After subdividing these boxes $Q_i' = [a_i, b_i] \times [c_i, d_i]$, we can arrange that the boxes $Q_i'$ have disjoint interiors. We then approximate $\xi$ by the step function
$$
\sigma=\sum_{i=1}^{m'} \xi(x_i^*) \chi_{Q_i'}
$$
where $x_i^*$ is an arbitrary point of $Q_i'$ and where $\chi_{Q_i'}\colon G(\widetilde X_0) \to \R$ is the characteristic function of $Q_i'$. By construction, $|\xi-\sigma|\leq \epsilon_0$ except possibly on the boundary of the boxes $Q_i'$.

Then, for every $\phi \in \HHH(\widetilde X_0)$, 
\begin{equation}
\begin{split}
\label{eqn:Continuity1}
\Bigl\vert  \int_{G(\widetilde X_0)}( \xi \circ \phi - \sigma \circ \phi) \,&d(L_{f_n} - L_{f_\infty}) \Bigr\vert\\
&\leq \epsilon_0 \sum_{i=1}^{m'} \Bigl(L_{f_n}\bigl(\phi^{-1}(Q_i') \bigr)+ L_{f_\infty} \bigl( \phi^{-1}(Q_i') \bigr) \Bigr )
\\
&\leq \epsilon_0 \sum_{j=1}^{m} \Bigl(L_{f_n}\bigl(\phi^{-1}(Q_j) \bigr)+ L_{f_\infty} \bigl( \phi^{-1}(Q_j) \bigr) \Bigr )
\end{split}
\end{equation}
using the properties that the boundary of a box has Liouville measure 0 and that $\bigcup_{i=1}^{m'} Q_i'$ is contained in $\bigcup_{j=1}^m Q_j$.

Similarly, once we have chosen the boxes $Q_i'$ to approximate $\xi$ by a step function, Lemma~\ref{lem:ConvergenceLiouvilleBoxes} shows that
\begin{equation}
\begin{split}
 \label{eqn:Continuity2}
 \Bigl\vert  \int_{G(\D)}(  \sigma \circ \phi) \,d(L_{f_n}-&L_{f_\infty}) \Bigr\vert     \\
 &=  \Bigl\vert \sum_{i=1}^{m'} \xi\bigl( \phi(x_i^*) \bigr) \Bigl( L_{f_n} \bigl(\phi^{-1}(Q_i') \bigr)-L_{f_\infty}\bigl( \phi^{-1}(Q_i') \bigr) \Bigr ) \Bigr\vert
 \\
&\to 0 \text { as } n \to \infty,
\end{split}
\end{equation}
and this uniformly in $\phi \in \HHH(\widetilde X_0)$. 

Suppose that we are given $\epsilon>0$, and that we  have chosen the boxes $Q_j$ to cover the support of $\xi$. Once this choice is made, Lemma~\ref{lem:ConvergenceLiouvilleBoxes} then shows that the term
$$
\sum_{j=1}^{m} \Bigl(L_{f_n}\bigl(\phi^{-1}(Q_j) \bigr)+ L_{f_\infty} \bigl( \phi^{-1}(Q_j) \bigr) \Bigr )
$$
occurring on the last line of Equation~(\ref{eqn:Continuity1}) is uniformly bounded. We can therefore pick a number  $\epsilon_0>0$ so that the contribution of (\ref{eqn:Continuity1}) is less than $\epsilon/2$.  After choosing the boxes $Q_i'$ so that (\ref{eqn:Continuity0}) holds for this $\epsilon_0$, the contribution of (\ref{eqn:Continuity2}) will be less than $\epsilon/2$ for $n$ sufficiently large. Combining (\ref{eqn:Continuity1}) and (\ref{eqn:Continuity2}), we conclude that 
$$
\Bigl\vert\int_{G(\widetilde X_0)} \xi\circ \phi \ d(L_{f_\infty}-L_{f_\infty})
\Bigr\vert < \epsilon
$$
for $n$ sufficiently large, and this uniformly in $\phi \in \HHH(\widetilde X_0)$. This proves the continuity property of Proposition~\ref{prop:LiouvilleContinuous}. 
\end{proof}

\subsection{The inverse map $\LL^{-1} \colon \LL \bigl( \T(X_0) \bigr) \to \T(X_0)$ is continuous}

\begin{prop}
\label{prop:InverseLiouvilleContinuous}

 The inverse $\LL^{-1} \colon \LL \bigl( \T(X_0) \bigr) \to \T(X_0)$ of the Liouville embedding $\LL \colon \T(X_0) \to \CC(X_0)$ is continuous, for the Teichm\"uller topology on $\T(X_0)$ and for the uniform weak* topology on $\CC(X_0)$. 
\end{prop}
\begin{proof} Consider an element $[f_\infty]$ and a 
sequence $\big\{[f_n]\big\}_{n \in \N}$ of elements of the Teichm\"uller space $\T(X_0)$ such that  the Liouville currents  $ L_{f_n}\in \CC(X_0)$ converge to $ L_{f_\infty}$ for the uniform weak* topology. We want to show that  $[f_n]$ converges to $[f_\infty] $ for the Teichm\"uller topology of $ \T(X_0)$.    

As usual, represent the class $[f_n]\in \T(X_0)$ by quasiconformal maps $f_n \colon X_0 \to X_n$, and consider their quasiconformal lifts $\widetilde f_n \colon  \widetilde X_0 \to \widetilde X_n$ and quasisymmetric extensions $\widetilde f_n \colon \partial_\infty \widetilde X_0 \to \partial_\infty \widetilde X_n$.

\begin{lem}
\label{lem:UniformlyQuasisym}
 The quasisymmetric constants $M(f_n)$ of the quasisymmetric maps $\widetilde f_n \colon \partial_\infty \widetilde X_0 \to \partial_\infty \widetilde X_n$ are uniformly bounded. 
\end{lem}

\begin{proof} We want to show that, as $Q\subset G(\widetilde X_0)$ ranges over all symmetric boxes in $\widetilde X_0$, the Liouville masses $L_{f_n}(Q)$ are uniformly bounded, independently of $n$ and $Q$. For this, choose a symmetric box $Q_0 \subset G(\widetilde X_0)$, and a test function $\xi \colon G(\widetilde X_0) \to \R$ with compact support such that $\xi\geq 1$ over the box $Q_0$. 

By definition of the uniform weak* topology, 
$$
\int_{G(\widetilde X_0)} \xi\circ \phi \, dL_{f_n} \to \int_{G(\widetilde X_0)} \xi\circ \phi \, dL_{f_\infty} \text{ as } n \to \infty
$$
uniformly over all biholomorphic maps $\phi \in \HHH(\widetilde X_0)$. The limit is uniformly bounded by Lemma~\ref{lem:LiouvilleBounded2}. It follows that  the integrals $\int_{G(\widetilde X_0)} \xi\circ \phi \, dL_{f_n} $ are  bounded by a constant $C$ independent of $n$ and $\phi \in \HHH(\widetilde X_0)$.

Every symmetric box $Q\subset G(\widetilde X_0)$ is of the form $\phi^{-1} (Q_0)$ for some $\phi \in \HHH(\widetilde X_0)$. Then, since $\xi\geq 1$ over  $Q_0$, 
$$
L_{\widetilde X_n}\big ( \widetilde f_n(Q) \big)=L_{f_n}( Q) = L_{f_n}(\phi^{-1} (Q_0)) \leq  \int_{G(\widetilde X_0)} \xi\circ \phi \, dL_{f_n}  \leq C
$$
so that the quasisymmetric constant $M( f_n)=M(\widetilde f_n)$ are bounded by $C/\log 2$.
\end{proof}

\begin{lem}
\label{lem:QuasisymTendsTo1}
 The quasisymmetric constant $M(f_n \circ f_\infty^{-1})$ converges to $1$ as $n$ tends to $\infty$. 
\end{lem}

\begin{proof} We will use a proof by contradiction. If the property does not hold, there exists an $\epsilon_0>0$ and a subsequence $\big\{ [f_{n_k}] \big\}_{k \in \N}$ such that $M(f_{n_k} \circ f_\infty^{-1})>1+\epsilon_0 $ for every $k$. (Recall that the quasisymmetric constant is always greater than or equal to 1). By definition of  the quasisymmetric constant, this means that there exists a symmetric box $Q_{n_k}'$ in $\widetilde X_\infty$ such that $L_{\widetilde X_{n_k}}\big(\widetilde f_{n_k} \circ \widetilde f_\infty^{-1} (Q_{n_k}') \big) > (1+\epsilon_0)\log 2$. We then have a box $Q_{n_k} = \widetilde f_\infty^{-1}(Q_{n_k}') \subset G(\widetilde X_0)$ such that $L_{f_\infty}(Q_{n_k}) = \log 2$ and $L_{f_{n_k}}(Q_{n_k}) > (1+\epsilon_0)\log 2$.
%\marginpar{Explain better. Dragomir is confused.}

Fix three points $a_0$, $b_0$, $c_0 \in \partial_\infty \widetilde X_0$, counterclockwise in this order. Then, there exists a biholomorphic map $\phi_{n_k} \in \HHH(\widetilde X_0)$ such that the box $\phi_{n_k}( Q_{n_k})$ is of the form $[a_0, b_0] \times [c_0, d_{n_k}]$ for some point $d_{n_k}$ in the open interval $\left] c_0, a_0 \right[ \subset \partial_\infty \widetilde X_0$. 

Since $\widetilde f_\infty \colon \widetilde X_0 \to \widetilde X_\infty$ is quasisymmetric and $L_{f_\infty}(Q_{n_k}) = \log 2$, Proposition~\ref{prop:QuasiSymmetricEta} shows that the Liouville mass $L_{\widetilde X_0}(\phi_{n_k}( Q_{n_k}))= L_{\widetilde X_0}(Q_{n_k})$ is bounded between two positive constants. It then follows from Lemma~\ref{lem:LiouvilleAndCrossratios} that the point $d_{n_k}$ stays within a compact subset of the  interval $\left]c_0, a_0\right[$. Refining the subsequence if necessary, we can therefore assume that $d_{n_k}$ converge to some point $d_\infty \in \left]c_0, a_0\right[$ as $k$ tends to $\infty$. In other words, the box  $\phi_{n_k} (Q_{n_k})$ converge to the box $Q_\infty = [a_0, b_0] \times [c_0, d_\infty]$  as $k$ tends to $\infty$. 

For an $\epsilon>0$ to be specified later, choose  intervals $\left]a_0', a_0''\right[ $, $\left]b_0'', b_0'\right[ $, $\left]c_0', c_0''\right[ $ and $\left]d_\infty'', d_\infty '\right[  \subset \partial_\infty \widetilde X_0 $ respectively containing the points $a_0$, $b_0$, $c_0$, $d_\infty $, and small enough that the following property holds. The box $Q_\infty$ is contained in  $Q_\infty'= [a_0', b_0'] \times [c_0', d_\infty']$ and contains $Q_\infty''= [a_0'', b_0''] \times [c_0'', d_\infty'']$. By Lemma~\ref{lem:UniformlyQuasisym}, the maps $\widetilde f_n  \colon \widetilde X_0 \to \widetilde X_n$ are uniformly quasisymmetric. Therefore, noting that the closure of $Q_\infty' - Q_\infty''$ is  the union of the four boxes $[a_0', b_0'] \times [c_0', c_0'']$, $[a_0', b_0'] \times [d_\infty'', d_\infty']$, $[a_0', a_0''] \times [c_0', d_\infty']$ and $[b_0'', b_0'] \times [c_0', d_\infty']$, we can use Proposition~\ref{prop:QuasiSymmetricEta} to choose the intervals $\left]a_0', a_0''\right[ $, $\left]b_0'', b_0'\right[ $, $\left]c_0', c_0''\right[ $ and $\left]d_\infty'', d_\infty '\right[ $  small enough that
\begin{equation}
\begin{split}
\label{eqn:QuasiSymTendsTo1:1}
  L_{f_n} \big( \phi( Q_\infty' - Q_\infty'' )\big) &< \epsilon
  \\
  \text{and }  L_{f_\infty} \big( \phi( Q_\infty' - Q_\infty'' )\big) &< \epsilon
\end{split}
\end{equation}
for every $n$ and every $\phi \in \HHH(\widetilde X_0)$.

By construction, $Q_\infty$ is contained in the interior of $Q_\infty'$, and contains $Q_\infty''$ in its interior. Let  $\xi \colon G(\widetilde X_0) \to [0,1]$ be a continuous  test function that is identically 1 on the box $Q_\infty''$ and 0 outside of $Q_\infty'$.  For $k$ large enough, the box $\phi_{n_k}(Q_{n_k}) $ is  very close to $Q_\infty$ and therefore $Q_{\infty}'' \subset \phi_{n_k}(Q_{n_k})  \subset Q_\infty'$. As a consequence, $ \chi_{\phi_{n_k}^{-1}(Q_\infty'')} \leq \xi\circ \phi_{n_k} \leq  \chi_{\phi_{n_k}^{-1}(Q_\infty')} $ and $ \chi_{\phi_{n_k}^{-1}(Q_\infty'')} \leq  \chi_{Q_{n_k}} \leq  \chi_{\phi_{n_k}^{-1}(Q_\infty')} $  if $\chi_A \colon G(\widetilde X_0) \to \{0,1\}$ denotes the characteristic function of the subset $A\subset G(\widetilde X_0) $. It follows that for $k$ sufficiently large
\begin{align*}
 \Big|    \int_{G(\widetilde X_0)} \xi\circ \phi_{n_k} \, dL_{f_{n_k}} 
 - L_{f_{n_k}}( Q_{n_k} ) \Big|
  \leq L_{f_{n_k}}\big( \phi_{n_k} ^{-1}( Q_\infty' - Q_\infty'' ) \big) 
 <\epsilon
\end{align*}
by (\ref{eqn:QuasiSymTendsTo1:1}), and
\begin{equation}
\begin{split}
\label{eqn:QuasiSymTendsTo1:2}
 \int_{G(\widetilde X_0)} \xi\circ \phi_{n_k} \, dL_{f_{n_k}}
  &> L_{f_{n_k}}( Q_{n_k} ) -\epsilon \\
  &> \log2 +\epsilon_0\log 2 -\epsilon
\end{split}
\end{equation}
since the boxes $Q_{n_k}$ were chosen so that $L_{f_n}(Q_{n_k}) > (1+\epsilon_0)\log 2$.

Similarly, 
\begin{align*}
 \Big|    \int_{G(\widetilde X_0)} \xi\circ \phi_{n_k} \, dL_{f_\infty} 
 - L_{f_\infty}( Q_{n_k} ) \Big|
  \leq L_{f_\infty}\big( \phi_{n_k} ^{-1}( Q_\infty' - Q_\infty'' ) \big) 
 <\epsilon
\end{align*}
and
\begin{equation}
\begin{split}
\label{eqn:QuasiSymTendsTo1:3}
 \int_{G(\widetilde X_0)} \xi\circ \phi_{n_k} \, dL_{f_\infty}
  &< L_{f_{\infty}}( Q_{n_k} ) +\epsilon \\
  &< \log2 + \epsilon
\end{split}
\end{equation}
since $L_{f_\infty}(Q_{n_k}) = \log 2$. 

But, if we had chosen $\epsilon>0$ small enough that $2\epsilon< \epsilon_0\log 2$, the inequalities (\ref{eqn:QuasiSymTendsTo1:2}) and (\ref{eqn:QuasiSymTendsTo1:3}) are incompatible with the fact that 
$$
 \int_{G(\widetilde X_0)} \xi\circ \phi_{n_k} \, dL_{f_{n_k}} 
 \to 
  \int_{G(\widetilde X_0)} \xi\circ \phi_{n_k} \, dL_{f_\infty} 
  \text{ as } k \to \infty
  $$
  by uniform weak* convergence of $L_{f_{n_k}}$ to $L_{f_\infty}$. This contradiction proves Lemma~\ref{lem:QuasisymTendsTo1}. 
\end{proof}

By the property of Lemma~\ref{lem:QuasisymTendsTo1}, Theorem~\ref{thm:BeurlingAhlforsImproved} then shows that $[f_n] \in \T(X_0)$ converges to $[f_\infty]$ for the Teichm\"uller metric. This completes the proof of Proposition~\ref{prop:InverseLiouvilleContinuous}. 
\end{proof}

\subsection{The image $ \LL \bigl( \T(X_0) \bigr)$ of the Liouville embedding is closed}

\begin{prop}
\label{prop:LiouvilleClosed}
The image $ \LL \bigl( \T(X_0) \bigr)$ of the Liouville embedding $\LL \colon \T(X_0) \to \CC(X_0)$ is closed in the space $\CC(X_0)$ of bounded geodesic currents. 
\end{prop}

\begin{proof} As before, the metrizability property of Lemma~\ref{lem:GeodCurrentSpaceMetrizable} enables us to argue in terms of sequences. 
Let $[f_n] \in \T(X_0)$ be a sequence in the Teichm\"uller space such that the associated Liouville geodesic currents $\LL\big([f_n] \big)=L_{f_n}$ converge to some geodesic current $\alpha_\infty \in \CC(X_0)$. We want to show that $\alpha_\infty$ is also in the image $ \LL \bigl( \T(X_0) \bigr)$. 

As usual, lift the quasiconformal diffeomorphisms $f_n \colon X_0 \to X_n$ to maps $\widetilde f_n \colon \widetilde X_0 \to \widetilde X_n$ between universal covers, and consider the quasisymmetric extension $\widetilde f_n \colon \partial_\infty \widetilde X_0 \to  \partial_\infty \widetilde X_n$. Because the Liouville geodesic currents $L_{f_n}$ converge to $\alpha_\infty$ for the uniform weak* topology and because the limit $\alpha_\infty$ is  bounded, the argument that we already used in the proof of Lemma~\ref{lem:UniformlyQuasisym} shows that the quasisymmetric constants $M(\widetilde f_n)$ are uniformly bounded.

Fix three points $a_0$, $b_0$, $c_0$ in this order in the circle at infinity $\partial_\infty \widetilde X_0$. Then, there is a unique biholomorphic map $\widetilde g_n \colon \widetilde X_n \to \D$ sending $\widetilde f_n( a_0)$ to $1$, $\widetilde f_n( b_0)$ to $\I$ and $\widetilde f_n( c_0)$ to $-1$. The maps $\widetilde g_n \circ \widetilde f_n \colon\partial_\infty  \widetilde X_0 \to\partial \D$ are uniformly quasisymmetric, and send the three points $a_0$, $b_0$, $c_0$ to the fixed points $1$, $\I$, $-1$. It easily follows that these maps $\widetilde g_n \circ \widetilde f_n$ are equicontinuous, so that we can extract a subsequence $\widetilde g_{n_k} \circ \widetilde f_{n_k}$ that converges to a homeomorphism $\widetilde f_\infty \colon \partial_\infty \widetilde X_0 \to \partial \D$ for the topology of uniform convergence (see for instance \cite[\S II.5]{LV} or \cite[\S16]{GL}).

By uniform quasisymmetry of the $\widetilde f_n$, the limit $\widetilde f_\infty$ is quasisymmetric. Also, if $\phi \colon \widetilde X_0 \to \widetilde X_0$ is the biholomorphic map of $\widetilde X_0$ defined by an element $\phi \in \pi_1(X_0)$ of the fundamental group, $ \widetilde f_\infty \circ \phi \circ\widetilde f_\infty^{-1} = \lim_{k\to \infty}  \widetilde f_{n_k} \circ \phi \circ\widetilde f_{n_k}^{-1} $ is a linear fractional map that is the restriction to $\partial\D$  of a biholomorphic map of $\D$. As $\phi$ ranges over all elements of $\pi_1(X_0)$, these $ \widetilde f_\infty \circ \phi \circ\widetilde f_\infty^{-1} $ define a discrete biholomorphic action of $\pi_1(X_0)$ on $\D$, and we can consider the Riemann surface $X_\infty = \D/\pi_1(X_0)$. 

The Douady-Earle Extension Theorem \cite{DE} (see also our proof of Theorem~\ref{thm:BeurlingAhlforsImproved}) then provides a quasiconformal extension $\widetilde f_\infty \colon \widetilde X_0 \to \D$ of $f_\infty \colon \partial_\infty \widetilde X_0 \to \partial \D$ that commutes with the actions of $\pi_1(X_0)$ on $\widetilde X_0$ and $\D$, and therefore descends to a quasiconformal map $f_\infty \colon X_0 \to X_\infty =  \D/\pi_1(X_0)$.

The uniform convergence of $\widetilde g_{n_k} \circ \widetilde f_{n_k}$ to $\widetilde f_\infty$ as $k \to \infty$ does not imply that $[f_{n_k}]\in \T(X_0)$ necessarily converges to $[f_\infty]$ for the Teichm\"uller topology. However, it is enough to guarantee that the pullback $L_{f_\infty}$ of the Liouville measure $L_\D$ by $\widetilde f_\infty$ is the weak* limit of the pullback of $L_\D$ by $\widetilde g_{n_k} \circ \widetilde f_{n_k}$, which also is the pullback $L_{f_{n_k}}$ of $L_{\widetilde X_{n_k}}$ by $\widetilde f_{n_k}$. Therefore $\alpha_\infty \in \CC(X_0)$, which was defined as the uniform weak* limit of the Liouville geodesic currents $L_{f_n}$, is equal to $L_{f_\infty}= \LL\big( [f_\infty]\big)$. In particular, $\alpha_\infty$ is in the image of $\LL$, as requested. 
\end{proof}

 \subsection{The Liouville embedding is proper}
 
\begin{prop}
\label{prop:LiouvilleProper}
 The Liouville embedding $\LL \colon \T(X_0) \to \CC(X_0)$ is proper. 
\end{prop}

\begin{proof}
Recall that a map is \emph{proper} if the preimage of a bounded set is bounded. 
We therefore need to prove the following property: Let $B $ be a subset of $\T(X_0)$ such that
$$
\sup_{[f] \in B} \sup_{\phi \in \HHH(\widetilde X_0)} \Big\vert \int_{G(\widetilde X_0)} \xi \circ \phi \, dL_f \Big\vert \leq C(\xi)
$$
 for every continuous function $\xi \colon G(\widetilde X_0) \to \R$ with compact support and for some constant $C(\xi)$ depending on $\xi$; then $B$ is bounded for the Teichm\"uller metric of $\T(X_0)$. 
 
For such a subset $B$, choose a symmetric box $Q_0 \subset  G(\widetilde X_0) $ and a non-negative function $\xi \colon G(\widetilde X_0) \to \R$ with compact support such that $\xi\geq 1$ over the box $Q_0$. Then, as in the proof of Lemma~\ref{lem:UniformlyQuasisym},  $L_f(Q) \leq C(\xi)$ for every symmetric box $Q$ and every $[f]\in B$, and the quasisymmetric constants $M(f)$ are uniformly bounded over $B$. By Theorem~\ref{thm:BeurlingAhlforsImproved}, this proves that $B$ is bounded by the Teichm\"uller metric. 
\end{proof}

The combination of Propositions~\ref{prop:LiouvilleContinuous}, \ref{prop:InverseLiouvilleContinuous}, \ref{prop:LiouvilleClosed} and \ref{prop:LiouvilleProper} proves Theorem~\ref{thm:LiouvilleEmbedding}, namely that the Liouville embedding $\LL \colon \T(X_0) \to \CC(X_0)$ is proper and induces a homeomorphism between $\T(X_0)$ and a closed subset of $\CC(X_0)$.

We are going to need a slightly stronger version of this result.

\subsection{The projectivization of the Liouville embedding}
\label{sect:ProjectiveLiouville}
 The group $\R^+$ of positive real numbers acts by multiplication on the space $\CC(X_0)$ of bounded geodesic currents. Let $\PCC(X_0)=\big( \CC(X_0)-\{0\} \big)/\R^+$ be the  quotient of $\CC(X_0)-\{0\}$ under this action. We endow the  space $\PCC(X_0)$  with the quotient of the uniform weak* topology of $\CC(X_0)$. 

 The elements of $\PCC(X_0)$ are \emph{projective bounded geodesic currents} in the Riemann surface~$X_0$.

Composing the Liouville embedding $\LL \colon \T(X_0) \to \CC(X_0)$ with the projection $\CC(X_0) \to \PCC(X_0)$ gives a continuous map  $\PL \colon \T(X_0) \to \PCC(X_0)$, which we call the \emph{projective Liouville embedding}. 
The following result shows that this projective Liouville embedding is really an embedding.

\begin{thm}
 \label{thm:ProjectiveLiouville}
 The map $\PL \colon \T(X_0) \to \PCC(X_0)$ induces a homeomorphism between the Teichm\"uller space $\T(X_0)$ and a subset of the space $\PCC(X_0)$ of projective bounded geodesic currents. 
\end{thm}
\begin{proof} The map  $\PL \colon \T(X_0) \to \PCC(X_0)$ is injective. Indeed, if $\PL\big( [f_1] \big) =  \PL\big( [f_2] \big)$ in $\PCC(X_0)$, the Liouville current $\LL\big( [f_2] \big) =L_{f_2}$ is equal to  $t \LL\big( [f_1] \big) =t  L_{f_1}$ in $\CC(X_0)$ for some number $t>0$. The property of Lemma~\ref{lem:LiouvilleMassOrthoBox}, that
 $$
 \E^{-L_f(Q)} + \E^{-L_f(Q^\perp)} =1
 $$
for every $[f] \in \T(X_0)$ and every box $Q\subset G(\widetilde X_0)$ with orthogonal box $Q^\perp$, then shows that necessarily $t =1$. The injectivity of $\PL \colon \T(X_0) \to \PCC(X_0)$ then follows from  the  injectivity of  the Liouville embedding $\LL \colon \T(X_0) \to \CC(X_0)$ (Proposition~\ref{prop:LiouvilleInjective}).

 The projective Liouville embedding $\PL$ was defined as the composition of two continuous maps, and is consequently continuous. Therefore, we only have to show that its inverse $\PL^{-1} \colon \PL \big ( \T(X_0) \big) \to \T(X_0)$ is continuous. 
 
 For this, consider a sequence of points $[f_n] \in \T(X_0)$ such that  $\lim_{n\to \infty}\PL\big( [f_n] \big)= \PL\big( [f_\infty] \big)$ in $\PCC(X_0)$ for some $[f_\infty] \in \T(X_0)$. We want to show that $\lim_{n\to \infty} [f_n]=[f_\infty]$ in $\T(X_0)$. 
 
 By definition of the quotient topology, the property that $\lim_{n\to \infty}\PL\big( [f_n] \big)= \PL\big( [f_\infty] \big)$ means that there exists a sequence  $r_n \in \R^+$ such that $\frac1{r_n} \LL\big( [f_n] \big)= \frac1{r_n} L_{f_n}$ converges to  $ \LL\big( [f_\infty] \big)=L_{f_\infty}$ in $\CC(X_0)$, for the uniform weak* topology. In particular, $\frac1{r_n}L_{f_n}$ converges to  $ L_{f_\infty}$ for the (non uniform) weak* topology and, by Lemma~\ref{lem:Weak*ConvImpliesConvMasses}, it follows that $\frac1{r_n}L_{f_n}(Q)$ converges to  $L_{f_\infty}(Q)$ for every box $Q \subset G(\widetilde X_0)$. Another application of Lemma~\ref{lem:LiouvilleMassOrthoBox}, to a single box $Q$,  then shows that necessarily $\lim_{n\to \infty} r_n =1$. 
 
 As a consequence, $\lim_{n\to \infty}\LL\big( [f_n] \big)= \LL\big( [f_\infty] \big)$ in $\CC(X_0)$. Since the inverse map $\LL^{-1} \colon \LL \bigl( \T(X_0) \bigr) \to \T(X_0)$ is continuous by Proposition~\ref{prop:InverseLiouvilleContinuous}, if follows that $\lim_{n\to \infty} [f_n]=[f_\infty]$ in $\T(X_0)$ as required. 
\end{proof}

\section{A boundary for the Teichm\"uller space}

\subsection{Measured geodesic laminations} 

 A \emph{measured geodesic lamination} in the Riemann surface $X_0$ is a  geodesic current $\alpha\in  \mathcal C(X_0)$ such that:
 \begin{enumerate}
\item $\alpha$ is balanced, in the sense that it is invariant under the involution $\tau \colon G(\widetilde X_0) \to G(\widetilde X_0)$ that reverses the orientation of each geodesic $g \in G(\widetilde X_0)$;

\item any two distinct geodesics $g$, $g'$ of the support $\mathrm{Supp}(\alpha)\subset G(\widetilde X_0)$ are disjoint in $\widetilde X_0$, unless $g'=\tau(g)$;
\end{enumerate}

By equivariance of $\alpha$, its support is invariant under the action of $\pi_1(X_0)$ and therefore descends to a \emph{geodesic lamination} $\lambda_\alpha$ in $X_0$, namely to a family   of disjoint simple complete geodesics (for the Poincar\'e metric of $X_0$) whose union forms a closed subset of $X_0$. Recall that a geodesic is \emph{complete} if it cannot be extended to a longer geodesic, and that it is \emph{simple} if it does not transversely intersect itself. 

Beware that, in contrast to the classical case where $X_0$ is compact, the union of the geodesics of the geodesic lamination $\lambda_\alpha$ can have nonempty interior in $X_0$, and that this subset can have several decompositions as a union of pairwise disjoint complete geodesics. 

A measured geodesic lamination is \emph{bounded} if it is  bounded as a geodesic current, as defined in \S \ref{sect:GeodesicCurr}. Let $\ML(X_0) \subset \CC(X_0)$ denote the space of bounded measured geodesic laminations in the Riemann surface $X_0$.

\subsection{The Thurston boundary of $\T(X_0)$} 
\label{subsect:ThurstonBdry}

As in \S \ref{sect:ProjectiveLiouville}, consider the projective Liouville embedding $\PL \colon \T(X_0) \to \PCC(X_0)$ from the Teichm\"uller space $\T(X_0)$ to the space $\PCC(X_0)$  of projective bounded geodesic currents. We saw in Theorem~\ref{thm:ProjectiveLiouville} that $\PL$ induces a homeomorphism from $\T(X_0)$ to its image $\PL \big( \T(X_0) \big) \subset \PCC(X_0)$. 

By analogy with the case where $X_0$ is compact, we define the \emph{Thurston boundary} of $\T(X_0)$ as the boundary of this embedding, namely as the set of points of $\PCC(X_0)$ that are in the closure of $\PL \big( \T(X_0) \big)$ but are not contained in $\PL \big( \T(X_0) \big)$. 

Our next goal is to describe this closure. Note that the space $\ML(X_0)$ of bounded measured geodesic laminations is invariant under the action of $\R^+$ on $\CC(X_0)$. It therefore makes sense to consider its image $\PML(X_0) =\big( \ML(X_0)-\{0\} \big) /\R^+$ in $\PCC(X_0)$. By definition, the points of $\PML(X_0)$ are \emph{projective bounded measured geodesic laminations} in $X_0$.

\begin{prop}
\label{prop:ThurstonBdryConsistsOfMeasuredLam}
 The Thurston boundary of the Teichm\"uller space $\T(X_0)$ is contained in the space $\PML(X_0)$ of projective bounded measured geodesic laminations.
\end{prop}
\begin{proof} Let $\alpha \in \CC(X_0)$ be a bounded geodesic current whose image $\langle \alpha \rangle \in \PCC(X_0)$ is in the Thurston boundary.  In particular, $\langle \alpha \rangle$ is in the closure of $\PL \big( \T(X_0) \big)$, and there exists a sequence $[f_n ] \in \T(X_0)$ and numbers $t_n>0$ such that 
$$
\alpha = \lim_{n\to \infty} \frac1{t_n} \LL\big( [f_n] \big) = \lim_{n\to \infty} \frac1{t_n} L_{f_n}.
$$

We claim that $t_n \to \infty$ as $n\to \infty$. Indeed, we would otherwise find a subsequence $t_{n_k}$ converging to some $t_\infty\geq 0$ as $k \to \infty$. Then, $t_\infty \alpha = \lim_{k \to \infty} L_{f_{n_k}}$ would belong to $\LL\big( \T(X_0) \big)$ since this image is closed by Theorem~\ref{thm:LiouvilleEmbedding}. Note that $t_\infty $ cannot be equal to $ 0$, as otherwise $\LL\big( \T(X_0) \big)$ would contain the trivial geodesic current $0\in \CC(X_0)$ while Liouville currents clearly are never trivial. But it cannot be different from 0 either, as this would otherwise contradict  the fact that $\langle \alpha \rangle$ is not allowed to belong to $\PL \big( \T(X_0) \big)$, by definition of the Thurston boundary. 

Now suppose, in search of a contradiction, that $\alpha$ is not a measured geodesic lamination. This means that the support of $\alpha$ contains two geodesics $g$, $g' \in G(\widetilde X_0)$ that cross each other in $\widetilde X_0$. We can then find  a box $Q \subset G(\widetilde X_0)$ containing $g$ in its interior such that the orthogonal box $Q^\perp$ contains $g'$ in its interior (possibly after reversing the orientation of $g'$). In particular, $\alpha(Q)>0$ and $\alpha(Q^\perp)>0$. In addition, by countable additivity of $\alpha$, we can choose the points of $\partial_\infty \widetilde X_0$ delimiting $Q$ so that $\alpha(\partial Q) = \alpha (\partial Q^\perp) =0$. 
Then, by weak* convergence (see Lemma~\ref{lem:Weak*ConvImpliesConvMasses}),   
$$\alpha(Q) = \lim_{n\to \infty}\frac1{t_n} L_{f_n}(Q) \text{ and } \alpha(Q^\perp) = \lim_{n\to \infty}\frac1{t_n} L_{f_n}(Q^\perp),$$ so that $$\lim_{n\to \infty} L_{f_n}(Q)=\lim_{n\to \infty} L_{f_n}(Q^\perp)=\infty$$ 
since we established that $t_n \to \infty$ as $n\to \infty$.  But this contradicts Lemma~\ref{lem:LiouvilleMassOrthoBox}, and the fact that $ \E^{-L_{f_n}(Q)} + \E^{-L_{f_n}(Q^\perp)} =1$. 

Therefore, the support of $\alpha$ is a geodesic lamination, and $\langle \alpha \rangle$ belongs to the space  $\PML(X_0)$ of projective bounded measured geodesic laminations.
\end{proof}

We prove the converse  of Proposition~\ref{prop:ThurstonBdryConsistsOfMeasuredLam} as Corollary~\ref{cor:MeasuredLamLimitLiouville} in the next section. The combination of these two statements will show:
\begin{thm}
\label{thm:ThurstonBdryMeasureLam}
  The Thurston boundary of the Teichm\"uller space $\T(X_0)$ is exactly equal to the space $\PML(X_0)$ of projective bounded measured geodesic laminations. \qed
\end{thm}

\section{Earthquakes}
We will use earthquakes as a tool to show that every projective bounded measured geodesic lamination is contained  in the Thurston boundary of $\T(X_0)$.  The key technical step is Theorem~\ref{thm:LimitEarthquakes} below, which is of independent interest.

\subsection{Earthquakes} Let $\lambda$ be a geodesic lamination in the Riemann surface $X_0$, namely a family of disjoint simple complete geodesics in $X_0$ whose union is closed in $X_0$. Let $\widetilde \lambda \subset G(\widetilde X_0)$ consist of those geodesics which project to one of the geodesics of $\lambda$. In particular, $\widetilde\lambda$ is invariant under the involution $\tau \colon G(\widetilde X_0) \to G(\widetilde X_0)$ that acts by reversing the orientation of each geodesic. A simple argument also shows that $\widetilde \lambda$ is closed in $G(\widetilde X_0)$. 

If $[f]$, $[f'] \in \T(X_0)$ are two points of the Teichm\"uller space of $X_0$, we say that $[f']$ is obtained from $[f]$ by a \emph{left earthquake along $\lambda$} if 
$$
L_f(Q) \leq L_{f'}(Q)
$$
for every box of geodesics $Q=[a,b] \times [c,d] \subset G(\widetilde X_0)$ such that $\{a,c\} \in \partial_\infty \widetilde S$ are the endpoints of one of the geodesics of $\widetilde \lambda$.

Thurston \cite{Thu1} shows how to quantify the increase in Liouville masses by a measure  on the closed subset $\widetilde \lambda \subset G(\widetilde X_0)$, namely by a measure $\alpha$ on $G(\widetilde X_0)$ whose support is contained in $\widetilde \lambda$. In addition, $\alpha$  is invariant under the action of the fundamental group $\pi_1(X_0)$, and consequently is a measured geodesic lamination. A  subtler consequence of the fact that $f$ is quasiconformal is that $\alpha$ is bounded; see  \cite{Thu1, Sar1, Sar2, GHL, EMM}. 

Thurston also introduced an inverse construction \cite{Thu1, EpM} which, given a point $[f] \in \T(X_0)$ and a bounded measured geodesic lamination $\alpha \in \ML(X_0)$, produces another element $[f'] \in \T(X_0)$ that is obtained from $[f]$ by a left earthquake along the support $\lambda_\alpha$ of $\alpha$, with amplitude determined by the measure $\alpha$. We then write that $[f'] = E^\alpha [f]$. 

Finally, Thurston shows \cite{Thu1}  that for any two $[f]$, $[f']  \in \T(X_0)$ there exists a unique $\alpha \in \ML(X_0)$ such that $[f'] = E^\alpha [f]$. See also \cite{Ke}. 

\begin{rem}
\label{rem:BoundedLamAndQuasiconfEarthqke}
 We  should  emphasize the close relationship between the boundedness property for measured geodesic laminations and the quasiconformal geometry underlying the Teichm\"uller space. Thurston's construction  \cite{Thu1} makes sense in the broader context of diffeomorphisms $f\colon X_0 \to X$ whose lift to universal covers continuously extends to a homeomorphism $\partial_\infty \widetilde X_0 \to \partial_\infty \widetilde X$. These are  not necessarily quasiconformal, so that they do not necessarily define an element $[f] \in \T(X_0)$, but the equivalence relation defining the Teichm\"uller space makes sense in this more general context. Thurston shows that any two such $f\colon X_0 \to X$ and $f'\colon X_0 \to X'$ are related by an earthquake, namely that $[f']= E^\alpha [f]$ for some measured geodesic lamination $\alpha$ which is not necessarily bounded. However, when $X_0$ is noncompact, there is no easy characterization of which measured geodesic laminations $\alpha \in \mathcal{ML}(X_0)$ occur in this way. The results mentioned above show that, when $f$ is quasiconformal, $E^\alpha[f]$ is well-defined and realized by a quasiconformal diffeomorphism $f'$ precisely when $\alpha$ is bounded.
 
 This distinction is of course irrelevant when $X_0$ is compact, as every diffeomorphism $f\colon X_0 \to X$ is then quasiconformal,  and every measured geodesic lamination is bounded by Proposition~\ref{prop:UniformNonUniformWhenCompact}. 
\end{rem}

For a bounded measured geodesic lamination $\alpha \in \ML(X_0)$ and a number  $t>0$, let $t\alpha$ be the bounded measured geodesic lamination obtained by multiplying the measure $\alpha$ by $t$. The following theorem investigates the behavior of $E^{t\alpha} [f] \in \T(X_0)$ under the Liouville embedding $\LL \colon \T(X_0) \to \CC(X_0)$. 

\begin{thm}
\label{thm:LimitEarthquakes}
Let $\alpha \in \ML(X_0)$ be a bounded measured geodesic lamination in the Riemann surface $X_0$. Then, for every $[f] \in \T(X_0)$,
$$
\lim_{t\to \infty} \frac1t \LL \big( E^{t\alpha} [f] \big) = \alpha
$$
for the uniform weak* topology on the space $\CC(X_0)$ of geodesic currents. 
\end{thm}

The proof of Theorem~\ref{thm:LimitEarthquakes} will occupy the rest of this section. However, it has the following immediate corollary, which completes the proof of Theorem~\ref{thm:ThurstonBdryMeasureLam}. 

\begin{cor}
\label{cor:MeasuredLamLimitLiouville}
The space  $ \PML(X_0)$ of projective bounded measured geodesic laminations is contained in the Thurston boundary of the Teichm\"uller space $\T(X_0)$. 
\end{cor}

\begin{proof}
 Theorem~\ref{thm:LimitEarthquakes} shows that every projective bounded measured geodesic lamination $\langle \alpha\rangle \in \PML(X_0)$ is in the closure of the image of the projective Liouville embedding $\PL \colon \T(X_0) \to \PCC(X_0)$. A Liouville geodesic current has full support in $G(\widetilde X_0)$, and a measured geodesic lamination cannot have full support. It follows that  $\langle \alpha\rangle  \in \PML(X_0)$  does not belong to the image $\PL \big( \T(X_0) \big)$, and therefore is in the Thurston boundary of $\T(X_0)$ by definition of this boundary. 
\end{proof}

\subsection{Elementary earthquakes}
\label{subsect:ElemEartqke}

 The construction of the earthquake deformations $E^\alpha [f]$ is based on the following special case. 

Let $\widetilde X_0$ be a simply connected conformally hyperbolic Riemann surface. (We are using a tilde in the notation to remind the reader that the surface is simply connected, and therefore equal to its universal cover.) In particular, $\widetilde X_0$ is biholomorphically equivalent to the unit disk~$\D$.  

 For a geodesic  $g\in G(\widetilde X_0)$ and a number $t\in \R$, the \emph{elementary earthquake of amplitude $t$ along $g$} is the homeomorphism  $E_g^t\colon   \T(\widetilde X_0) \to \T(\widetilde X_0)$ defined as follows. 
 
 Let $[f] \in \T(\widetilde X_0)$ be a point in the Teichm\"uller space of $\widetilde X_0$, represented by a quasiconformal diffeomorphism $f\colon \widetilde X_0 \to \widetilde X_1$. If $g_1$ is the geodesic of $\widetilde X_1$ that is the image of $g$ under the map $f\colon G(\widetilde X_0) \to G(\widetilde X_1)$ induced by $f$, and let $\phi_t \colon \widetilde X_1 \to \widetilde X_1$ be the hyperbolic isometry that preserves $g_1$ and acts by translation of $t\in \R$ along $g_1$ for the orientation of $g_1$. Then $E_g^t[f] \in \T(\widetilde X_0)$ is represented by any quasiconformal extension of the quasisymmetric homeomorphism $E^t_g f \colon \partial_\infty \widetilde X_0 \to \partial_\infty \widetilde X_1$ that coincides with $f$ on the component of $\partial_\infty \widetilde X_0 - \partial g$ that sits to the left of $g$, and with $\phi_t \circ f$ on the other component of $\partial_\infty \widetilde X_0 - \partial g$. Equivalently, $E_g^t[f]$ is represented by the quasisymmetric homeomorphism $\phi_t^{-1} \circ E^t_g f \colon \partial_\infty \widetilde X_0 \to \partial_\infty \widetilde X_1$ that coincides with $\phi_t^{-1} \circ f$ on the component of $\partial_\infty \widetilde X_0 - \partial g$ that sits to the left of $g$, and with $f$ on the other component of $\partial_\infty \widetilde X_0 - \partial g$. 
 
From the fact that $\phi_t$ is an isometry of $\widetilde X_1$, it easily follows that reversing the orientation of the geodesic $g$ does not change $E_g^t[f] \in \T(\widetilde X_0)$.

General earthquakes $E^\alpha \colon \T(\widetilde X_0) \to \T(\widetilde X_0)$ are constructed from elementary earthquakes as follows. 

First consider the case where $\delta \in \ML(\widetilde X_0)$ is a Dirac measure with finite support $\{ g_1, g_2, \dots, g_k, \bar g_1, \bar g_2, \dots, \bar g_k \} \subset G(\widetilde X_0)$, where $\bar g_i = \tau (g_i)$ is  obtained by reversing the orientation of the geodesic $g_i \in  G(\widetilde X_0)$. Then, $E^\delta$ is defined as
$$
E^\delta = E_{g_1}^{d_1} \circ E_{g_2}^{d_2}\circ  \dots\circ  E_{g_k}^{d_k}
$$
where $d_i = \delta\big( \{g_i\} \big) = \delta\big( \{\bar g_i\} \big) $. Note that the elementary earthquakes $E_{g_i}^{d_i} \colon \T(\widetilde X_0) \to \T(\widetilde X_0) $ commute because the geodesics $g_i$ are disjoint. 

In the general case, we approximate the measured geodesic lamination $\alpha \in \ML(\widetilde X_0)$ by Dirac measures $\delta$ as above, and  define
$$
E^\alpha [f] = \lim_{\delta \to \alpha} E^\delta[f]
$$
for every $[f]\in \T(\widetilde X_0)$, where the limit is taken as the Dirac measure $\delta$ tends to $\alpha$ for the weak* topology. The boundedness of $\alpha$ is used to show that the limit really exists. See \cite{Thu1, EpM, Sar1} for details. 

When $\widetilde X_0$ is the universal cover of a conformally hyperbolic Riemann surface $X_0$ and when $\alpha \in \ML(X_0) \subset \ML(\widetilde X_0)$, the above construction is equivariant with respect to the action of $\pi_1(X_0)$ on $ \T(\widetilde X_0)$, and the earthquake $E^\alpha \colon \T(\widetilde X_0) \to \T(\widetilde X_0)$  therefore descends to a continuous map $E^\alpha \colon \T( X_0) \to \T( X_0)$.

\subsection{Two lemmas on elementary earthquakes}

We will make frequent use of the following two lemmas.

\begin{lem}
\label{lem:ElementaryEarthquake}
Let $Q = [a,b]\times [c,d]$ be a box of geodesics in $G(\widetilde X_0)$, and let  $g\in G(\widetilde X_0)$ be a geodesic with endpoints $x$, $y\in \partial_\infty \widetilde X_0 - \{a,b,c,d\}$. Consider the image $E_{g}^t [f]$ of $[f]\in \T(\widetilde X_0)$ under the elementary earthquake of amplitude $t>0$ along $g$. 

\begin{enumerate}
 \item[(0)] If $x$ and $y$ are in the same component of $\partial_\infty \widetilde X_0 - \{a,b,c,d\}$, then $L _{E^t_{g}[f]} (Q) = L _{[f]} (Q)$ is independent of $x$ and $y$.
 \item[(a)]  It $x\in \left] a,b \right[$ and $y \in \left] c,d\right[$ as in Figure~{\upshape \ref{fig:ElementaryEarthquakes}(a)}, $L _{E^t_{g}[f]} (Q) $ is a decreasing function of $x$ and $y$ for the boundary orientation of $\partial_\infty \widetilde X_0$.
  \item [(b)] It $x\in \left] b,c \right[$ and $y \in \left] d,a \right[$ as in Figure~{\upshape \ref{fig:ElementaryEarthquakes}(b)}, $L _{E^t_{g}[f]} (Q) $ is an increasing function of $x$ and  $y$.
% \item [(c)] It $x\in \left] a,b \right[$ and $y \in \left] b,c \right[$ as in Figure~{\upshape \ref{fig:ElementaryEarthquakes}(c)}, $L _{E^t_{g}[f]} (Q) $ is a decreasing function of $x$ and an increasing function of  $y$.
 %\item [(d)] It $x\in \left] a,b \right[$ and $y \in \left] d ,a \right[$ as in Figure~{\upshape \ref{fig:ElementaryEarthquakes}(d)}, $L _{E^t_{g}[f]} (Q) $is an increasing function of $x$ and a decreasing function of  $y$.
\end{enumerate}
\end{lem}

The statement is expressed in a more pictorial way by Figure~\ref{fig:ElementaryEarthquakes}. 
%\begin{figure}[htbp]

%\SetLabels
%(  .105* -.3 ) (a) \\
%(  .375* -.3 ) (b) \\
%(  .63* -.3 ) (c) \\
%(  .905* -.3 ) (d) \\
%( 0.02 *  0.02)  $a$ \\
%( 0.2 * 0.02 ) $b$  \\
%( 0.2 * .9 ) $c$  \\
%( 0.02 * .9 ) $d$  \\
%( 0.28 *  0.02)  $a$ \\
%( 0.46 * 0.02 ) $b$  \\
%( 0.46 * .9 ) $c$  \\
%( 0.28 * .9 ) $d$  \\
%( 0.54 *  0.02)  $a$ \\
%( 0.72 * 0.02 ) $b$  \\
%( 0.72 * .9 ) $c$  \\
%( 0.54 * .9 ) $d$  \\
%( 0.805 *  0.02)  $a$ \\
%( 0.985 * 0.02 ) $b$  \\
%( 0.985 * .9 ) $c$  \\
%( 0.805 * .9 ) $d$  \\
%(0.16 * -0.05 ) $x$  \\
%(0.48 *  .74) $x$  \\
%(.68 * -0.07 ) $x$  \\
%(.83 * -0.04 ) $x$  \\
%(0.12 *  1.07) $y$  \\
%( .245*  .56) $y$  \\
%( .75* .35 ) $y$  \\
%(.772 *  .62) $y$  \\
%(0.14 *  .5) $g$  \\
%(0.36 *  .47) $g$  \\
%(0.67 *  .35) $g$  \\
%(0.865 *  .45) $g$  \\
%\endSetLabels
%\centerline{\AffixLabels{\includegraphics{ElementaryEarthquakes.eps}}}

\begin{figure}[htbp]

\SetLabels
(  .16* -.25 ) (a) \\
(  .83* -.25 ) (b) \\
%(  .63* -.3 ) (c) \\
%(  .905* -.3 ) (d) \\
( 0.03 *  0.06)  $a$ \\
( 0.32 * 0.06 ) $b$  \\
( 0.31 * .88 ) $c$  \\
( 0.03 * .88 ) $d$  \\
( 0.69 *  0.06)  $a$ \\
( 0.98 * 0.06 ) $b$  \\
( 0.97 * .88 ) $c$  \\
( 0.69 * .88 ) $d$  \\
%( 0.54 *  0.02)  $a$ \\
%( 0.72 * 0.02 ) $b$  \\
%( 0.72 * .9 ) $c$  \\
%( 0.54 * .9 ) $d$  \\
%( 0.805 *  0.02)  $a$ \\?
%( 0.985 * 0.02 ) $b$  \\
%( 0.985 * .9 ) $c$  \\
%( 0.805 * .9 ) $d$  \\
(0.26 * -0.04 ) $x$  \\
(1.0 *  .74) $x$  \\
%(.68 * -0.07 ) $x$  \\
%(.83 * -0.04 ) $x$  \\
(0.19 *  1.05) $y$  \\
( .64*  .56) $y$  \\
%( .75* .35 ) $y$  \\
%(.772 *  .62) $y$  \\
(0.22 *  .5) $g$  \\
(0.83 *  .47) $g$  \\
%(0.67 *  .35) $g$  \\
%(0.865 *  .45) $g$  \\
\endSetLabels
\centerline{\AffixLabels{\includegraphics{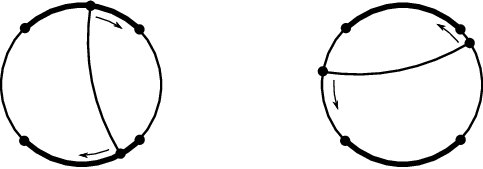}}}
\vskip 25 pt

\caption{The arrows indicate the direction in which the endpoints of $g$ can be moved  in order to increase $L _{E_g^t [f]} (Q)$ when $t>0$}
\label{fig:ElementaryEarthquakes}
\end{figure}
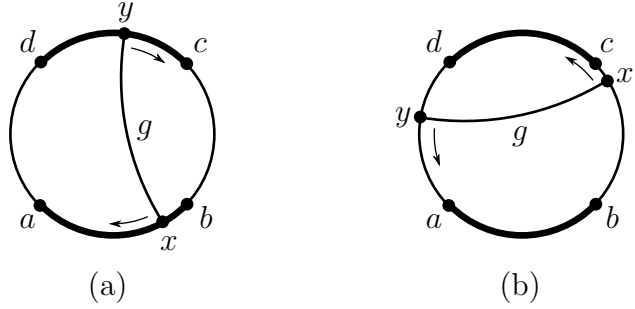

%The reader will easily check that, up to exchanging the labelling of the two intervals $[a,b]$, $[c,d] \subset \partial_\infty \widetilde X_0$ or of the endpoints $x$, $y \in \partial_\infty \widetilde X_0$, Lemma~\ref{lem:ElementaryEarthquake} lists all possible configurations for the box $Q$ and the geodesic $g$. 

\begin{proof}[Proof of Lemma~{\upshape\ref{lem:ElementaryEarthquake}(0)}] If $x$ and $y$ are in the same component of $\partial_\infty \widetilde X_0 - \{a,b,c,d\}$, let  $[f]$ be represented by  a quasisymmetric  homeomorphism $f\colon \partial_\infty \widetilde X_0 \to \partial_\infty \widetilde X_1$. Then, by definition of the elementary earthquake, $E_g^t[f]$ is represented by a quasisymmetric homeomorphism $E_g^tf$ that coincides with $f$ at the points $a$, $b$, $c$, $d$. If follows that  $E_g^tf(Q)=f(Q)$ in $G(\widetilde X_1)$, so that $L _{E^t_{g}[f]} (Q) = L _{[f]} (Q)$. \end{proof}

\begin{proof}[Proof of Lemma~{\upshape\ref{lem:ElementaryEarthquake}(a)}] 
In this second case (a), we can represent $[f]$ by a quasiconformal diffeomorphism $f\colon\widetilde X_0 \to \HH$ valued in the upper half-space 
$$
\HH = \{ z \in \C; \mathrm{Im}( z)>0 \}. 
$$
In addition, we can arrange that $f(y)=\infty$, and set $\alpha=f(a)$, $\beta=f(b)$, $\gamma= f(c)$, $\delta=f(d)$ and $\xi=f(x) $. Note that $\delta < \alpha < \xi < \beta < \gamma$ in $\R$. 

Then, by Lemma~\ref{lem:LiouvilleAndCrossratios},
$$
L _{[f]} (Q) = L_{\HH} \big( [\alpha, \beta] \times [ \gamma, \delta] \big) =  \log\frac{(\alpha - \gamma)(\beta -\delta)}{(\alpha -\delta)(\beta -\gamma)} . 
$$
Also, the hyperbolic isometry of $\HH$ that acts by translation of $t$ along the geodesic $\xi\infty$ is the map $z\mapsto \E^t z +\xi-\E^t \xi$. Therefore
\begin{align*}
\frac{\mathrm d}{\mathrm d\xi} L _{E_g^t [f]} (Q)& =\frac{\mathrm d}{\mathrm d\xi}  \log\frac{( \alpha- \E^t \gamma- \xi +\E^t \xi )(\E^t \beta +\xi-\E^t \xi -\delta)}{(\alpha -\delta)(  \E^t \beta -\E^t \gamma )}\\
&= \frac{-1+\E^t}{ \alpha- \E^t \gamma- \xi +\E^t \xi} + \frac{1-\E^t}{\E^t \beta +\xi-\E^t \xi -\delta}\\
&= \frac{1-\E^t}{ (\xi- \alpha ) + \E^t (\gamma- \xi )} + \frac{1-\E^t}{\E^t( \beta -\xi)+ (\xi- \delta)}\ <0
\end{align*}
where the inequality comes from the fact that $\delta < \alpha < \xi < \beta < \gamma$ and $t>0$. 

It follows that $L _{E_g^t [f]} (Q)$ is a decreasing function of $\xi = f(x)\in \R$, and therefore of the endpoint $x\in \partial_\infty \widetilde X_0$ of the geodesic $g$. 

By symmetry, $L _{E_g^t [f]} (Q)$ is also a decreasing function of the endpoint $y$.
\end{proof}

\begin{proof}[Proof of Lemma~{\upshape\ref{lem:ElementaryEarthquake}(b)}] Consider the orthogonal box $Q^\perp$ of $Q$. Case~(a) shows that $L _{E_g^t [f]} (Q^\perp)$ is a decreasing function of the endpoints $x$ and $y$. The relation between $L _{E_g^t [f]} (Q)$ and $L _{E_g^t [f]} (Q^\perp)$ provided by Lemma~\ref{lem:LiouvilleMassOrthoBox} then shows that  $L _{E_g^t [f]} (Q)$  is an increasing function of $x$ and $y$. 
\end{proof}

\begin{lem}
\label{lem:DiagonalElementaryEarthquake}
Let $E_{ac}^t \colon \T(\widetilde X_0) \to \T(\widetilde X_0)$ be the elementary earthquake associated to the diagonal geodesic $ac$  of the box $Q=[a,b]\times [c,d]$. Then, for every $[f] \in  \T(\widetilde X_0)$ and every $t> 0$,
$$
 t+ \log \big( \E^{L _{[f]} (Q)}-1\big)<
L _{E_{ac}^t [f]} (Q) < t + L _{[f]} (Q).
$$
\end{lem}

\begin{proof}
Represent the class $[f] \in  \T(\widetilde X_0)$ by a quasiconformal map $f\colon \widetilde X_0 \to \HH$ such that $f(a) =0$, $f(b)=\beta$, $f(c) = \infty$ and $f(d) = -1$. Then, as  in the proof of Lemma~\ref{lem:ElementaryEarthquake}(a) (with $\alpha=\xi=0$, $\gamma=\eta = \infty$ and $\delta=-1$),
$$
L _{E_{ac}^t [f]} (Q) = \log (\E^t \beta+1).
$$
In particular, the case $t=0$ gives that $\beta= \E^{L _{[f]} (Q)}-1$. 

Then, because $t> 0$, 
$$
L _{E_{ac}^t [f]} (Q) = t+ \log (\beta+\E^{-t}) < t+ \log (\beta+1) = t+ L _{[f]} (Q)
$$
while
\begin{equation*}
L _{E_{ac}^t [f]} (Q) = t+ \log (\beta+\E^{-t}) > t+ \log (\beta) = t+ \log \big( \E^{L _{[f]} (Q)}-1\big). \qedhere
\end{equation*}
\end{proof}

\subsection{Simple convergence on boxes} 
This section is devoted to proving Lemma~\ref{lem:EarthquakesAndBoxes2}, which is a key technical step in the proof of Theorem~\ref{thm:LimitEarthquakes}. As a warm-up, we begin with a simpler statement. 

It will be convenient to say that, for a geodesic current $\alpha \in \CC(\widetilde X_0)$, the box $Q=[a,b] \times [c,d]$ is \emph{$\alpha$--generic} if the subset of $G(\widetilde X_0)$ consisting of those geodesics with one endpoint in $\{a,b,c,d\}$ has $\alpha$--mass 0. Using the countable additivity of $\alpha$, there can be at most  countably many $x\in \partial_\infty \widetilde X_0$ such that the set of geodesics passing through $x$ has positive $\alpha$--mass.  As a consequence, every box can be arbitrarily approximated by an $\alpha$--generic box. 

\begin{lem}
\label{lem:EarthquakesAndBoxes1}
Let $\alpha \in \ML(X_0)$ be a bounded measured geodesic lamination. Then, for every $\alpha$--generic box $Q \subset G(\widetilde X_0)$,
$$
\lim_{t\to + \infty} \frac1{t} \, \LL \big(E^{t \alpha}[f] \big) (Q) = \alpha(Q).
$$
\end{lem}
\begin{proof} As usual, let the box $Q$ be described as $Q=[a,b]\times[c,d]$ with $a$, $b$, $c$, $d\in \partial_\infty \widetilde X_0$. 

We will split the proof into several steps. 

\medskip
\noindent\textsc{Step 1.}  \emph{ $\displaystyle \liminf_{t\to + \infty} \frac1{t} \, \LL \big(E^{t \alpha}[f]\big) (Q) \geq \alpha(Q)$.}
\medskip

We only need to consider the case where $ \alpha(Q)>0$.

Then, because of the hypothesis that $Q$ is $\alpha$--generic, there is a strictly smaller box $Q' = [a, b'] \times [c, d']$ such that $a<b'<b$, $c<d'<d$ and $\alpha(Q')$ is arbitrarily close to $\alpha(Q)$. Since $\alpha(Q')$ is close to $\alpha(Q)>0$ it is different from $0$, and $Q'$ meets the support of $\alpha$. Among the (disjoint) geodesics of the support of $\alpha$ that are contained in $Q'$, let $a''d''$ be the one that is closest to the interval $[d', a] \subset \partial_\infty \widetilde X_0$, and let $b''c''$ be the one closest to $[b', c]$, in such a way that $a \leq a'' \leq b'' \leq b'$ and $c \leq c'' \leq d'' \leq d'$.  See Figure~\ref{fig:EarthquakeLimit1}.

\begin{figure}[htbp]

\SetLabels
( .08 * .11 )  $a$\\
(  .92* .11 )  $b$\\
( .91 * .88 )  $c$\\
( .08 * .88 ) $d$ \\
 (.17 *  -.02) $a''$ \\
( .72 *  -.06 )  $b''$\\
( .84 *  -.0 ) $b'$ \\
( .8 * .96 )  $c''$\\
( .2 * .94 )  $d'$\\
( .33 * .99 )  $d''$\\
\endSetLabels
\centerline{\AffixLabels{\includegraphics{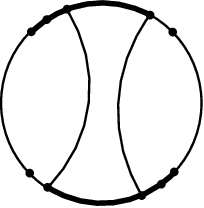}}}
\caption{Step 1 of the proof of Lemma~\ref{lem:EarthquakesAndBoxes1}}
\label{fig:EarthquakeLimit1}
\end{figure}
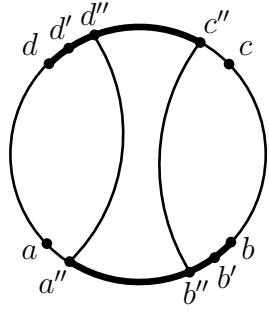

We now consider the box $Q'' = [a'', b] \times [c'', d]$. Our construction is specially designed that the geodesics $g$ of the support of $\alpha$ are of four distinct types with respect to  $Q'' = [a'', b] \times [c'', d]$: 
\begin{enumerate}
 \item $g$ has both endpoints in the closure of the same component of $ \partial_\infty \widetilde X_0 - \{ a'', b, c'', d\}$;
 \item $g$ has one endpoint in $[a'', b]$ and one endpoint in $[b,c'']$;
 \item $g$ has one endpoint in $[c'', d]$ and one endpoint in $[d, a]$;
 \item $g$ has one endpoint in $[a'',b]$ and another endpoint in $[c'', d]$.
 \end{enumerate}
Indeed, the presence of the geodesics $a'' d''$ and $b''c''$ in the support of $\alpha$ excludes all other cases. 

We can therefore decompose $\alpha$ as a sum of measured geodesic laminations
$$
\alpha = \alpha_o + \alpha_b + \alpha_d + \alpha_{Q''}
$$
where 
\begin{itemize}
 \item the support of $\alpha_b$ consists of geodesics of type (2), which encircle the point $b$;
 \item the support of $\alpha_d$ consists of geodesics of type (3), which encircle the point $d$;
 \item the support of $\alpha_{Q''}$ consists of geodesics of type (4), which are contained in the box $Q''$ (after a possible orientation reversal);
 \item the support of $\alpha_o$ consists of geodesics of type (1) (where $o$ stands for ``other''). 
\end{itemize}

 This decomposes the earthquake $E^{t\alpha}\colon \T(\widetilde X_0) \to  \T(\widetilde X_0)$ as a composition 
$$
E^{t\alpha}=  E^{t\alpha_o} \circ E^{t\alpha_d} \circ E^{t\alpha_b} \circ E^{t\alpha_{Q''}} .
$$
For notational convenience, set $[f_1] = E^{t\alpha_{Q''}} [f]$, $[f_2] = E^{t\alpha_b} [f_1]$, $[f_3] = E^{t\alpha_d} [f_2]$ and $[f_4] = E^{t\alpha_o} [f_3] =E^{t\alpha}[f] $.

We begin by estimating $\LL\big( [f_1] \big)(Q'') = \LL \big( E^{t\alpha_{Q''}} [f] \big)(Q'')$. 

If we approximate the measured lamination $\alpha_{Q''}$ by a Dirac measure supported on a finite set $\{ g_1, g_2, \dots, g_k, \bar g_1, \bar g_2, \dots, \bar g_k\}$ of disjoint geodesics  in the support of $\alpha_{Q''}$ and assigning mass $a_i>0$ to the atom $g_i$,  then by construction  $E^{t\alpha_Q''}$ is approximated by the product of elementary earthquakes  
$$E^{ta_1}_{g_1} \circ E^{ta_1}_{g_2} \circ \dots \circ E^{ta_n}_{g_n}.$$
By definition of $a''$ and $c''$, the geodesics of the support of $\alpha_{Q''}$ actually have one endpoint in $[a'', b''] \subset [a'', b']$ and one endpoint in $[c'', d''] \subset [c'', d']$. Lemma~\ref{lem:ElementaryEarthquake}(a)  shows that, for each such geodesic $g$, 
$$
 \LL \big( E_g^u [f'] \big)(Q) \geq  \LL \big( E_{b'd'}^u [f'] \big)(Q)
$$
for every $[f'] \in \T(\widetilde X_0)$ and every $u>0$. It follows that 
$$
 \LL \big(E^{ta_1}_{g_1}  E^{ta_1}_{g_2}  \dots  E^{ta_n}_{g_n}[f]\big)  (Q'') \geq 
  \LL \big(E^{t(a_1+a_2+\dots+a_n)}_{b'd'}[f]\big)  (Q'')
$$
and, passing to the limit as we improve the approximation of $\alpha_{Q''}$ by Dirac measures, that
$$
\LL\big( [f_1] \big)(Q'') = \LL \big(E^{t\alpha_{Q''}}[f]\big)  (Q'') \geq 
  \LL \big(E^{t\alpha(Q'')}_{b'd'}[f]\big)  (Q'')
$$
for every $t>0$. 

The box $Q'' = [a'', b] \times [c'', d]$ contains the box $Q'''= [b', b] \times [d', d]$. Lemma~\ref{lem:DiagonalElementaryEarthquake} then shows that
\begin{equation}
\label{eqn:EquakeLowerEstim1}
 \begin{split}
 \LL\big( [f_1] \big)(Q'') & \geq 
  \LL \big(E^{t\alpha(Q'')}_{b'd'}[f]\big)  (Q'')  \geq 
  \LL \big(E^{t\alpha(Q'')}_{b'd'}[f]\big)  (Q''') \\
  & \geq  t\alpha(Q'') + \log \big(\E^{L_{[f]}(Q''')}-1 \big).
\end{split}
\end{equation}

After this estimate for $ \LL\big( [f_1] \big)(Q'') $, we now consider $[f_2] = E^{t\alpha_b}[f_1]$. By construction, the Liouville current $ \LL\big ([f_2] \big) =  \LL\big (E^{t\alpha_b}[f_1] \big)$ is the pullback of $ \LL\big ([f_1] \big) $ by a homeomorphism of $G(\widetilde X_0)$ that sends $Q''=[a'', b] \times [c'', d]$ to a larger box  $Q_1''=[a'', b_1] \times [c'', d]$ with $b\leq b_1<c''$. Therefore,
\begin{equation}
\label{eqn:EquakeLowerEstim2}
  \LL\big ([f_2] \big)(Q'')=  \LL\big (E^{t\alpha_b}[f_1] \big)(Q'') =  \LL\big ([f_1] \big) (Q''_1) \geq  \LL\big ([f_1] \big)(Q'')
\end{equation}
since $Q_1''$ contains $Q''$.

Similarly, 
\begin{equation}
\label{eqn:EquakeLowerEstim3}
  \LL\big ([f_3] \big)(Q'') = \LL\big (E^{t\alpha_d}[f_2] \big)(Q'') \geq  \LL\big ([f_2] \big)(Q'').
\end{equation}

Finally, $ \LL\big ([f_4] \big) =  \LL\big (E^{t\alpha_o}[f_3] \big)$ is the pullback of $ \LL\big ([f_3] \big) $ by a homeomorphism of $G(\widetilde X_0)$ that sends $Q''$ to itself. Therefore
\begin{equation}
\label{eqn:EquakeLowerEstim4}
  \LL\big ([f_4] \big)(Q'') = \LL\big ([f_3] \big)(Q'').
\end{equation}

Combining Equations~(\ref{eqn:EquakeLowerEstim1}--\ref{eqn:EquakeLowerEstim4}), we conclude that 
\begin{equation}
\label{eqn:EquakeLowerEstim5}
 \LL\big (E^{t\alpha}[f] \big)(Q) \geq \LL\big (E^{t\alpha}[f] \big)(Q'') =  \LL\big ([f_4] \big)(Q'') \geq  t\alpha(Q'') + \log \big(\E^{L_{[f]}(Q''')}-1 \big).
\end{equation}

We now use the key property that $b'<b$ and $d'<d$, so that the box $Q'''= [b', b] \times [d', d]$ has nonempty interior and $L_{[f]}(Q''')>0$. It consequently follows from (\ref{eqn:EquakeLowerEstim5}) that
$$
\liminf_{t\to +\infty}  \frac1t  \LL\big (E^{t\alpha}[f] \big)(Q) \geq \alpha(Q''). 
$$

By definition of the box $Q''$, its mass $\alpha(Q'')$ for the measured lamination $\alpha$ is equal to $\alpha(Q')$. Also, because $Q$ is $\alpha$--generic, the box $Q' = [a, b'] \times [c, d']$ can be chosen so that $\alpha(Q')$ is arbitrarily close to $\alpha(Q)$. It follows that
$$
\liminf_{t\to +\infty}  \frac1t  \LL\big (E^{t\alpha}[f] \big)(Q) \geq \alpha(Q),
$$
which completes the proof of this  Step 1.

\medskip
\noindent\textsc{Step 2.} \emph{If $\alpha(Q) > 0$, then $\displaystyle \limsup_{t\to + \infty} \frac1{t} \, \LL \big(E^{t \alpha}[f]\big) (Q) \leq \alpha(Q)$.}
\medskip

The property that $\alpha(Q) > 0$ prevents any geodesic of the support of $\alpha$ from having one endpoint in $[b,c]$ and one endpoint in $[d,a]$. As in Step 1, we can therefore break down $\alpha$ as a sum of measured laminations
$$
\alpha = \alpha_Q +\alpha_a +\alpha_b +\alpha_c +\alpha_d + \alpha_o
$$
where
\begin{itemize}
 \item each geodesic of the support of $\alpha_Q$ has one endpoint in $[a,b]$ and one endpoint in $[c,d]$, and therefore belongs to $Q=[a,b] \times [c,d]$ after a possible orientation-reversal;
 \item each geodesic of the support of $\alpha_a$ has one endpoint in $[d,a]$ and one endpoint in $[a,b]$, and therefore encircles $a$;
 \item each geodesic of the support of $\alpha_b$ has one endpoint in $[a,b]$ and one endpoint in $[b,c]$, and therefore encircles $b$;
 \item each geodesic of the support of $\alpha_c$ has one endpoint in $[b,c]$ and one endpoint in $[c,d]$, and therefore encircles $c$;
 \item each geodesic of the support of $\alpha_d$ has one endpoint in $[c,d]$ and one endpoint in $[d,a]$, and therefore encircles $d$;
 \item each geodesic of the support of $\alpha_o$ has its two endpoints in the closure of the same component of $\partial_\infty \widetilde X_0 - \{ a,b,c,d\}$. 
 \end{itemize}

Then,
$$
E^{t\alpha}[f] = E^{t\alpha_o} \circ E^{t\alpha_a} \circ E^{t\alpha_c} \circ E^{t\alpha_Q} \circ E^{t\alpha_b} \circ E^{t\alpha_d} [f].
$$

In order to estimate $ \LL \big(E^{t\alpha}[f]\big)  (Q)$, set $[f_1] = E^{t\alpha_d} [f]$, $[f_2] = E^{t\alpha_b} [f_1]$, $[f_3] = E^{t\alpha_Q} [f_2]$, $[f_4] = E^{t\alpha_c} [f_3]$, $[f_5] = E^{t\alpha_a} [f_4]$ and $[f_6] = E^{t\alpha_o} [f_5] =E^{t\alpha}[f] $.

We will proceed backwards in our estimates, beginning with the simpler cases.

By construction of earthquakes, $\LL \big(E^{t\alpha}[f]\big)  = \LL \big([f_6]\big) = \LL \big(E^{t\alpha_o} [f_5]\big)$ is the pullback of $\LL \big([f_5]\big)$ by a quasi-symmetric homeomorphism of $\partial_\infty \widetilde X_0$ which sends the box $Q$ to itself. Therefore,
\begin{equation}
\label{eqn:EquakeUpperEst1}
 \LL \big(E^{t\alpha}[f]\big)  (Q)= \LL \big([f_6]\big)  (Q)=\LL \big([f_5]\big) (Q) .
\end{equation}

Again by construction of earthquakes, $ \LL \big([f_5]\big)  = \LL \big(E^{t\alpha_a} [f_4]\big) $  is the pullback of $\LL \big([f_4]\big)  $ by a homeomorphism of $\partial_\infty \widetilde X_0$ which fixes the points $b$, $c$, $d$, and which moves the point $a$ in the positive direction of $\partial_\infty \widetilde X_0$. As a consequence, this homeomorphism sends the box $Q=[a,b]\times[c,d]$ to a smaller box $Q_1=[a_1,b]\times[c,d] \subset Q$ with $a_1 \in [a,b]$, and 
\begin{equation}
\label{eqn:EquakeUpperEst2}
\LL \big([f_5]\big) (Q) = \LL \big([f_4]\big) (Q_1) \leq \LL \big([f_4]\big) (Q).
\end{equation}

The same argument applied to $ \LL \big([f_4]\big)  = \LL \big(E^{t\alpha_c} [f_3]\big) $ shows that
\begin{equation}
\label{eqn:EquakeUpperEst3}
\LL \big([f_4]\big) (Q) = \LL \big([f_3]\big) (Q_2) \leq \LL \big([f_3]\big) (Q)
\end{equation}
for some box $Q_2=[a,b]\times[c_2,d] \subset Q$. 

We now use Lemmas~\ref{lem:ElementaryEarthquake} and \ref{lem:DiagonalElementaryEarthquake} to estimate $\LL \big([f_3]\big) (Q) = \LL \big(E^{t\alpha_Q} [f_2]\big) (Q)$.

If we approximate the measured lamination $\alpha_Q$ by a Dirac measure based at a finite set $\{ g_1, g_2, \dots, g_k, \bar g_1, \bar g_2, \dots, \bar g_k \}$ of disjoint geodesics  in $Q$ and assigning mass $a_i>0$ to the atom $g_i$,  then by construction  $E^{t\alpha_Q}$ is approximated by the product of elementary earthquakes  
$$E^{ta_1}_{g_1} \circ E^{ta_1}_{g_2} \circ \dots \circ E^{ta_n}_{g_n}.$$
If $ac$ denotes the diagonal of the box $Q$, going from $a$ to $c\in \partial_\infty \widetilde X_0$, Lemma~\ref{lem:ElementaryEarthquake}(a) shows that 
$$\LL \big(E^{ta_i}_{g_i}[f']\big)  (Q) \leq   \LL \big(E^{ta_i}_{ac} [f']\big) (Q) $$
for every $[f']\in \T(\widetilde X_0)$. The combination of Lemmas~\ref{lem:ElementaryEarthquake} and \ref{lem:DiagonalElementaryEarthquake} then shows that
\begin{align*}
 \LL \big(E^{ta_1}_{g_1}  E^{ta_1}_{g_2}  \dots  E^{ta_n}_{g_n}[f_2]\big)  (Q) &\leq   \LL \big(E^{t(a_1+a_2+\dots+a_n)}_{ac}[f_2]\big)  (Q)\\
 & \leq  \LL \big([f_2]\big)  (Q) +  t(a_1+a_2+\dots+a_n).
\end{align*}

Passing to the limit as we use better and better approximations of $\alpha_Q$  by Dirac measures, we conclude that 
\begin{equation}
\label{eqn:EquakeUpperEst4}
\LL \big([f_3]\big) (Q) = \LL \big(E^{t\alpha_Q} [f_2]\big) (Q)  \leq  \LL \big([f_2]\big)  (Q) +  t \alpha(Q).
\end{equation}

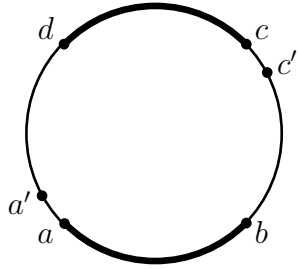
\begin{figure}[htbp]

\SetLabels
( .08 * .09 )  $a$\\
(  .91* .09 )  $b$\\
( .91 * .86 )  $c$\\
( .08 * .86 ) $d$ \\
( -.02 * .2 )  $a'$\\
( 1.01 * .73 )  $c'$\\
\endSetLabels
\centerline{\AffixLabels{\includegraphics{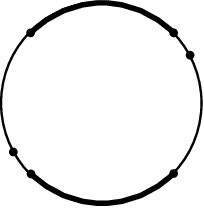}}}
\caption{Step 2 of the proof of Lemma~\ref{lem:EarthquakesAndBoxes1}}
\label{fig:EarthquakeLimit2}
\end{figure}

Estimating $\LL \big([f_2]\big) (Q) = \LL \big(E^{t\alpha_b} [f_1]\big) (Q)$ will require more care. In particular, we need to split the geodesics of the support of $\alpha_b$ into those that have one endpoint near $c$ and those that do not.  

Pick a point $c'$ in the open interval $\left] b,c\right[$ such that $\alpha \big( [a,b] \times \{c'\} \big)=0$, which can always be done by countable additivity of $\alpha$. We will later choose $c'$ close enough  to $c$ to ensure that $\alpha \big( [a,b] \times [c', c] \big)$ is small. See Figure~\ref{fig:EarthquakeLimit2}.

Let $\alpha_b'$ be the restriction of $\alpha$ to the box $[a,b] \times [c', c]$, and let $\alpha_b''$ be the restriction of $\alpha$ to $[a,b] \times [b, c']$. In particular, $\alpha_b = \alpha_b' + \alpha_b''$ by the property that $\alpha \big( [a,b] \times \{c'\}  \big)=0$.  

As in our analysis of $E^{t\alpha_a} [f_4]$ and $E^{t\alpha_c} [f_3]$, the Liouville current $\LL \big(E^{t\alpha_b''} [f_1]\big) $ is the pullback of $\LL \big([f_1]\big) $ under a homeomorphism of $\partial_\infty \widetilde X_0$ which fixes $a$, $c$, $d$ and moves $b$ to a point of the interval $[b,c']$. Therefore
$$
\LL \big(E^{t\alpha_b''} [f_1]\big)  (Q) \leq \LL \big([f_1]\big)  (Q_{c'}'') 
$$
where $Q_{c'}'' = [a,c'] \times [c,d]$. 

Then, as in our analysis of $E^{t\alpha_Q}[f_2]$, the combination of Lemmas~\ref{lem:ElementaryEarthquake} and \ref{lem:DiagonalElementaryEarthquake} gives that
\begin{equation}
\label{eqn:EquakeUpperEst5}
\begin{split}
\LL \big([f_2]\big) (Q) &= \LL \big(E^{t\alpha_b} [f_1]\big)  (Q) = \LL \big(E^{t\alpha_b'} E^{t\alpha_b''}[f_1]\big)  (Q) \\
&\leq \LL \big(E_{ac}^{t\alpha(Q_{c'}')} E^{t\alpha_b''}[f_1]\big)  (Q) \\
& \leq  \LL \big(E^{t\alpha_b''}[f_1]\big)  (Q) + t \alpha(Q_{c'}') \\
& \leq  \LL \big([f_1]\big)  (Q_{c'}'')  + t \alpha(Q_{c'}') 
\end{split}
\end{equation}
where $Q_{c'}'= [a,b] \times [c', c] $ and $Q_{c'}'' = [a,c'] \times [c,d]$. 

Similarly, to estimate $\LL \big([f_1]\big)  (Q_{c'}'') = \LL \big(E^{t\alpha_d} [f]\big) (Q_{c'}'')$, pick a point $a'$ in the open interval $\left] d,a\right[$ such that $\alpha \big( \{a'\} \times [c,d] \big)=0$, and split $\alpha_d$ as $\alpha_d = \alpha_d' + \alpha_d''$, where $\alpha_d'$ and $\alpha_d''$ are the respective restrictions of $\alpha_d$ to $[a', a] \times [c,d]$ and $[d,a'] \times [c,d]$. See Figure~\ref{fig:EarthquakeLimit2}.

Then, using the combination of Lemmas~\ref{lem:ElementaryEarthquake} and \ref{lem:DiagonalElementaryEarthquake} as in our analysis of $[f_2]= E^{t\alpha_b} [f_1]$, 
\begin{equation}
\label{eqn:EquakeUpperEst6}
\begin{split}
\LL \big([f_1]\big) (Q_{c'}'') &= \LL \big(E^{t\alpha_d} [f]\big)  (Q_{c'}'') = \LL \big(E^{t\alpha_d'} E^{t\alpha_d''}[f]\big)  (Q_{c'}'') \\
& \leq \LL \big(E^{t\alpha_d''}[f]\big)  (Q_{c'}'') + t \alpha(Q_{a'}') \\
& \leq  \LL \big([f]\big)  (Q_{a'c'}'')  + t \alpha(Q_{a'}') 
\end{split}
\end{equation}
where $Q_{a'}'= [a', a] \times [c,d] $ and $Q_{a'c'}'' = [a,c'] \times [c,a']$. 

Now, if we combine the estimates of (\ref{eqn:EquakeUpperEst1}--\ref{eqn:EquakeUpperEst6}), we get that 
\begin{equation}
\label{eqn:EquakeUpperEst7}
 \LL \big(E^{t\alpha}[f]\big)  (Q) \leq    \LL \big([f]\big)  (Q_{a'c'}'')  + t \alpha(Q_{a'}')  + t \alpha(Q_{c'}')  + t\alpha(Q) 
\end{equation}
for the boxes  $Q_{a'c'}'' = [a,c'] \times [c,a']$, $Q_{a'}'= [a', a] \times [c,d] $ and $Q_{c'}'= [a,b] \times [c', c] $.

Passing to the limit as $t$ tends to $\infty$, this gives
$$
\limsup_{t\to + \infty} \frac1{t} \, \LL \big(E^{t \alpha}[f]\big)  (Q) \leq   \alpha(Q_{a'}')  + \alpha(Q_{c'}') +\alpha(Q).  
$$

This property holds for any choice of points $a' \in \left] d,a\right[$ and $c' \in \left] b,c\right[$ (with $\alpha \big( [a,b] \times \{c'\} \big)=0$ and $\alpha \big( \{a'\} \times [c,d] \big)=0$). Letting $a'$ tend to $a$ and $c'$ tend to $c$, so that $\alpha(Q_{a'}')$ and $\alpha(Q_{c'}')$ respectively converge to $\alpha \big(  \{a\} \times [c,d]  \big)=0$ and $\alpha \big(  [a,b] \times \{c\}  \big)=0$ by our hypothesis that $Q$ is $\alpha$--generic, we conclude that
$$
\limsup_{t\to + \infty} \frac1{t} \, \LL \big(E^{t \alpha}[f]\big)  (Q) \leq   \alpha(Q).
$$
This concludes the proof of Step 2. 

In particular, the combination of Steps~1 and 2 shows that $\lim_{t\to + \infty} \frac1{t} \, \LL \big(E^{t \alpha}[f]\big)  (Q) =   \alpha(Q)$ when $\alpha(Q)>0$. 

We will rely on these first two  steps to settle the remaining cases. Recall that $Q^\perp$ denotes the orthogonal box of $Q$, as defined in \S \ref{subsect:GeodesicBoxes}. 

\medskip
\noindent\textsc{Step 3.} \emph{If $\alpha(Q)=0$ and  $\alpha( Q^\perp)> 0$, then $\displaystyle \lim_{t\to + \infty} \LL \big(E^{t \alpha}[f]\big) (Q) =0$. }
\medskip

We  rely on Lemma~\ref{lem:LiouvilleMassOrthoBox}, which shows that
\begin{equation}
\label{eqn:EquakeUpperEst8}
\E^{-\LL(E^{t \alpha}[f]) (Q) } +\E^{-\LL(E^{t \alpha}[f]) (Q^\perp) } =1.
\end{equation}
Because the box $Q$ is $\alpha$--generic, so is the orthogonal box $Q^\perp$. We can therefore apply  Step~1 to $Q^\perp$, which gives
$$
\liminf_{t\to + \infty} \frac1t \LL \big(E^{t \alpha}[f]\big) (Q^{\perp} ) \geq \alpha(Q^{\perp})>0
$$
and in particular implies that $ \LL(E^{t \alpha}[f])  (Q^{\perp} ) \to+ \infty$ as $t\to +\infty$. 

We  conclude that, as $t\to +\infty$, $\E^{-\LL(E^{t \alpha}[f]) (Q^\perp) } \to 0$  so that $\E^{-\LL(E^{t \alpha}[f]) (Q) } \to 1$ by (\ref{eqn:EquakeUpperEst8}), and therefore  $\LL(E^{t \alpha}[f]) (Q)  \to 0$.

\medskip
\noindent\textsc{Step 4.} \emph{If $\alpha(Q) =0$ and $\alpha( Q^\perp)=0$, then  $\displaystyle \lim_{t\to + \infty} \frac1t \LL \big(E^{t \alpha}[f]\big) (Q) =0$. }
\medskip

In the proof of Step 2, the only time we used the hypothesis that $\alpha(Q)>0$ was to guarantee that the support of $\alpha$ contained no geodesic of the interior of the orthogonal box $Q^\perp$. 

In the current setup of Step 4, the hypothesis that $\alpha(Q^\perp) =0$ implies that the support of $\alpha$ is disjoint from the interior of $Q^\perp$. We can therefore apply the arguments of Step 2  and conclude that
$$
 \limsup_{t\to + \infty} \frac1t \LL \big(E^{t \alpha}[f]\big) (Q) 
  \leq \alpha(Q)=0
$$
as required. 
\medskip

This concludes the proof of Lemma~\ref{lem:EarthquakesAndBoxes1}, by Steps 1 and 2 when $\alpha(Q)>0$, and by Steps 3 and 4 when $\alpha(Q)=0$. 
\end{proof}

We will need a more uniform version of Lemma~\ref{lem:EarthquakesAndBoxes1}.  The lemma below will allow us to enhance a weak* convergence to a uniform weak* convergence. Recall that the box $Q=[a,b] \times [c,d]$ is {$\alpha$--generic} if the subset of $G(\widetilde X_0)$ consisting of those geodesics with one endpoint in $\{a,b,c,d\}$ has $\alpha$--mass 0. 

\begin{lem}
\label{lem:EarthquakesAndBoxes2}
 Let $\{ \alpha_n \}_n$ be a sequence of bounded measured geodesic laminations converging,  as $n\to \infty$, to  a measure $\alpha$ on $G(\widetilde X_0)$ for the weak* topology. Then, for every sequence $\{t_n \}$ converging to $+\infty$ in $\R$ and for every $\alpha$--generic box $Q \subset G(\widetilde X_0)$ ,
$$
\lim_{n\to \infty} \frac1{t_n}\LL \big(E^{t_n \alpha_n}[f] \big)(Q) = \alpha(Q).
$$
\end{lem}

Note that the $\alpha_n$ are only required to converge to $\alpha$ for the weak* topology, not for the uniform weak* topology. As a consequence, $\alpha$ is clearly a measured geodesic lamination but is not necessarily bounded.

\begin{proof} This  follows from a careful inspection of the proof of Lemma~\ref{lem:EarthquakesAndBoxes1}. We repeat the steps of that proof. 

\medskip
\noindent\textsc{Step 1.}  \emph{ $\displaystyle \liminf_{n\to \infty} \frac1{t_n} \, \LL \big(E^{t_n \alpha_n}[f]\big) (Q) \geq \alpha(Q)$.}
\medskip

As in the proof of Lemma~\ref{lem:EarthquakesAndBoxes1}, assume $\alpha(Q)>0$ without loss of generality, and choose a smaller box $Q' = [a,b'] \times [c,d'] \subset Q$ with $a<b'<b$ and $c<d'<d$, and  with $\alpha(Q')>0$ close to $\alpha(Q)$. By countable additivity of $\alpha$ we can arrange that $Q'$ is $\alpha$--generic and in particular that $\alpha(\partial Q')=0$. 

For $n$ large enough, $\alpha_n(Q')>0$ by Lemma~\ref{lem:Weak*ConvImpliesConvMasses} and our hypothesis that $\alpha(\partial Q')=0$, and the support of $\alpha_n$ therefore meets $Q'$. Among the geodesics of the support of $\alpha_n$ that are contained in $Q'$, let $a_n''d_n''$ be the one that is closest to the interval $[d', a] \subset \partial_\infty \widetilde X_0$, and let $b_n''c_n''$ be the one closest to $[b', c]$, in such a way that $a \leq a_n'' \leq b_n'' \leq b'$ and $c \leq c_n'' \leq d_n'' \leq d'$.  Set $Q_n'' = [a_n'', b] \times [c_n'', d]$. 

The arguments used in Step~1 of  the proof of Lemma~\ref{lem:EarthquakesAndBoxes1} then show that, as in (\ref{eqn:EquakeLowerEstim4}), 
\begin{equation*}
 \LL\big (E^{t_n\alpha_n}[f] \big)(Q) \geq \LL\big (E^{t_n\alpha_n}[f] \big)(Q_n'') \geq  t_n\alpha_n(Q_n'') + \log \big(\E^{L_{[f]}(Q''')}-1 \big).
\end{equation*}
for the box $Q'''= [b', b] \times [d', d]$. 

By definition of the box $Q_n''$, its mass $\alpha_n(Q_n'')$ for the measured lamination $\alpha_n$ is equal to $\alpha_n(Q')$. Since we arranged that $\alpha(\partial Q')=0$, Lemma~\ref{lem:Weak*ConvImpliesConvMasses} then shows that $\alpha_n(Q_n'')= \alpha_n(Q')$ converges to $\alpha(Q')$ as $n$ tends to infinity. Therefore,
$$ \liminf_{n\to \infty} \frac1{t_n} \, \LL \big(E^{t_n \alpha_n}[f]\big) (Q) \geq \alpha(Q').$$

As $Q'$ can be chosen so that $\alpha(Q')$ is arbitrarily close to $\alpha(Q)$, we conclude that 
$$ \liminf_{n\to \infty} \frac1{t_n} \, \LL \big(E^{t_n \alpha_n}[f]\big) (Q) \geq \alpha(Q)$$
as required. 

\medskip
\noindent\textsc{Step 2.} \emph{If $\alpha(Q) > 0$, then $\displaystyle \limsup_{n\to  \infty} \frac1{t_n} \, \LL \big(E^{t_n \alpha_n}[f]\big) (Q) \leq \alpha(Q)$.}
\medskip

As in Step~2 of  the proof of Lemma~\ref{lem:EarthquakesAndBoxes1}, pick a point $c' \in \left]b,c \right[$ close to $c$, and a point $a'\in \left]d,a\right[$ close to $a$, such that $\alpha \big( [a,b] \times \{c'\}  \big)=0$ and $\alpha \big( \{a'\} \times [c,d] \big)=0$. Then,  the same argument as in that Step~2  shows that, for every $n$,
\begin{equation}
 \LL \big(E^{t_n\alpha_n}[f]\big)  (Q) \leq    \LL \big([f]\big)  (Q_{a'c'}'')  + t_n \alpha_n(Q_{a'}')  + t _n\alpha_n(Q_{c'}')  + t_n\alpha_n(Q) 
\end{equation}
for the boxes  $Q_{a'c'}'' = [a,c'] \times [c,a']$, $Q_{a'}'= [a', a] \times [c,d] $ and $Q_{c'}'= [a,b] \times [c', c] $. 

By choice of the points $a'$ and $c'$, $\alpha(\partial Q_{a'}')=\alpha(\partial Q'_{c'})=0$. We can therefore apply Lemma~\ref{lem:Weak*ConvImpliesConvMasses} when passing to the limit, and conclude that
$$\limsup_{n\to  \infty} \frac1{t_n} \, \LL \big(E^{t_n \alpha_n}[f]\big) (Q) \leq 
 \alpha(Q_{a'}')  + \alpha(Q_{c'}')  + \alpha(Q) .$$
 
 Choosing $a'$ and $c'$ so that $ \alpha(Q_{a'}') $ and $ \alpha(Q_{c'}') $ are arbitrarily small, we conclude that 
 $$\limsup_{n\to  \infty} \frac1{t_n} \, \LL \big(E^{t_n \alpha_n}[f]\big) (Q) \leq  \alpha(Q) .$$
 
 \medskip
\noindent\textsc{Step 3.} \emph{If $\alpha(Q)=0$ and  $\alpha(Q^\perp) > 0$, then $\displaystyle \lim_{n\to \infty} \LL \big(E^{t_n \alpha_n}[f]\big) (Q) =0$. }
\medskip

The argument  is identical to that used for Step~3 of the proof of Lemma~\ref{lem:EarthquakesAndBoxes1}.

\medskip
\noindent\textsc{Step 4.} \emph{If $\alpha(Q) =0$ and $\alpha(Q^\perp)=0$, then  $\displaystyle \lim_{n\to  \infty} \frac1{t_n} \LL \big(E^{t_n \alpha_n}[f]\big) (Q) =0$. }
\medskip

In the proof of Lemma~\ref{lem:EarthquakesAndBoxes1}, we used the fact that the support of $\alpha$ is disjoint from the interior of $Q^\perp$  to reduce this step to Step 2. However, although $\alpha(Q^\perp) =0$, it is here quite possible that $\alpha_n(Q^\perp)>0$ and  that the support of $\alpha_n$ meets the interior of $Q^\perp$. 

Let us decompose each $\alpha_n$ as a sum $\alpha_n = \alpha_n^{Q^\perp} + \alpha_n'$ of two measured geodesic laminations $\alpha_n^{Q^\perp} $ and $ \alpha_n'$ such that:
\begin{itemize}
 \item every geodesic of the support of  $\alpha_n^{Q^\perp} $ is contained in the orthogonal box $Q^\perp$, after a possible orientation-reversal;
 \item the support of $\alpha_n'$ is disjoint from the interior of $Q^\perp$. 
\end{itemize}

As in Step~2 of the proof of Lemma~\ref{lem:EarthquakesAndBoxes1}, pick a point $c' \in \left]b,c \right[$ close to $c$, and a point $a'\in \left]d,a\right[$ close to $a$, such that $\alpha \big( [a,b] \times \{c'\}  \big)=0$ and $\alpha \big( \{a'\} \times [c,d] \big)=0$. Because the support of $\alpha_n'$ is disjoint from the interior of $Q^\perp$, we can then apply to $\alpha_n'$ this Step~2 of  the proof of Lemma~\ref{lem:EarthquakesAndBoxes1} and show that, for every $n$,
\begin{equation}
\label{eqn:EquakeBoxUniform1}
 \LL \big(E^{t_n\alpha_n'}[f]\big)  (Q) \leq    \LL \big([f]\big)  (Q_{a'c'}'')  + t_n \alpha_n'(Q_{a'}')  + t_n \alpha_n'(Q_{c'}')  + t_n\alpha_n'(Q) 
\end{equation}
for the boxes  $Q_{a'c'}'' = [a,c'] \times [c,a']$, $Q_{a'}'= [a', a] \times [c,d] $ and $Q_{c'}'= [a,b] \times [c', c] $. Compare Equation~(\ref{eqn:EquakeUpperEst7}). 

Then, by  Lemma~\ref{lem:ElementaryEarthquake}(b) and  Lemma~\ref{lem:DiagonalElementaryEarthquake},
\begin{equation}
\label{eqn:EquakeBoxUniform2}
\begin{split}
 \LL \big(E^{t_n\alpha_n}[f]\big) (Q)&=
 \LL \big(E^{t_n\alpha_n^{Q^\perp}} E^{t_n\alpha_n'}  [f] \big)  (Q) 
  \\
 &\leq  \LL \big(E^{t_n\alpha_n^{Q^\perp}\kern -2pt(Q^\perp)}_{ac} E^{t_n\alpha_n'}  [f] \big)  (Q) 
  \\
 &\leq   \LL \big( E^{t_n\alpha_n'}  [f] \big)  (Q)
 + t_n\alpha_n^{Q^\perp}\kern -2pt(Q^\perp)
\end{split}
\end{equation}

Combining (\ref{eqn:EquakeBoxUniform1}) and (\ref{eqn:EquakeBoxUniform2}), we conclude that 
\begin{equation}
\begin{split}
 \LL \big(E^{t_n\alpha_n}[f]\big) (Q)
 &\leq \LL \big([f]\big)  (Q_{a'c'}'')   + t_n \alpha_n'(Q_{a'}')  + t_n \alpha_n'(Q_{c'}')  + t_n\alpha_n'(Q)   + t_n\alpha_n^{Q^\perp}\kern -2pt(Q^\perp)
  \\
&\leq \LL \big([f]\big)  (Q_{a'c'}'')   + t_n \alpha_n(Q_{a'}')  + t_n \alpha_n(Q_{c'}')  + t_n\alpha_n(Q)   + t_n\alpha_n(Q^\perp).
\end{split}
\end{equation}

Because the boxes $Q$, $Q^\perp$, $Q'_{a'}$, $Q'_{c'}$ are $\alpha$--generic, $\alpha_n(Q'_{a'})\to \alpha(Q'_{a'})$, $\alpha_n(Q'_{c'})\to \alpha(Q'_{c'})$, $\alpha_n(Q) \to \alpha(Q)=0$ and $\alpha_n(Q^\perp) \to \alpha(Q^\perp)=0$ as $n \to \infty$. It follows that
$$
\limsup_{n\to  \infty} \frac1{t_n} \LL \big(E^{t_n \alpha_n}[f]\big) (Q)  \leq  \alpha(Q'_{a'}) +  \alpha(Q'_{c'}).
$$

We can make $\alpha(Q'_{a'})$  arbitrarily close to $\alpha \big(  \{a\} \times [c,d]  \big)=0$ and  $\alpha(Q'_{c'})$  arbitrarily close to $\alpha \big( [a,b]\times  \{c\}  \big)=0$ by choosing $a'$ sufficiently close to $a$ and $c'$ sufficiently close to $c$. This proves that 
$$
\lim_{n\to  \infty} \frac1{t_n} \LL \big(E^{t_n \alpha_n}[f]\big) (Q)  =0.
$$

The combination of Steps 1, 2, 3 and 4 completes the proof of Lemma~\ref{lem:EarthquakesAndBoxes2}. 
\end{proof}

\subsection{Uniform weak* convergence of earthquake paths}

We are now ready to prove Theorem~\ref{thm:LimitEarthquakes}, which we restate here as:

\begin{thm}
 Let $\alpha \in \ML( X_0)$ be a bounded measured geodesic lamination and let $[f]\in \T(X_0)$ be a point of the Teichm\"uller space of $X_0$. Consider the left earthquake $E^{t\alpha} \colon \T(X_0) \to \T(X_0)$ for $t\in \R$, and the Liouville embedding $\LL \colon \T(X_0) \to \CC(X_0)$ from $\T(X_0)$ to  the space $\CC(X_0)$ of bounded geodesic currents. Then,
 $$
 \lim_{t\to \pm \infty} \frac 1{|t|}\LL \big( E^{t\alpha}[f] \big) = \alpha
 $$
 for the uniform weak* topology of $\CC(X_0)$. 
\end{thm}

\begin{proof} By symmetry between left and right earthquakes, we can restrict attention to the limit as $t \to + \infty$. 

	It is easier to use a proof by contradiction. Suppose the property false. Then, because the uniform weak* topology is metrizable (Lemma~\ref{lem:GeodCurrentSpaceMetrizable}), there exists a sequence of real numbers $t_n$ such that $t_n \to +\infty$ as $n\to \infty$ but such that $ \frac 1{t_n}\LL \big( E^{t_n\alpha}[f] \big)=\frac 1{t_n} L_{E^{t_n\alpha}[f]} $ does not converge to $\alpha$ for the uniform weak* topology. Passing to a subsequence if necessary, this means that there exists a lower bound $\epsilon>0$,  a test function $\xi \colon G(\widetilde X_0) \to \R$ with compact support and a sequence of biholomorphic maps $\phi_n \in \HHH(\widetilde X_0)$ such that 
\begin{equation}
\label{eqn:ConvergenceEarthquakes1}
\bigg| \frac1{t_n} \int_{G(\widetilde X_0)}  \xi \circ \phi_n \, dL_{E^{t_n\alpha}[f]} - \int_{G(\widetilde X_0)}  \xi \circ \phi_n \, d\alpha \bigg| > \epsilon
\end{equation}
 for every $n$. 
 
 Let $\alpha_n$ be the push forward of the measure $\alpha$ under the homeomorphism $G(\widetilde X_0) \to G(\widetilde X_0)$ induced by $\phi_n$. Then $\alpha_n$ is clearly a  measured geodesic lamination, and is bounded by definition of this property. Also, by definition of the push forward,
 $$
 \int_{G(\widetilde X_0)}  \xi \circ \phi_n \, d\alpha = \int_{G(\widetilde X_0)}  \xi  \, d\alpha_n.
 $$
 
Lift the quasiconformal diffeomorphism $f\colon X_0 \to X$ representing $[f]\in \T(X_0)$ to $\widetilde f\colon \widetilde X_0 \to \widetilde X$. Then, in the Teichm\"uller space $\T(\widetilde X_0)$ of the universal cover, diagram chasing in the construction of elementary earthquakes shows that $E^t_g [\widetilde f \circ \phi_n]= E^t_{\phi_n(g)}[\widetilde f]$ for every geodesic $g\in G(\widetilde X_0)$ and every $t\in \R$. It follows that $E^{t_n \alpha} [\widetilde f \circ \phi_n]= E^{t_n\alpha_n}[\widetilde f]$. As a consequence, the Liouville current $L_{E^{t_n\alpha_n}[ f]}=L_{E^{t_n\alpha_n}[\widetilde f]}$ is the push forward of $L_{E^{t_n\alpha}[ f]}=L_{E^{t_n\alpha}[\widetilde f]}$ under the homeomorphism $\phi_n \colon G(\widetilde X_0) \to G(\widetilde X_0)$ induced by $\phi_n\in \HHH(\widetilde X_0)$. In particular,
 $$
 \int_{G(\widetilde X_0)}  \xi \circ \phi_n \, dL_{E^{t_n\alpha}[f]}
 =
 \int_{G(\widetilde X_0)}  \xi  \, dL_{E^{t_n\alpha_n}[f]}
 $$
 and we can rewrite (\ref{eqn:ConvergenceEarthquakes1}) as
\begin{equation}
\label{eqn:ConvergenceEarthquakes2}
\bigg| \frac1{t_n} \int_{G(\widetilde X_0)}  \xi \, dL_{E^{t_n\alpha_n}f} - \int_{G(\widetilde X_0)}  \xi  \, d\alpha_n \bigg| > \epsilon.
\end{equation}

 For every continuous function $\xi' \colon G(\widetilde X_0) \to \R$ with compact support, the associated weak* seminorms
 $$
\vert \alpha_n \vert_{\xi'} = \Bigl\vert \int_{G(\widetilde X_0)} \xi' \, d\alpha_n \Bigr\vert =  \Bigl\vert \int_{G(\widetilde X_0)} \xi' \circ \phi_n \, d\alpha \Bigr\vert
$$
are uniformly bounded because the measured geodesic lamination  $\alpha \in \ML(X_0)$ is bounded. By weak* compactness (see for instance \cite[chap. III, \S 1, n${}^{\mathrm o}$9]{Bou}) we can therefore assume, after passing to a subsequence, that $\alpha_n$ converges to some measured geodesic lamination $\beta$ for the weak* topology (but not necessarily for the uniform weak* topology). 

Lemma~\ref{lem:EarthquakesAndBoxes2} then states that for every $\beta$--generic box $Q$
$$
\lim_{n \to \infty} \frac1{t_n} L_{E^{t_n \alpha_n}f} (Q) = \lim_{n \to \infty} \frac1{t_n} \LL\big( E^{t_n \alpha_n}[f] \big) (Q) =\beta(Q).
$$
But this will contradict (\ref{eqn:ConvergenceEarthquakes2}) if we approximate the test function $\xi$ by a $\beta$--generic step function, namely by a linear combination of the characteristic functions of a finite family of $\beta$--generic boxes. 

Therefore, our original assumption cannot hold, and $ \frac 1{|t|}\LL \big( E^{t\alpha}[f] \big) $ converges to $ \alpha$  for the uniform weak* topology as $t\to +\infty$. 
\end{proof}

\section{Naturality under quasiconformal diffeomorphisms}

We conclude with a remark that our constructions are natural with respect to quasiconformal diffeomorphisms. 

Let $f \colon X_1 \to X_2$ be a quasiconformal diffeomorphism between two conformally hyperbolic Riemann surfaces. If we lift $f$ to a quasiconformal diffeomorphism $\widetilde f \colon \widetilde X_1 \to \widetilde X_2$ between universal covers, the quasisymmetric extension  $\widetilde f \colon \partial_\infty \widetilde X_1 \to  \partial_\infty  \widetilde X_2$ induces a homeomorphism $\widetilde f \colon G(\widetilde X_1) \to G(\widetilde X_2)$ and therefore a bijection $F \colon \mathcal C(X_1) \to \mathcal C(X_2)$ between the corresponding spaces of geodesic currents. 

\begin{lem}
\label{lem:QuasiConfRespectsUniformWeak*}
 The above bijection restricts to a homeomorphism $F \colon \CC(X_1) \to \CC(X_2)$, when the spaces $\CC(X_1)$ and $\CC(X_2)$ of bounded geodesic currents are endowed with the uniform weak* topology. 
 \end{lem}
\begin{proof} The main issue to deal with is that the definition of bounded geodesic currents in $X_1$ and of the uniform weak* topology of $\CC(X_1)$ involves the space $\HHH(\widetilde X_1)$ of biholomorphic maps of the universal cover $\widetilde X_1$, whereas the corresponding notions in $X_2$ involve $\HHH(\widetilde X_2)$. Our proof will use an \textit{ad hoc} correspondence between  $\HHH(\widetilde X_1)$ and $\HHH(\widetilde X_2)$. 

Arbitrarily pick three distinct points $x_1$, $y_1$, $z_1 \in \partial_\infty \widetilde X_1$, counterclockwise in this order,  in the circle at infinity of $\widetilde X_1$  and three distinct points  $x_2$, $y_2$, $z_2  \in \partial_\infty \widetilde X_2$, also in counterclockwise order. Then, for every biholomorphic map $\phi \in \HHH(\widetilde X_2)$, there exists a unique $\rho(\phi) \in \HHH(\widetilde X_1)$ sending the three points $\widetilde f^{-1} \circ \phi^{-1}(x_2)$, $\widetilde f^{-1} \circ \phi^{-1}(y_2)$, $\widetilde f^{-1} \circ \phi^{-1}(z_2)$ to $x_1$, $y_1$, $z_1$, respectively. This provides a bijection $\rho \colon \HHH(X_2) \to \HHH(X_1)$ characterized by the property that for every $\phi \in \HHH(\widetilde X_2)$ the map $\phi \circ \widetilde f \circ \rho(\phi)^{-1}$ sends our base points $x_1$, $y_1$, $z_1 \in \partial_\infty \widetilde X_1$ to the base points  $x_2$, $y_2$, $z_2  \in \partial_\infty \widetilde X_2$, respectively. 

We temporarily postpone the proof that $F$ sends $\CC(X_1)$ to $\CC(X_2)$, as the argument will be a simpler version of our proof that the restriction $F \colon \CC(X_1) \to \CC(X_2)$ is continuous.

To prove that $F \colon \CC(X_1) \to \CC(X_2)$ is continuous, consider a sequence of bounded geodesic currents $\alpha_n \in \CC(X_1)$ converging to $\alpha_\infty$ as $n\to \infty$, for the uniform weak* topology. We want to show that $F(\alpha_n)$ converges to $F(\alpha_\infty)$ in $\CC(X_2)$, namely that
\begin{equation}
\label{eqn:QuasiConfRespectsUniformWeak*0}
\Vert F(\alpha_n) -F(\alpha_\infty) \Vert_{\xi}
 = \sup_{\phi \in \HHH(\widetilde X_2)}\ \Bigl\vert \int_{G(\widetilde X_2)} \xi\circ \phi    \, d F(\alpha_n) 
 - \int_{G(\widetilde X_2)} \xi\circ \phi    \, d F(\alpha_\infty)  \Bigr\vert \to 0 \text{ as } n\to \infty
\end{equation}
for every continuous function $\xi \colon G(\widetilde X_2) \to \R$ with compact support. It is easier to use a proof by contradiction. 

Suppose that (\ref{eqn:QuasiConfRespectsUniformWeak*0}) does not hold, in search for a contradiction. Then, passing to a subsequence if necessary, there exists $\delta>0$ and  a sequence of biholomorphic maps $\phi_n \in \HHH(\widetilde X_2)$ such that 
\begin{equation}
\label{eqn:QuasiConfRespectsUniformWeak*1}
 \Bigl\vert \int_{G(\widetilde X_2)} \xi\circ \phi_n    \, d F(\alpha_n) 
 - \int_{G(\widetilde X_2)} \xi\circ \phi_n    \, d F(\alpha_\infty)  \Bigr\vert > \delta
\end{equation}
for every $n$. 
Then, by definition of the measure $F(\alpha_n)$, 
\begin{equation}
\label{eqn:QuasiConfRespectsUniformWeak*2}
\begin{split}
 \int_{G(\widetilde X_2)} \xi\circ \phi_n    \, d F(\alpha_n) 
&=   \int_{G(\widetilde X_1)} \xi\circ \phi_n \circ \widetilde f \, d \alpha_n 
\\
&=   \int_{G(\widetilde X_1)} \xi\circ \big(\phi_n \circ \widetilde f \circ \rho(\phi_n)^{-1} \big) \circ \rho(\phi_n) \, d \alpha_n  
\end{split}
\end{equation}
for the bijection $\rho \colon \HHH(\widetilde X_1) \to \HHH(\widetilde X_2)$ defined above. Similarly, 
\begin{equation}
\label{eqn:QuasiConfRespectsUniformWeak*3}
 \int_{G(\widetilde X_2)} \xi\circ \phi_n    \, d F(\alpha_\infty) 
=   \int_{G(\widetilde X_1)} \xi\circ \big(\phi_n \circ \widetilde f \circ \rho(\phi_n)^{-1} \big) \circ \rho(\phi_n) \, d \alpha_\infty
\end{equation}

The functions $\widetilde f_n = \phi_n \circ \widetilde f \circ \rho(\phi_n)^{-1}  \colon \partial_\infty \widetilde X_1 \to  \partial_\infty \widetilde X_2$ are uniformly quasisymmetric since $M(\widetilde f_n) = M(\widetilde f)$, and by construction send  $x_1$, $y_1$, $z_1 \in \partial_\infty \widetilde X_1$ to  $x_2$, $y_2$, $z_2  \in \partial_\infty \widetilde X_2$, respectively. By a classical equicontinuity property (see \cite[\S II.5]{LV}), they consequently form a relatively compact family in the space of quasisymmetric homeomorphisms $ \partial_\infty \widetilde X_1 \to  \partial_\infty \widetilde X_2$, for the topology of uniform convergence. Passing to a subsequence if necessary, we can therefore assume that the $\widetilde f_n  \colon \partial_\infty \widetilde X_1 \to  \partial_\infty \widetilde X_2$ uniformly converge to some homeomorphism $\widetilde f_\infty$. Then, as $n\to \infty$,  the induced homeomorphisms $\widetilde f_n \colon G(\widetilde X_1) \to G(\widetilde X_2)$ converge to $\widetilde f_\infty \colon G(\widetilde X_1) \to G(\widetilde X_2)$ uniformly on compact subsets of $G(\widetilde X_1)$. 

By Equations (\ref{eqn:QuasiConfRespectsUniformWeak*2}) and (\ref{eqn:QuasiConfRespectsUniformWeak*3})
\begin{equation}
\label{eqn:QuasiConfRespectsUniformWeak*4}
\begin{split}  \Bigl\vert \int_{G(\widetilde X_2)} \xi\circ \phi_n    \, d F(\alpha_n) 
 &- \int_{G(\widetilde X_2)} \xi\circ \phi_n    \, d F(\alpha_\infty)  \Bigr\vert 
 \\
 &=  \Bigl\vert \int_{G(\widetilde X_1)} \xi\circ \widetilde f_n \circ \rho(\phi_n) \, d \alpha_n  
 - \int_{G(\widetilde X_1)} \xi\circ \widetilde f_n \circ \rho(\phi_n) \, d \alpha_\infty   \Bigr\vert 
 \\
 &\leq  \Bigl\vert \int_{G(\widetilde X_1)} \xi\circ \widetilde f_n \circ \rho(\phi_n) \, d \alpha_n  
 - \int_{G(\widetilde X_1)} \xi\circ \widetilde f_\infty \circ \rho(\phi_n) \, d \alpha_n   \Bigr\vert 
 \\
 &\qquad +  \Bigl\vert \int_{G(\widetilde X_1)} \xi\circ \widetilde f_\infty \circ \rho(\phi_n) \, d \alpha_n  
 - \int_{G(\widetilde X_1)} \xi\circ \widetilde f_\infty \circ \rho(\phi_n) \, d \alpha_\infty   \Bigr\vert 
  \\
 &\qquad +  \Bigl\vert \int_{G(\widetilde X_1)} \xi\circ \widetilde f_\infty \circ \rho(\phi_n) \, d \alpha_\infty 
 - \int_{G(\widetilde X_1)} \xi\circ \widetilde f_n \circ \rho(\phi_n) \, d \alpha_\infty   \Bigr\vert .
\end{split}
\end{equation}

Choose a nonnegative continuous function $\eta\colon G(\widetilde X_2) \to \R$ with compact support  that is constantly 1 on a neighborhood of the support of $\xi$. For an arbitrary $\epsilon>0$, the fact that $\widetilde f_n$ converges to $\widetilde f_\infty$ uniformly on compact subsets implies that
$$
| \xi \circ \widetilde f_n - \xi \circ \widetilde f_\infty | \leq \epsilon  \eta\circ \widetilde f_\infty
$$
for $n$ large enough, so that 
\begin{equation}
\label{eqn:QuasiConfRespectsUniformWeak*5}
\begin{split}
  \Bigl\vert \int_{G(\widetilde X_1)} \xi\circ \widetilde f_n \circ \rho(\phi_n) \, d \alpha_n  
 - \int_{G(\widetilde X_1)} \xi\circ \widetilde f_\infty \circ \rho(\phi_n) \, d \alpha_n   \Bigr\vert 
 &\leq \epsilon \int_{G(\widetilde X_1)} \eta\circ \widetilde f_\infty  \circ \rho(\phi_n) \, d \alpha_n 
 \\
 &\leq \epsilon \sup_{\phi \in \HHH(\widetilde X_1)}  \int_{G(\widetilde X_1)} \eta\circ \widetilde f_\infty  \circ \phi \, d \alpha_n 
 \\
& \leq \epsilon \Vert \alpha_n \Vert_{ \eta\circ \widetilde f_\infty} 
\end{split}
\end{equation}
for $n$ large enough, where $\Vert\  \Vert_{ \eta\circ \widetilde f_\infty}$ is the (uniform weak*) seminorm on $\CC(X_1)$ defined by the function $\eta\circ \widetilde f_\infty \colon G( \widetilde X_1) \to \R$. 
Similarly
\begin{equation}
\label{eqn:QuasiConfRespectsUniformWeak*6}
  \Bigl\vert \int_{G(\widetilde X_1)} \xi\circ \widetilde f_\infty \circ \rho(\phi_n) \, d \alpha_\infty 
 - \int_{G(\widetilde X_1)} \xi\circ \widetilde f_n \circ \rho(\phi_n) \, d \alpha_\infty   \Bigr\vert 
 \leq \epsilon \Vert \alpha_\infty \Vert_{ \eta\circ \widetilde f_\infty}
\end{equation}
for $n$ large enough. 
Finally,
\begin{equation}
\label{eqn:QuasiConfRespectsUniformWeak*7}
\Bigl\vert \int_{G(\widetilde X_1)} \xi\circ \widetilde f_\infty \circ \rho(\phi_n) \, d \alpha_n  
 - \int_{G(\widetilde X_1)} \xi\circ \widetilde f_\infty \circ \rho(\phi_n) \, d \alpha_\infty   \Bigr\vert 
 \leq  \Vert \alpha_n-  \alpha_\infty \Vert_{ \xi\circ \widetilde f_\infty}. 
\end{equation}

Combining the inequalities of (\ref{eqn:QuasiConfRespectsUniformWeak*4}--\ref{eqn:QuasiConfRespectsUniformWeak*7}) we conclude that, for every $\epsilon>0$, 
\begin{equation}
\label{eqn:QuasiConfRespectsUniformWeak*8}
  \Bigl\vert \int_{G(\widetilde X_2)} \xi\circ \phi_n    \, d F(\alpha_n) 
 - \int_{G(\widetilde X_2)} \xi\circ \phi_n    \, d F(\alpha_\infty)  \Bigr\vert 
 \leq \epsilon \Vert \alpha_n \Vert_{ \eta\circ \widetilde f_\infty} 
 +\epsilon \Vert \alpha_\infty \Vert_{ \eta\circ \widetilde f_\infty}
 + \Vert \alpha_n-  \alpha_\infty \Vert_{ \xi\circ \widetilde f_\infty}.
\end{equation}
for $n$ large enough. 

However, $\Vert \alpha_n \Vert_{ \eta\circ \widetilde f_\infty} \to \Vert \alpha_\infty \Vert_{ \eta\circ \widetilde f_\infty} $ and $ \Vert \alpha_n-  \alpha_\infty \Vert_{ \xi\circ \widetilde f_\infty} \to 0$ as $n\to \infty$ since $\alpha_n \to \alpha_\infty$ in $\CC(X_1)$, so that (\ref{eqn:QuasiConfRespectsUniformWeak*8}) contradicts (\ref{eqn:QuasiConfRespectsUniformWeak*1}) for $\epsilon$ small enough. 

This contradiction proves (\ref{eqn:QuasiConfRespectsUniformWeak*0}), and shows that the function $F \colon \CC(X_1) \to \CC(X_2)$ is continuous. 

A symmetric argument shows that the inverse $F^{-1} \colon \CC(X_2) \to \CC(X_1)$  is continuous, so that $F \colon \CC(X_1) \to \CC(X_2)$ is a homeomorphism.

We had postponed the proof that our original function $F \colon \mathcal C(X_1) \to \mathcal C(X_2)$ sends bounded geodesic current to bounded geodesic current. This is a simpler version of the above continuity proof. For a bounded geodesic current $\alpha \in \CC(X_2)$, suppose in search of a contradiction that the geodesic current $F(\alpha) \in \mathcal C(X_2)$ is not bounded. As in (\ref{eqn:QuasiConfRespectsUniformWeak*0}) and (\ref{eqn:QuasiConfRespectsUniformWeak*1}), this means that there exists a continuous function $\xi \colon G(\widetilde X_2) \to \R$ with compact support and  a sequence of biholomorphic maps $\phi_n \in \HHH(\widetilde X_2)$ such that 
\begin{equation}
\label{eqn:QuasiConfRespectsUniformWeak*9}
 \Bigl\vert \int_{G(\widetilde X_2)} \xi\circ \phi_n    \, d F(\alpha) 
 \Bigr\vert \to \infty \text{ as } n \to \infty. 
\end{equation}

Passing to a subsequence if necessary, we can again arrange that the functions $\widetilde f_n = \phi_n \circ \widetilde f \circ \rho(\phi_n)^{-1} \colon G(\widetilde X_1) \to G(\widetilde X_2)$ converge to some homeomorphism $\widetilde f_\infty$, uniformly on compact subsets of $G(\widetilde X_1)$. Then, given $\epsilon>0$ and a continuous function $\eta\colon G(\widetilde X_2) \to \R$ with compact support  that is constantly 1 on a neighborhood of the support of $\xi$,
\begin{equation}
\begin{split}
  \Bigl\vert \int_{G(\widetilde X_2)} \xi\circ \phi_n    \, d F(\alpha) 
 \Bigr\vert
 &=  \Bigl\vert \int_{G(\widetilde X_1)} \xi\circ \widetilde f_n \circ \rho(\phi_n)    \, d\alpha
 \Bigr\vert
 \\
 &\leq  \Bigl\vert \int_{G(\widetilde X_1)} \xi\circ \widetilde f_n \circ \rho(\phi_n)    \, d\alpha
- \int_{G(\widetilde X_1)} \xi\circ \widetilde f_\infty \circ \rho(\phi_n)    \, d\alpha
 \Bigr\vert
 \\
&\qquad\qquad \qquad\qquad +
  \Bigl\vert \int_{G(\widetilde X_1)} \xi\circ \widetilde f_\infty \circ \rho(\phi_n)    \, d\alpha
 \Bigr\vert
 \\
 & \leq \epsilon \Vert \alpha \Vert_{ \eta\circ \widetilde f_\infty}
 + \Vert \alpha \Vert_{ \xi \circ \widetilde f_\infty}
\end{split}
\end{equation}
for $n$ large enough, as in (\ref{eqn:QuasiConfRespectsUniformWeak*4}--\ref{eqn:QuasiConfRespectsUniformWeak*8}). But this clearly contradicts (\ref{eqn:QuasiConfRespectsUniformWeak*9}), and therefore concludes our proof that the geodesic current $F(\alpha)$ is bounded. 

As a consequence, the bijection $F\colon \mathcal C(X_1) \to \mathcal C(X_2)$ restricts to a map $F\colon \CC(X_1) \to \CC(X_2)$, which we already proved is a homeomorphism for the uniform weak* topologies. 
\end{proof}

The quasiconformal diffeomorphism $f \colon X_1 \to X_2$ also induces a map $F_{\mathrm T} \colon \T(X_1) \to \T(X_2)$ between Teichm\"uller spaces, by the property that $F_{\mathrm T} \big( [g] \big) = [g \circ f^{-1}]\in \T(X_2)$ for every $[g] \in \T(X_1)$ represented by a quasiconformal diffeomorphism $g \colon X_1 \to X$. It is immediate from definitions that $F_{\mathrm T} $ is an isometry for the Teichm\"uller metrics of $\T(X_1)$ and $\T(X_2)$.

It is also immediate from definitions that this construction is well-behaved with respect to the Liouville embeddings $\LL_1 \colon \T(X_1) \to \CC(X_1)$ and  $\LL_2 \colon \T(X_2) \to \CC(X_2)$. More precisely, the diagram
$$
\xymatrix{
\CC(X_1)\ar[r]^F  & \CC(X_2)\\
\T(X_1)\ar[r]^{F_{\mathrm T}} \ar[u]^{\LL_1} & \T(X_2)  \ar[u]_{\LL_2}
}
$$
is commutative. 

The following property is then an automatic consequence of the continuity of $F \colon \CC(X_1) \to \CC(X_2)$. 

\begin{prop}
\label{prop:QuasiconformalExtendsBoundary}
Let $f \colon X_1 \to X_2$ be a quasiconformal diffeomorphism between two conformally hyperbolic Riemann surfaces. Then the isometry $F_{\mathrm T} \colon \T(X_1) \to \T(X_2)$ induced by $f$ continuously extends to the Thurston bordifications $\T(X_1) \cup \PML(X_1)$ and $\T(X_2) \cup \PML(X_2)$ of {\upshape\S \ref{subsect:ThurstonBdry}}. \qed
\end{prop}

In particular, we can consider the case where $X_1=X_2$. The \emph{quasiconformal mapping class group} of a conformally hyperbolic Riemann surface $X_0$ is the group
$$
\MCG(X_0) = \{ \text{quasiconformal diffeomorphisms } f \colon X_0 \to X_0 \}/\sim,
$$
where the equivalence relation $\sim$ identifies $f_1$, $f_2 \colon X_0 \to X_0$ when they are isotopic by an isotopy that moves points by a uniformly bounded amount, for the Poincar\'e metric. We refer to the results of \cite{EM} for several equivalent ways of expressing this relation. 

A quasiconformal diffeomorphism $g\colon X_0 \to X$ is a quasi-isometry for the Poincar\'e metrics of $X_0$ and $X$. It follows that, if the quasiconformal diffeomorphisms $f_1$, $f_2 \colon X_0 \to X_0$ are isotopic by an isotopy that moves points by a uniformly bounded amount, so are $g \circ f_1^{-1}$ and $g \circ f_2^{-1} \colon X_0 \to X_0$. As a consequence, if $f_1$, $f_2 \colon X_0 \to X_0$ represent the same element of $\MCG(X_0)$, the maps $F_1$, $F_2 \colon \T(X_0) \to \T(X_0)$ respectively induced by $f_1$ and $f_2$ coincide. This defines an isometric action of the quasiconformal mapping class group $\MCG(X_0)$ on the Teichm\"uller space $\T(X_0)$. 

Proposition~\ref{prop:QuasiconformalExtendsBoundary} immediately implies the following result. 

\begin{cor}
 The action of the quasiconformal mapping class group $\MCG(X_0)$ on the Teichm\"uller space $\T(X_0)$ continuously extends to the Thurston bordification $\T(X_0) \cup \PML(X_0)$ of {\upshape\S \ref{subsect:ThurstonBdry}}.  \qed
\end{cor}

\bibliographystyle{alpha}
 \bibliography{ThurstonBoundary}

\end{document}